\documentclass{article}
\usepackage{graphicx} % Required for inserting images

\usepackage{setspace}

\usepackage{amsmath}
\usepackage{amssymb}
\usepackage{amsthm}
\usepackage{mathtools}
\usepackage{empheq}
\usepackage{enumitem}
\usepackage{titlesec}
\usepackage{graphicx}
\usepackage{caption}
\usepackage{subcaption}
\usepackage{diagbox}
\usepackage{hyperref}

\usepackage[T1]{fontenc}
\usepackage[utf8]{inputenc}
\usepackage{comment}
\usepackage{indentfirst}
\usepackage[top=1.25in,left=1.25in,right=1.25in]{geometry}

\def\calP{\mathcal{P}}
\def\calR{\mathcal{R}}
\def\calS{\mathcal{S}}
\def\calT{\mathcal{T}}

\numberwithin{equation}{section}

\title{\Large{\uppercase{\bf
Cancellation and regularity for planar, $\mathbf{3}$-connected Kronecker products
}}}
\author{\Large{Ruben De March and Riccardo W. Maffucci}}
\date{}
\newcommand{\Addresses}{{
		\footnotesize
		R.~De~March \textsc{}\par\nopagebreak\vspace{-0.35cm}
		\textit{E-mail address}, R.~De~March: \href{mailto:rubendemarch@gmail.com}{\texttt{rubendemarch@gmail.com}}
  
		R.W.~Maffucci, \textsc{Dipartimento di Matematica, Universit\`a di Torino\\\indent Via Carlo Alberto 10, Turin 10123, Italy}\par\nopagebreak\vspace{-0.35cm}
		\textit{E-mail address}, R.W.~Maffucci: \href{mailto:riccardowm@hotmail.com}{\texttt{riccardowm@hotmail.com}} \ (corresponding author)

  }}

\def\q{\hspace{0.04cm}\square\hspace{0.04cm}}
\def\w{\wedge}

\newtheorem{thm}{Theorem}[section]
\newtheorem{lemma}[thm]{Lemma}
\newtheorem{prop}[thm]{Proposition}

\newtheorem{cor}[thm]{Corollary}
\newtheorem{defin}[thm]{Definition}
\newtheorem{ex}[thm]{Example}

\begin{document}
\titleformat{\section}
  {\Large\scshape}{\thesection}{1em}{}
\titleformat{\subsection}
  {\large\scshape}{\thesubsection}{1em}{} 
\maketitle

\vspace{-0.75cm}
{\centering\small{\textit{Dedicated to Prof. U. Cerruti on the occasion of his 75\textsuperscript{th} birthday}}\par}
\vspace{0.75cm}
\Addresses

\begin{abstract}
We investigate several properties of Kronecker (direct, tensor) products of graphs that are planar and $3$-connected (polyhedral, $3$-polytopal). This class of graphs was recently characterised and constructed by the second author \cite{mafkpr}.

Our main result is that cancellation holds for the Kronecker product of graphs when the product is planar and $3$-connected (it is known that Kronecker cancellation may fail in general). Equivalently, polyhedral graphs are Kronecker products in at most one way. This is a special case of the deep and interesting question, open in general, of Kronecker product cancellation for simple graphs: when does $A\w C\simeq B\w C$ imply $A\simeq B$?

We complete our investigation on simultaneous products by characterising and constructing the planar graphs that are Cartesian products in two distinct ways, and the planar, $3$-connected graphs that are both Kronecker and Cartesian products.

The other type of results we obtain are in extremal graph theory. We classify the polyhedral Kronecker products that are either face-regular or vertex-regular graphs. The face-regular ones are certain quadrangulations of the sphere, while the vertex-regular ones are certain cubic graphs (duals of maximal planar graphs). We also characterise, and iteratively construct, the face-regular subclass of graphs minimising the number of vertices of degree $3$.

\end{abstract}
\vspace{0.5cm}
% 05C75, 05C76, 05C62, 05C10, 05C40, 52B05, 52B10, 05C85, 05C83.

\section{Introduction}
\subsection{Connectivity, planarity, regularity, and graph products}
In this paper, we use the term `graph' for finite graphs with no multiple edges or loops, and `multigraph' when there may be multiple edges and/or loops. 
%In this paper, we deal mostly with graphs that have no multiple edges or loops. We refer to these as `graphs'. To refer to a member of the larger class of graphs that may contain multiple edges, we will use the term . 
A graph is planar if it may be drawn in the plane without any edges intersecting, except at vertices. Given an integer $k\geq 1$, a graph is $k$-connected if it has at least $k+1$ vertices, and however one removes $k-1$ or fewer of them, the resulting graph stays connected. The connectivity of a graph $G$ is $k$ if $G$ is $k$-connected but not $k+1$-connected. We call $G$ semi-hyper-$k$-connected \cite{zhme06,zhan09} if its vertex connectivity is $k$, and moreover every $k$-cut (i.e. separating set of $k$ vertices) removal results in exactly $2$ connected components. One of our points of focus are the planar, semi-hyper-$2$-connected graphs. Note that a planar graph is semi-hyper-$2$-connected if and only if however a planar embedding of $G$ is chosen, there is a region containing every $2$-cut, as in Figure \ref{fig:sh2c}.%(from \cite[Figure 9]{mafkpr}).

\begin{figure}[ht]
\centering
\includegraphics[width=2.75cm]{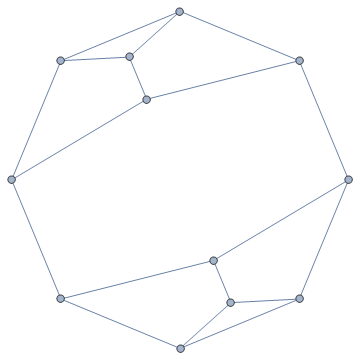}
\caption{A planar, semi-hyper-$2$-connected graph. The external region contains every $2$-cut.}
\label{fig:sh2c}
\end{figure}

The class of planar, $3$-connected graphs is exactly the class of $1$-skeletons (wireframes) of polyhedral solids, by the Rademacher-Steinitz Theorem \cite{radste}. For this reason, these graphs are called $3$-polytopes or polyhedra. They have several interesting properties: they may be immersed in the sphere in one and only one way, as remarked by Whitney \cite{whitco}. They are exactly the class of planar graphs where each region (face) is bounded by a polygon, and moreover two faces intersect either in the empty set, or at a vertex, or at an edge \cite[\S 4, p.~156]{hajetk}. For example, the class of maximal planar graphs is the subclass of polyhedra where each face is a triangle -- the triangulations of the sphere.

For a plane graph $G$ one may define the dual $G^*$, which depends on the planar embedding, and may contain multiple edges and/or loops. However, the dual graph of a $3$-polytope is not only a simple graph, but it is a uniquely determined $3$-polytope. One can compare with the concept of dual polyhedral solid in geometry.

We call a $3$-polytopal graph $s$-\textit{vertex-regular} if each of its vertices has the same degree $s$, i.e.~the graph is regular in the usual sense. We say that a $3$-polytope $\calP$ is $t$-\textit{face-regular} if each face has the same length $t$, meaning that each face is bounded by a $t$-gon. This is equivalent to $\calP^*$ being $t$-vertex-regular. For example, the dual of a maximal planar graph is a cubic (i.e.~$3$-regular) polyhedron. Since $3$-polytopes are planar, it follows from Euler's formula that $s,t\leq 5$. Since they are also $3$-connected, the only admissible values are in fact $3\leq s\leq 5$ and $3\leq t\leq 5$. Indeed, this is the first step in the graph-theoretical proof that there are exactly five polyhedra that are both vertex- and face-regular \cite[Theorem 1.38]{hahimo}, namely the Platonic solids, a fact known since antiquity.

Given two graphs $A,B$, one may define a product graph in several ways. Two of the most commonly studied \cite{haimkl} are the Kronecker $\w$ and Cartesian $\q$ products. One defines for the vertices
\[V(A\w B)=V(A\q B)=V(A)\times V(B),\]
where $\times$ denotes the Cartesian product of sets, and for the edges
\begin{equation*}
E(A\w B)=\{(a_1,b_1)(a_{2},b_2) : a_1a_2\in E(A) \text{ and } b_1b_2\in E(B)\}
\end{equation*}
and
\begin{equation*}
E(A\q B)=\{(a_1,b_1)(a_2,b_2) : (a_1=a_2 \text{ and } b_1b_2\in E(B)) \text{ or } (a_1a_2\in E(A) \text{ and } b_1=b_2)\}.
\end{equation*}

Fixing $B=K_2$, the graph $A\q K_2$ is called the prism over $A$ \cite{bieell} (in this sense, the usual prisms are the prisms over the polygons). The graph $A\w K_2$ is called the Kronecker cover or the double cover of $A$ \cite{wall76,farwal}.

\paragraph{Notation.} 
% Facial cycles will be written in the form
%\[[a_1,a_2,\dots,a_\ell].\]
We will write $C_\ell$ for the $\ell$-gon ($\ell$-cycle)
\[[u_1,u_2,\dots,u_{\ell}], \quad \ell\geq 3.\]
If these indices are considered modulo $\ell$, then we fix $u_0:=u_\ell$, so that for instance $u_{((\ell-1)+1 \mod\ell)}$ refers to $u_\ell$.
\\
We will denote by $P_m$ the elementary path on $m\geq 1$ vertices. For $n\geq 3$, $m\geq 1$ the Cartesian products
\[C_n\q P_m\]
are the so-called $m$-stacked $n$-gonal prisms. They will play an important role in our investigation. For $m\geq 2$, they are $3$-polytopes, a fact that we will henceforth use without mention.
\\
We say that a graph is outerplanar if it is planar and there exists a region containing every vertex. In what follows, $H$ will always denote an outerplanar, Hamiltonian graph, i.e.~$C_\ell$ possibly with some diagonals.
\\
We also define the ladder graph
\[F_{2\ell}:=C_{2\ell}+u_2u_{2\ell-1}+u_3u_{2\ell-2}+\dots+u_{\ell-1}u_{\ell+2}, \qquad \ell\geq 1.\]
Note that
\[F_{2\ell}\simeq P_\ell\q K_2,\]
where $\simeq$ is graph isomorphism. Also note that $F_{2\ell}$ is outerplanar, and it is Hamiltonian for $\ell\geq 2$.
\\
The letters $x,y$ will be reserved for the vertex labels of $K_2$ when considering
the Kronecker product $J\w K_2$. If $a\in V(J)$, we will denote the corresponding vertices of the product as
\[(a,x) \quad \text{and} \quad (a,y),\]
or simply as $ax$ and $ay$ if there is no possibility of confusion. If $ab\in E(J)$, we will denote the corresponding edges of the product as
\[(a,x)(b,y) \quad \text{and} \quad (a,y)(b,x).\]
The letter $\calP$ will always denote a $3$-polytopal graph.

\subsection{Summary of results from \cite{mafkpr}}
\label{sec:prior}
The motivation of this paper comes from the breakthrough results obtained in \cite{mafkpr} on the characterisation and construction of $3$-polytopal Kronecker products, and several further questions arising from these results. The first fact to know is that, if $\calP=A\wedge B$ is a $3$-polytope, then one of $A, B$ is $K_2$ \cite[Proposition 1.1]{mafkpr}. We will summarise the main results of \cite{mafkpr} in the following two statements.
%$J$ such that $\calP=J\wedge K_2$ is a $3$-polytope. This characterisation may be summarised in the following two statements.

\begin{thm}[{\cite[Theorems 1.3 and 1.4]{mafkpr}}]
\label{thm:0123}
Let $J$ be a planar graph. Then $\calP=J\wedge K_2$ is a polyhedron if and only if one of the following conditions holds:
\begin{enumerate}[label=(\theenumi)]
\setcounter{enumi}{-1}
\item \label{eq:c0} $J$ is of connectivity $2$. All odd regions of $J$, save exactly two, contain a $2$-cut. If $\{a,b\}$ is a $2$-cut in $J$, then $J-a-b$ has exactly two connected components, each containing an odd region;
\item \label{eq:c1} $J$ is a polyhedron with exactly $2$ odd faces, and these odd faces are disjoint;
\item \label{eq:c2} $J$ is a polyhedron with exactly $4$ odd faces, the intersection of any two of these odd faces is non-empty, and the intersection of any three of these odd faces is empty;
\item \label{eq:c3} $J$ is a polyhedron with at least $4$ odd faces, all of these odd faces except one intersect at a vertex, and the remaining odd face has non-empty intersection with all other odd faces.
\end{enumerate}
\end{thm}

To formulate the second statement, we need the following definition.
\begin{defin}
\label{cond1}
We say that the graph $J$ is the {\bf factor of a Kronecker $\mathbf{3}$-polytope} if it satisfies all of the following conditions. We have
\[J=J'+a_1b_1+a_2b_2+\dots+a_mb_m, \qquad a_1,b_1,a_m,b_m \text{ distinct,}\]
where $J'$ is a planar, bipartite graph, either $3$-connected, or semi-hyper-$2$-connected; there exists a planar immersion of $J'$ such that $a_1, b_1,a_2,b_2,\dots,a_m,b_m\in V(J')$ all lie on one region $\calR$, either in the order
\begin{equation}
	\label{eq:ord1}
a_1,a_2,\dots,a_m,b_m,b_{m-1},\dots,b_1,
\end{equation}
or in the order
\begin{equation}
	\label{eq:ord2}
a_1,a_2,\dots,a_m,b_1,b_2,\dots,b_m;
\end{equation}
for each $1\leq i\leq m$, the graph $J'+a_ib_i$ is not bipartite; any $2$-cut of $J'$ lies on $\calR$, but not between $b_1$ and $a_1$, or $a_m$ and $b_m$ in case of \eqref{eq:ord1}, or between $b_m$ and $a_1$, or $a_m$ and $b_1$ in case of \eqref{eq:ord2} (extrema included).
\end{defin}

Henceforth, whenever a graph $J$ is the factor of a Kronecker $3$-polytope, the notations
\[J',\calR,m,a_1, b_1,a_2,b_2,\dots,a_m,b_m\]
have the meaning of Definition \ref{cond1}. The terminology `factor of a Kronecker $3$-polytope' is motivated by the following result.

\begin{thm}[{\cite[Theorem 1.7]{mafkpr}}]
\label{thm:ab}
The Kronecker product $\calP=J\wedge K_2$ is a $3$-polytope if and only if $J$ satisfies Definition \ref{cond1}. To construct $\calP$, one starts with two copies of $J'$, with vertex labellings such that if $(v,x)$ belongs to one copy, then $(v,y)$ belongs to the other, and then one adds the edges
\[(a_i,x)(b_i,y), \quad (a_i,y)(b_i,x), \qquad 1\leq i\leq m.\]
\end{thm}

Perhaps surprisingly, $J$ need not be planar, e.g.~if $J$ is the graph obtained from the cube by adding the diagonals of a face, then $J$ is non-planar, while $J\w K_2\simeq C_4\q P_4$ is a $3$-polytope \cite[Introduction]{mafkpr}.

It was also shown in \cite{mafkpr} that if $\calP$ is a $3$-polytopal Cartesian product, then either $\calP\simeq C_n\q P_m$ for some integers $n\geq 3$ and $m\geq 2$, i.e.~$\calP$ is a stacked prism, or $\calP\simeq H\q K_2$, where $H$ is outerplanar and Hamiltonian, i.e.~$H$ is a polygon with some added diagonals \cite[Proposition 1.9]{mafkpr} -- see also \cite{behmah}. We note that an outerplanar graph is Hamiltonian if and only if it is $2$-connected.

Several natural questions arise from \cite{mafkpr}.
%\begin{empheq}[box=\fbox]{align}
		%\label{eq:que}
		%\notag&\text{Which face-regular $3$-polytopes are Kronecker products?} \\\notag&\text{Which vertex-regular $3$-polytopes are Kronecker products?}
        %\\\notag&\text{Which planar graphs are simultaneously Cartesian and Kronecker products?}
%\end{empheq}
\begin{align}
&\text{
Which $3$-polytopal Kronecker products are face-regular?
}\label{q1}\\
&\text{
Which $3$-polytopal Kronecker products are vertex-regular?
}\label{q2}\\
&\text{
Which $3$-polytopes are Kronecker and/or Cartesian graph products in more than one way?
}\label{q3}
\end{align}

For instance, the cube is a solution to all three questions, as it is $4$-face regular, $3$-vertex regular, the Cartesian product of the square and $K_2$, and the Kronecker product of the tetrahedron and $K_2$. It is also the smallest $3$-polytopal Kronecker product.

In \cite[Proof of Proposition 1.2]{mafkpr}, it was also proven that, if $\calP=J\w K_2$ is a polyhedral graph, then the connectivity of $\calP$ is $3$. The argument is summarised as follows. For $\calP$ a $(p,q)$ polyhedral graph with $r$ faces, since $\calP$ is bipartite it has no $3$-cycles, thus $2q\geq 4r$, with equality if and only if $\calP$ is a quadrangulation. Applying Euler's formula for planar graphs $p-q+r=2$, we obtain $2q\leq 4p-8$, so that by $3$-connectivity $\calP$ has at least eight vertices of degree $3$. The following natural question arises.
\begin{equation}
\label{q4}
\text{Which $3$-polytopal Kronecker products have the least number of degree $3$ vertices?} 
\end{equation}

Note that any solution to \eqref{q4} is also a solution to \eqref{q1}, as it is a quadrangulation. An equivalent formulation of \eqref{q4} is, which $3$-polytopal Kronecker products have degree sequence
\begin{equation}
\label{eq:ds}
4^{p-8},3^8
\end{equation}
(where exponents denote quantities of vertex degrees)?

The focus of our work is to answer the above (and related) questions in extremal and combinatorial graph theory.

\subsection{Main results on regular graphs}
\label{sec:res}
Let us consider \eqref{q1}. If $\calP$ is a $t$-face-regular polyhedron, then $3\leq t\leq 5$. If the $t$-face-regular polyhedron $\calP$ is also a Kronecker product of graphs, then it is bipartite, hence $t=4$ i.e.~$\calP$ is a quadrangulation of the sphere. For this reason, \eqref{q1} may be thought of as a question in extremal graph theory, as the solutions are also the polyhedral Kronecker products with largest edge/vertex ratio (by the handshaking lemma). For other extremal problems on polyhedral graphs, see \cite{mafpo2,mafpo3,mafwil,gasmaf}.

To answer \eqref{q1}, we have the following two theorems, to be proven in Section \ref{sec:quad}.

\begin{thm}
\label{thm:quad}
The product $\calP=J\wedge K_2$ is a $3$-connected quadrangulation of the sphere if and only if the following are all satisfied. The conditions of Definition \ref{cond1} hold for $J$; in $J'$, each region is quadrangular, except possibly the region $\calR$ containing the vertices $a_1,a_2,\dots,a_m,b_1,b_2,\dots,b_m$; we may label the vertices of $\calR$ by
\[\calR=[v_1,v_2,\dots,v_{2\ell}], \quad \ell\geq 2,\]
where
\begin{align}
\label{eq:risi}
\notag
a_i=v_{r_i},& \quad 1\leq i\leq m,\\
b_i=v_{s_i},& \quad 1\leq i\leq m,
\end{align}
$r_1=1$, $r_m$ is even, $s_1=r_m+1$, $s_m=2\ell$, and for every $1\leq i\leq m-1$ we have
\begin{equation}
\label{eq:rsrel}
(r_{i+1}-r_i)+(s_{i+1}-s_i)=2.
\end{equation}
\end{thm}

Theorem \ref{thm:quad} is illustrated in Figure \ref{fig:q}.
\begin{figure}[h!]
\centering
\includegraphics[width=5.25cm]{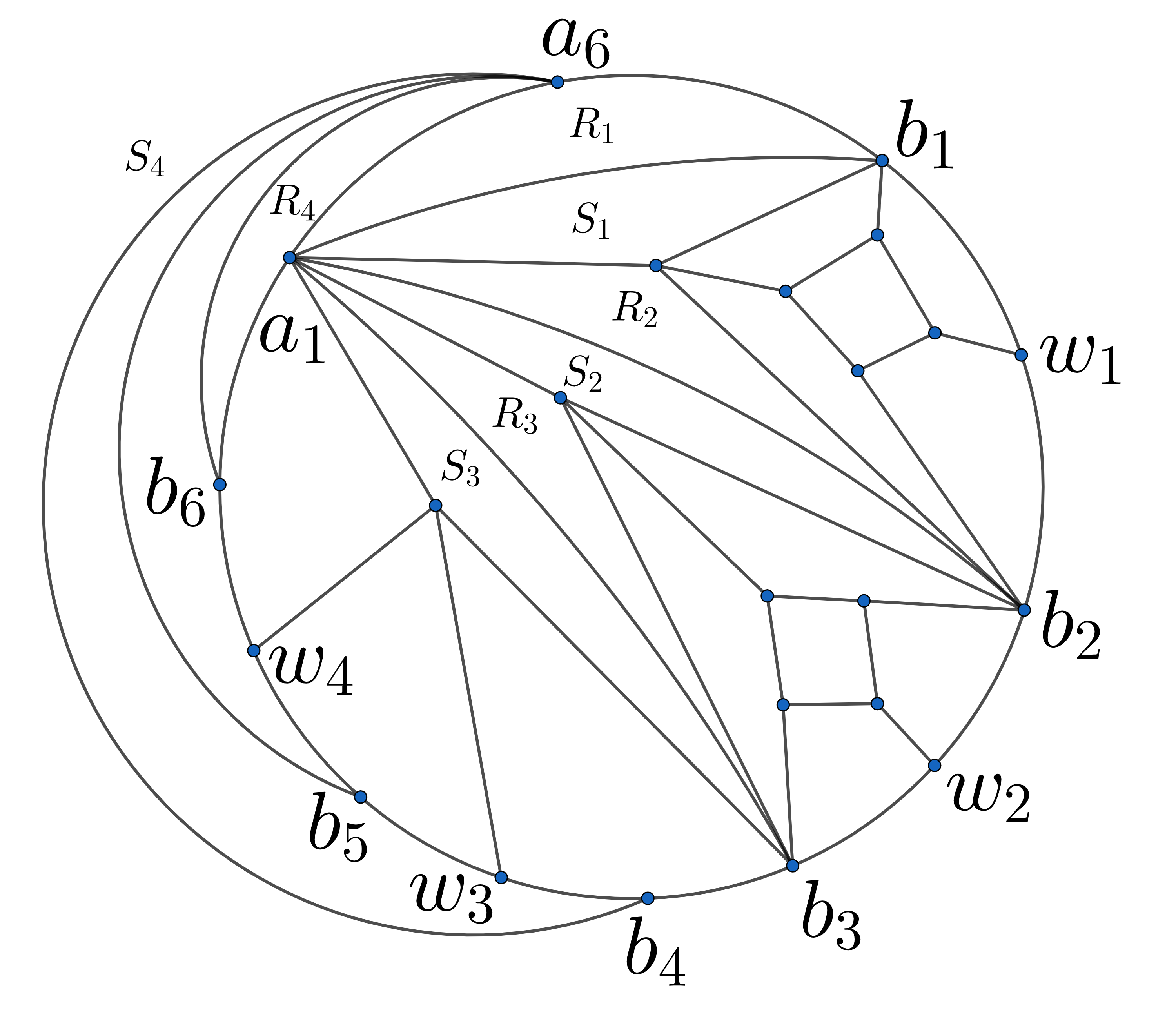}
\caption{Illustration of $J$ for Theorem \ref{thm:quad}. Here $m=6$, $2\ell=12$, $r_1=r_2=r_3=1$, $r_4=r_5=r_6=2$, $s_1=3$, $s_2=5$, $s_3=7$, $s_4=8$, $s_5=10$, and $s_6=12$.}
\label{fig:q}
\end{figure}

\begin{thm}
\label{thm:quadpl}
Let $J$ be planar, and $\calP=J\wedge K_2$ a $3$-connected quadrangulation of the sphere. Then $J$ is a polyhedron, and the odd faces of $J$ are $2\ell+1$ triangles with a common vertex and a $2\ell+1$-gon that has non-empty intersection with each of the triangles, for some $1\leq\ell\leq m-1$. In particular, $J$ satisfies Condition \ref{eq:c3} of Theorem \ref{thm:0123}. Moreover, there are infinitely many such graphs $J$.
\end{thm}

For example, in Figure \ref{fig:q} the vertex $a_1$ lies on seven triangular faces, and each of these faces has non-empty intersection with the external heptagon.

We remark that Theorem \ref{thm:quadpl} describes the types of faces in $J$ and their intersections, analogously to Theorem \ref{thm:0123}. On the other hand, Theorem \ref{thm:quad} discusses the structure of $J$ in the general case, much like Theorem \ref{thm:ab}.

The second author proved Theorems \ref{thm:0123} and \ref{thm:ab} independently from one another in \cite{mafkpr}. As opposed to this, interestingly, to prove Theorem \ref{thm:quadpl}, we will apply the general structure of $J$ uncovered in Theorem \ref{thm:quad}.

We now focus on \eqref{q4}, a special case of \eqref{q1}. The cube is of course a solution. In \cite{mafkpr}, a couple other solutions were found, namely the stacked cube $C_4\q P_4$ (recall that it is a Kronecker product), and the graph of Figure \ref{fig:ds2} (from \cite[Figure 2]{mafkpr}). More in general, we note that all stacked cubes have sequence \eqref{eq:ds}, and moreover, as we shall see in Theorem \ref{thm:dou}, the stacked cubes $C_4\q P_{2m}$, $m\geq 1$ are Kronecker products. More specifically, we record that the $2m$-stacked cube may be obtained via the Kronecker product of a non-planar graph with $K_2$ (see Corollary \ref{cor:tri} to follow), whereas the graph in Figure \ref{fig:ds2} is the Kronecker cover of a planar graph (Figure \ref{fig:ds1}).
\begin{figure}[ht]
\centering
\begin{subfigure}{0.45\textwidth}
\centering
\includegraphics[width=3.25cm]{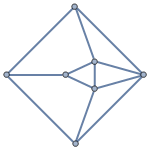}
\caption{A polyhedron $J$ of sequence $4443333$.}
\label{fig:ds1}
\end{subfigure}
    \hfill
\begin{subfigure}{0.53\textwidth}
\centering
\includegraphics[width=3.5cm]{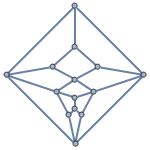}
\caption{The graph $J\w K_2$ is a polyhedron of sequence \eqref{eq:ds}.}
\label{fig:ds2}
\end{subfigure}
\caption{A solution to \eqref{q4}.}
\label{fig:ds}
\end{figure}

We have the following characterisation and construction for planar graphs $J$ such that $J\w K_2$ is an order $p$ polyhedral Kronecker product with degree sequence \eqref{eq:ds}.% Note that by definition $J$ has exactly four vertices of degree $3$, and the remaining of degree $4$.

\begin{thm}
\label{thm:3333}
Let $J\not\simeq K_4$ be a planar graph of order $p_J$, such that $J\w K_2$ is a $3$-polytope of order $p=2p_J$ with degree sequence \eqref{eq:ds}. Then $J$ is obtained from the graph in Figure \ref{fig:start} by iterating the transformations in Figure \ref{fig:38t} (at each step, one may take either \ref{fig:38t1}, or \ref{fig:38t2}, or the mirror image of \ref{fig:38t2}), followed by a final transformation either as in Figure \ref{fig:38f1}, or as in Figure \ref{fig:38f2}.
\\
Moreover, we can say the following about the faces and vertices of $J$. Four faces are triangular and the rest quadrangular. One triangular face is adjacent to all of the other three, and any two of the other three intersect in a vertex. In particular, $J$ satisfies Condition \ref{eq:c3} of Theorem \ref{thm:0123}.
\\
As for vertices, such a $J$ exists if and only if $p_J\equiv 1 \pmod 3$. Moreover, one of the vertices of degree $3$ is adjacent to all of the other three.
\end{thm}

\begin{figure}[ht]
\centering
\includegraphics[width=4.0cm]{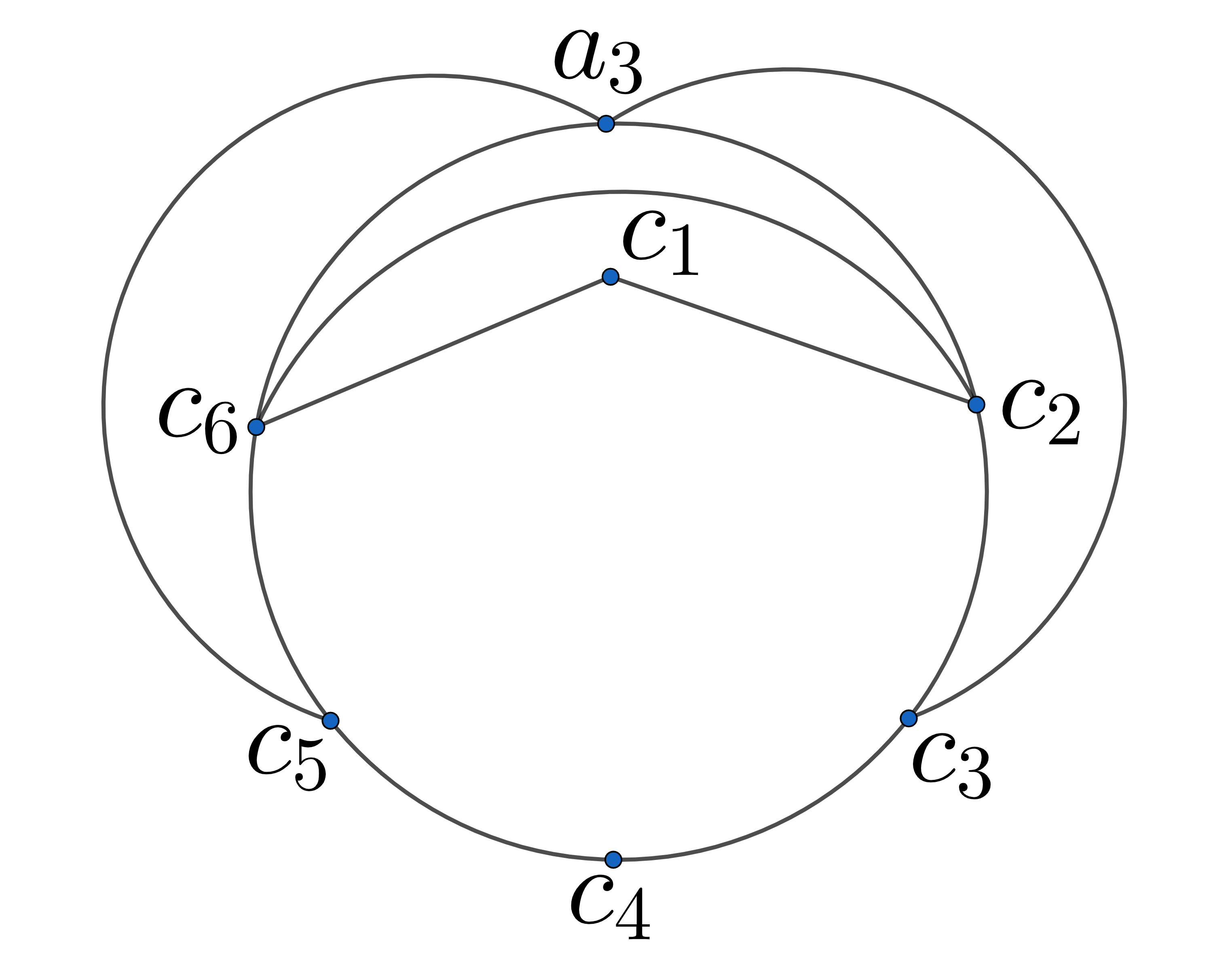}
\caption{Theorem \ref{thm:3333}, starting graph. The first transformation is applied to the inner hexagon $c_1,c_2,c_3,c_4,c_5,c_6$.}
\label{fig:start}
\end{figure}

\begin{figure}[h!]
\centering
\begin{subfigure}{0.49\textwidth}
\centering
\includegraphics[width=2.75cm]{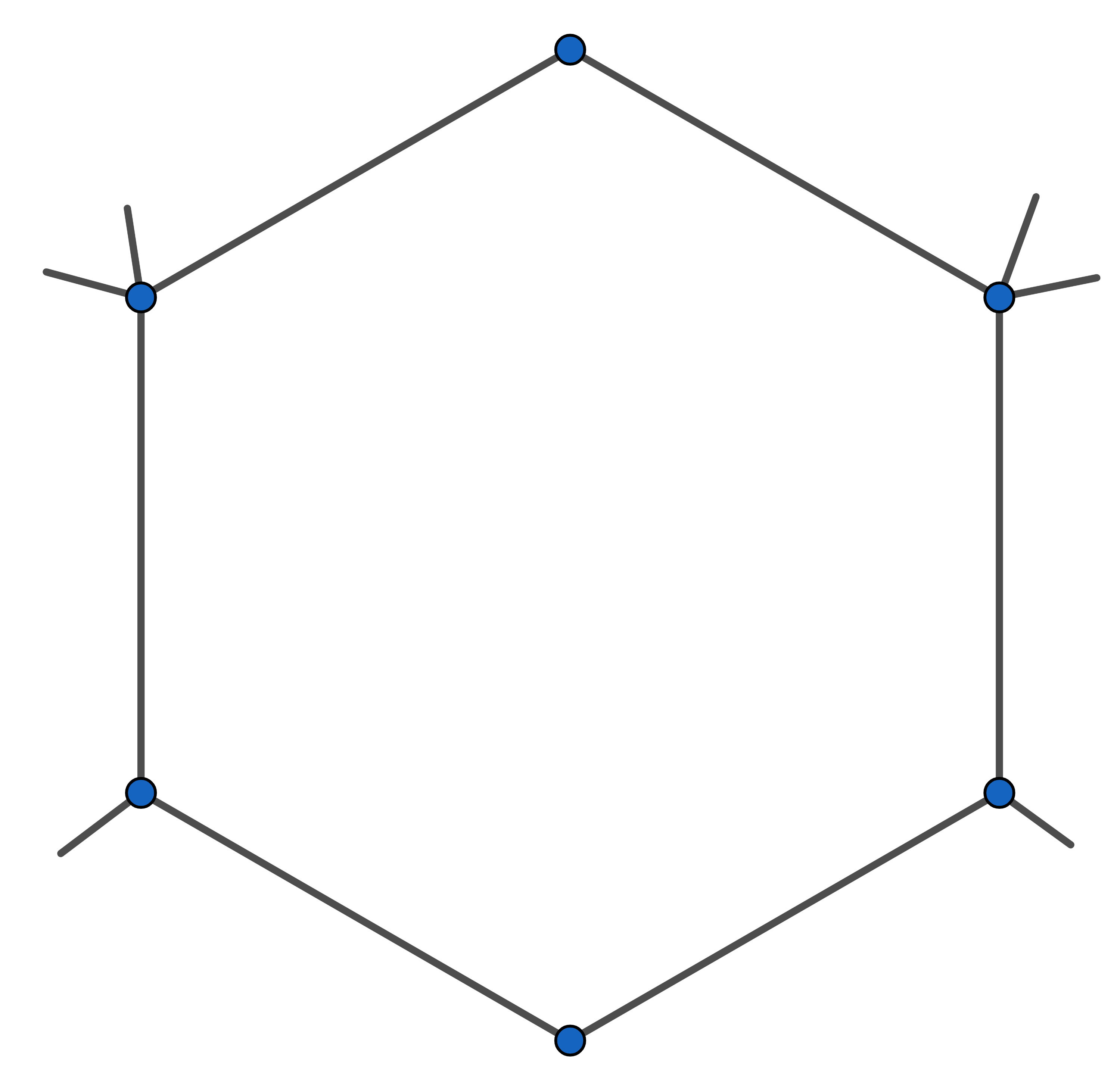}
\hspace{0.1cm}
$\rightarrow$
\hspace{0.1cm}
\includegraphics[width=2.75cm]{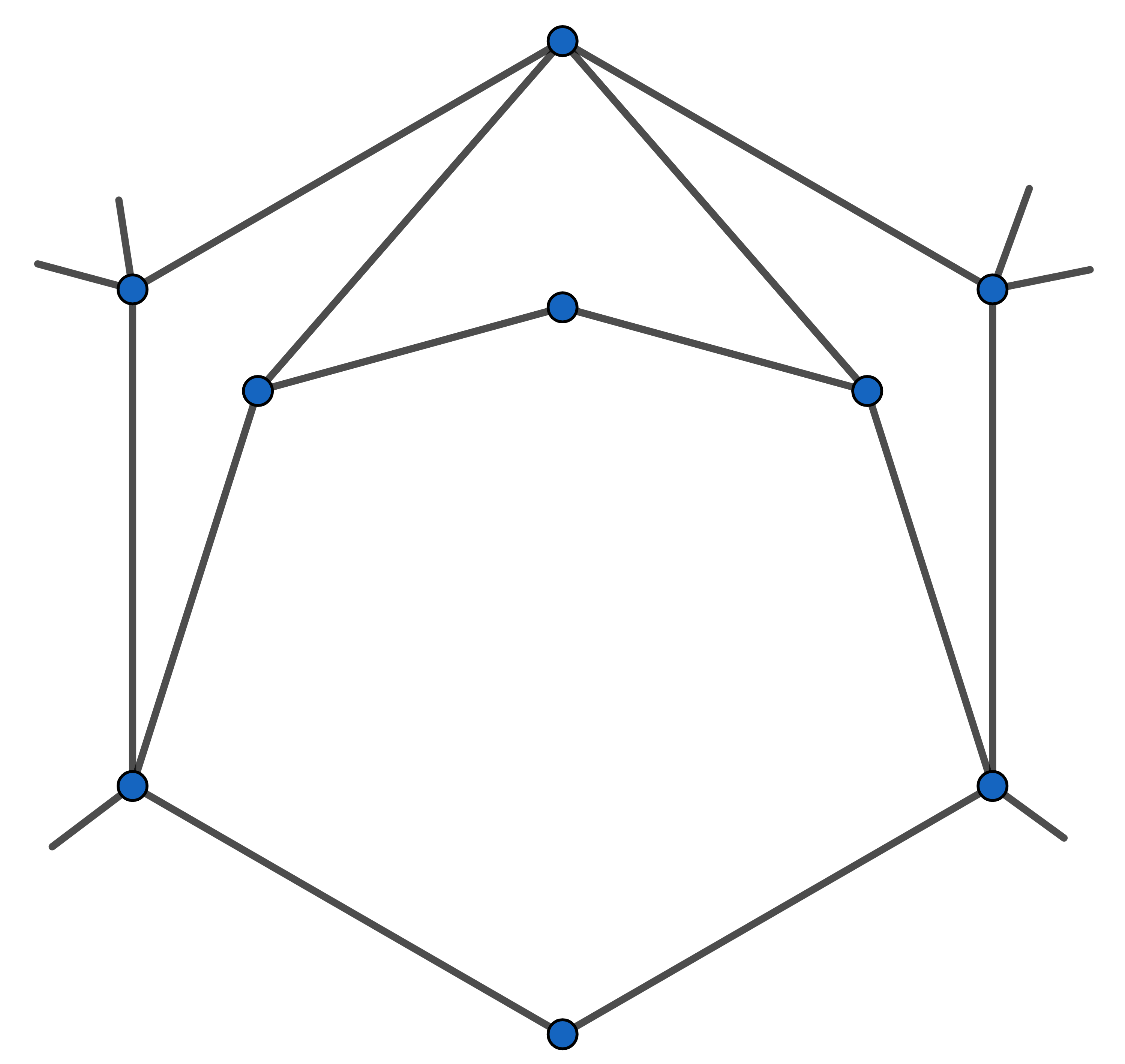}
\caption{}
\label{fig:38t1}
\end{subfigure}
    \hfill
\begin{subfigure}{0.49\textwidth}
\centering
\includegraphics[width=2.75cm]{38s.png}
\hspace{0.05cm}
$\rightarrow$
\hspace{0.05cm}
\includegraphics[width=2.75cm]{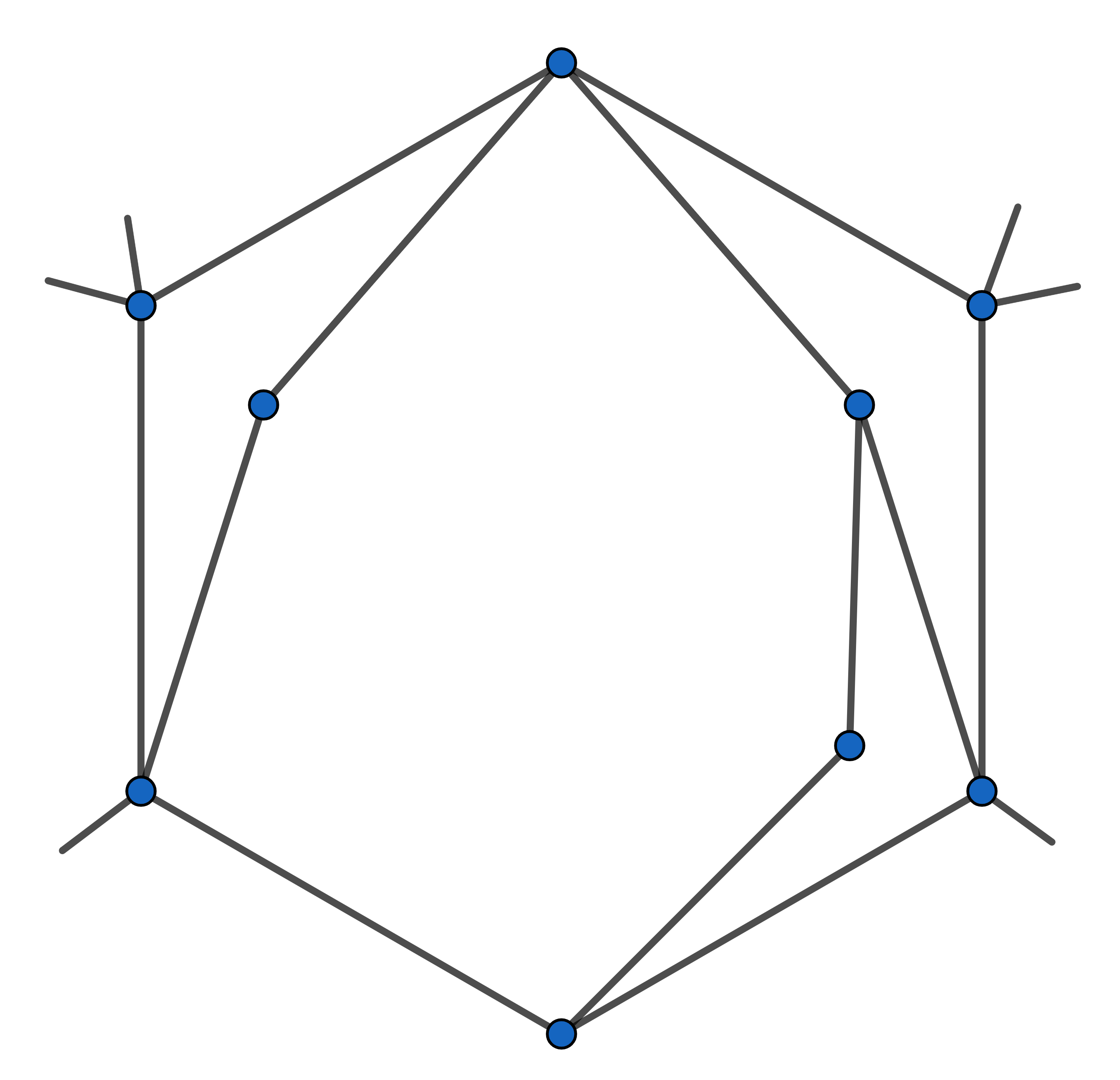}
\caption{}
\label{fig:38t2}
\end{subfigure}
\caption{Half-edges show the exact degrees of the vertices. All transformations are applied to a hexagon of vertex degrees $2,4,3,2,3,4$ in order. In the resulting graph, the next transformation is always applied to the inner hexagon, of same vertex degrees $2,4,3,2,3,4$. By mirror image of Figure \ref{fig:38t2} we mean w.r.t.~the vertical axis.}
\label{fig:38t}
\end{figure}

\begin{figure}[h!]
\centering
\begin{subfigure}{0.49\textwidth}
\centering
\includegraphics[width=2.75cm]{38s.png}
\hspace{0.05cm}
$\rightarrow$
\hspace{0.05cm}
\includegraphics[width=2.75cm]{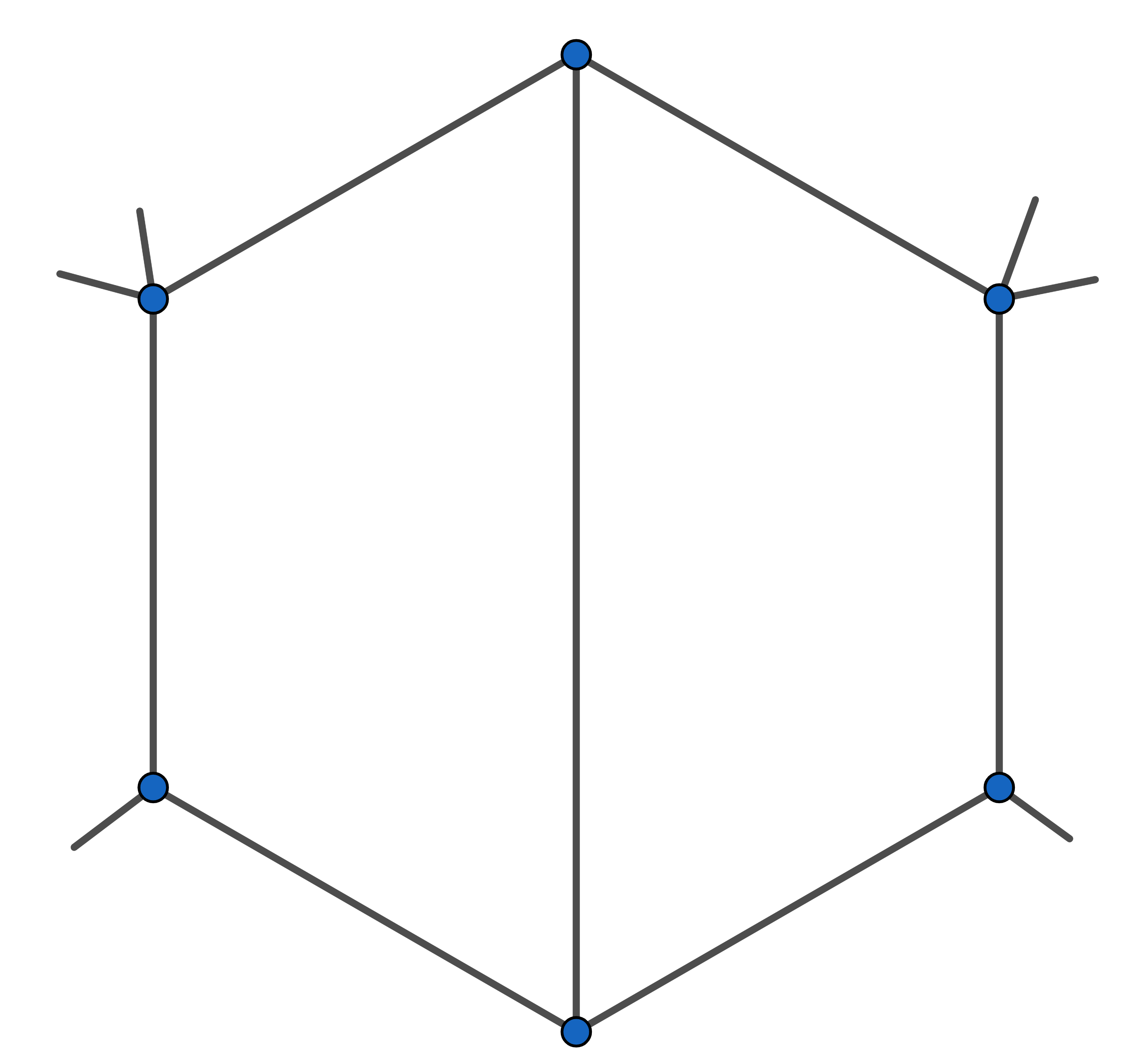}
\caption{}
\label{fig:38f1}
\end{subfigure}
    \hfill
\begin{subfigure}{0.49\textwidth}
\centering
\includegraphics[width=2.75cm]{38s.png}
\hspace{0.05cm}
$\rightarrow$
\hspace{0.05cm}
\includegraphics[width=2.75cm]{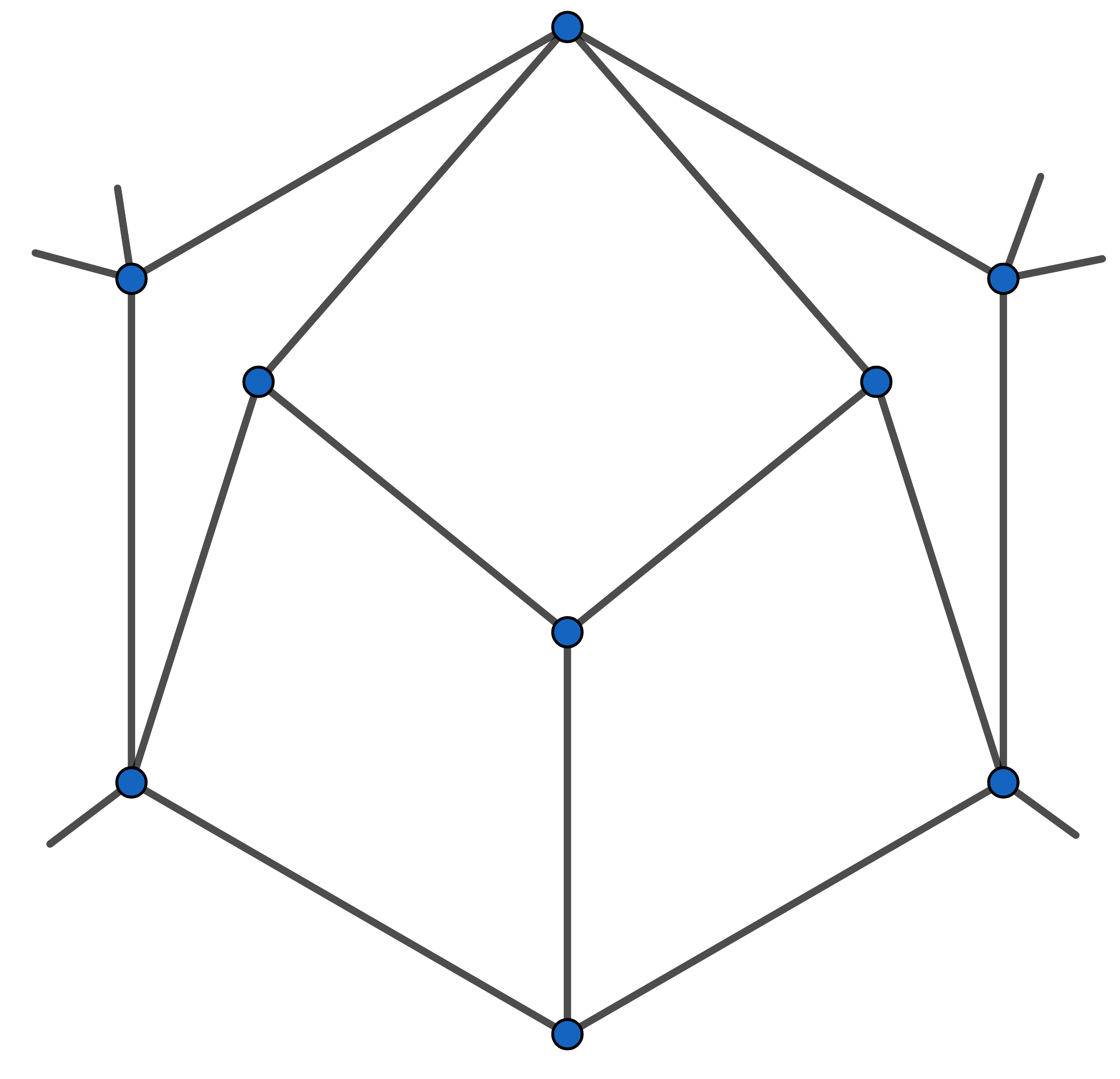}
\caption{}
\label{fig:38f2}
\end{subfigure}
\caption{The two possibilities for the final transformation in Theorem \ref{thm:3333}. Half-edges show the exact degrees of the vertices.}
\label{fig:38f}
\end{figure}

We now turn to \eqref{q2}. We begin with the above observation that if $\calP=J\wedge K_2$ is a polyhedron, then it has at least $8$ vertices of degree $3$, hence if $\calP=J\wedge K_2$ is a vertex-regular polyhedral Kronecker product, then $\calP$ is cubic. We note that this is the other extreme case opposed to \eqref{q4}, where the number of degree $3$ vertices was minimised.

%The following two theorems complement one another, by analysing the cases of $J$ planar and $J$ non-planar respectively. 

We have the following two theorems. They will be proven in Section \ref{sec:cub}.

\begin{thm}
\label{thm:cubpl}
Let $J$ be planar, and $\calP=J\wedge K_2$ a cubic $3$-polytope. If $J$ is a $3$-polytope, then it has at most four odd faces. Moreover, for each of the cases \ref{eq:c0}, \ref{eq:c1}, \ref{eq:c2}, \ref{eq:c3} for $J$ (Theorem \ref{thm:0123}), there are infinitely many solutions.
\end{thm}

\begin{thm}
\label{thm:cub}
The graph $\calP=J\wedge K_2$ is a cubic polyhedron if and only if $J$ is obtained from a cubic, planar, either $3$-connected or semi-hyper-$2$-connected multigraph $J_2$ by splitting certain edges of an even region $\calR_2$ one or more times, thus inserting the distinct vertices
\begin{equation}
\label{eq:addv}
a_1,a_2,\dots,a_m,b_1,b_2,\dots,b_m, \quad m\geq 2,
\end{equation}
followed by inserting the edges
\[a_1b_1,a_2b_2,\dots,a_mb_m,\]
such that the following are all satisfied. The region $\calR_2$ is adjacent to every odd region of $J_2$ (if any exist); the vertices \eqref{eq:addv} are inserted in $\calR_2$ either in the order 
\[a_1,a_2,\dots,a_m,b_m,b_{m-1},\dots,b_1\]
or in the order
\[a_1,a_2,\dots,a_m,b_1,b_2,\dots,b_m;\]
at least two edges of $\calR_2$ are split; each edge of $\calR_2$ is split an odd number of times if and only if it belongs to an odd region of $J_2$, thus the graph $J_1$ obtained from $J_2$ on inserting \eqref{eq:addv} is bipartite; $J_1$ is a simple graph; each of the graphs $J_1+a_ib_i$, $1\leq i\leq m$ is not bipartite.
\end{thm}

\begin{ex}
Theorem \ref{thm:cub} is illustrated in Figure \ref{fig:JeP}. We have a planar, cubic, semi-hyper-$2$-connected multigraph $J_2$ (Figure \ref{fig:J2}), a graph $J_1$ from which $J_2$ may be obtained by contracting the vertices $a_1,a_2,b_1,b_2$ (Figure \ref{fig:J1}), the cubic graph $J=J_1+a_1b_1+a_2b_2$ (Figure \ref{fig:J}), and the planar, $3$-connected, cubic graph $J\w K_2$ (Figure \ref{fig:JP}).
\begin{figure}[h!]
\centering
\begin{subfigure}{0.24\textwidth}
\centering
\includegraphics[width=3.5cm]{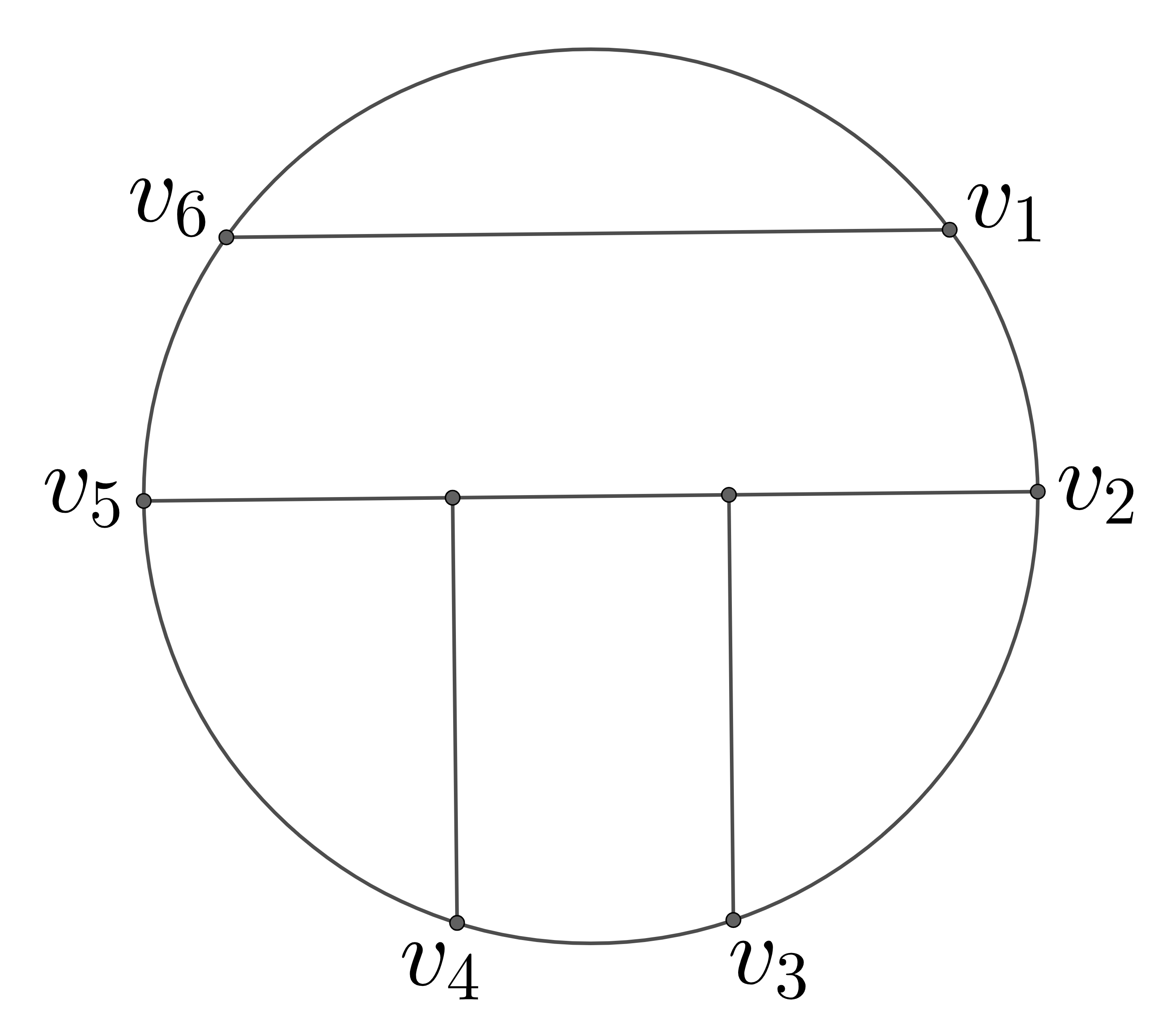}
\caption{$J_2$.}
\label{fig:J2}
\end{subfigure}
    \hfill
\begin{subfigure}{0.24\textwidth}
\includegraphics[width=3.5cm]{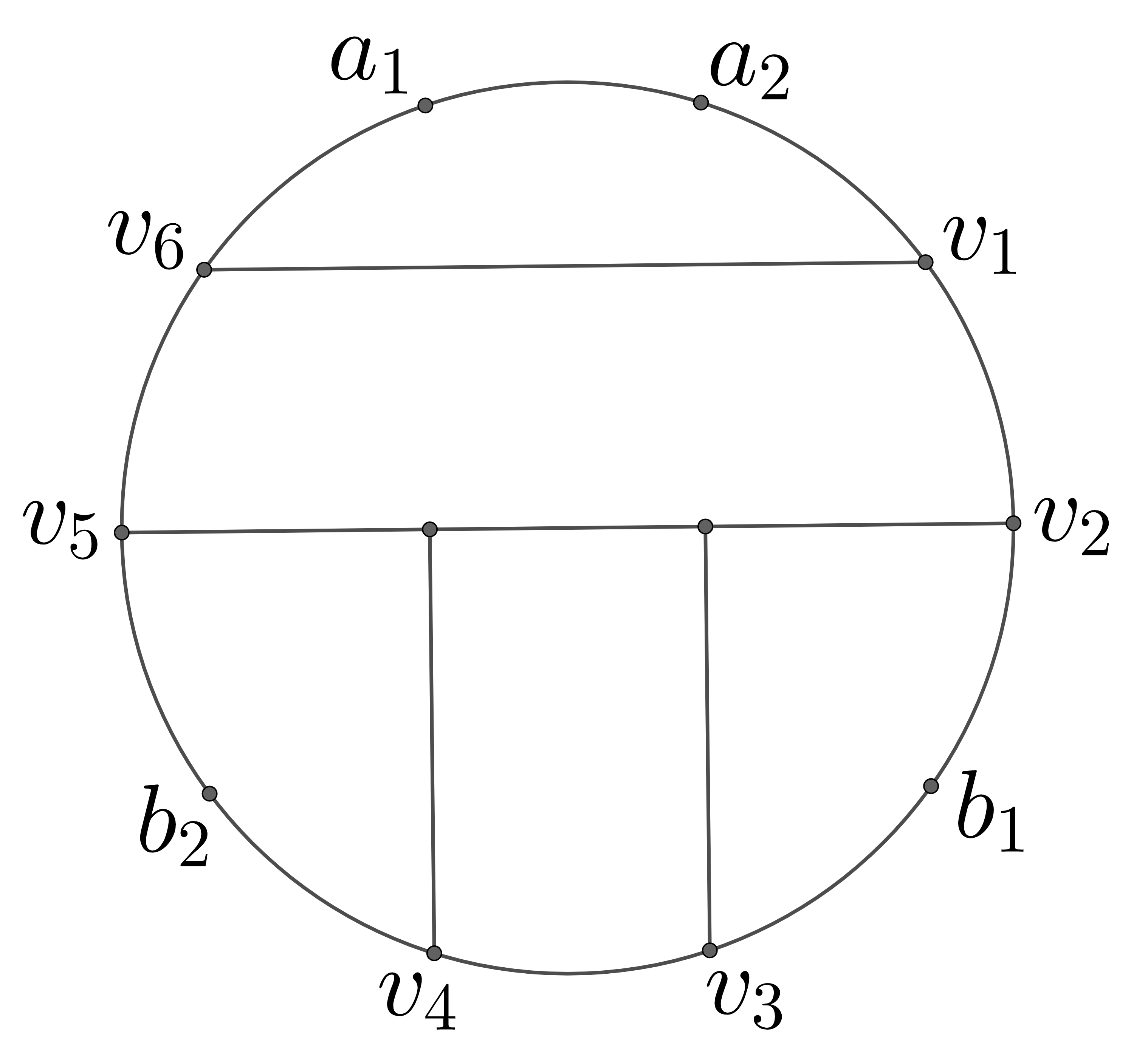}
\caption{$J_1$.}
\label{fig:J1}
\end{subfigure}
    \hfill
\begin{subfigure}{0.24\textwidth}
\centering
\includegraphics[width=3.75cm]{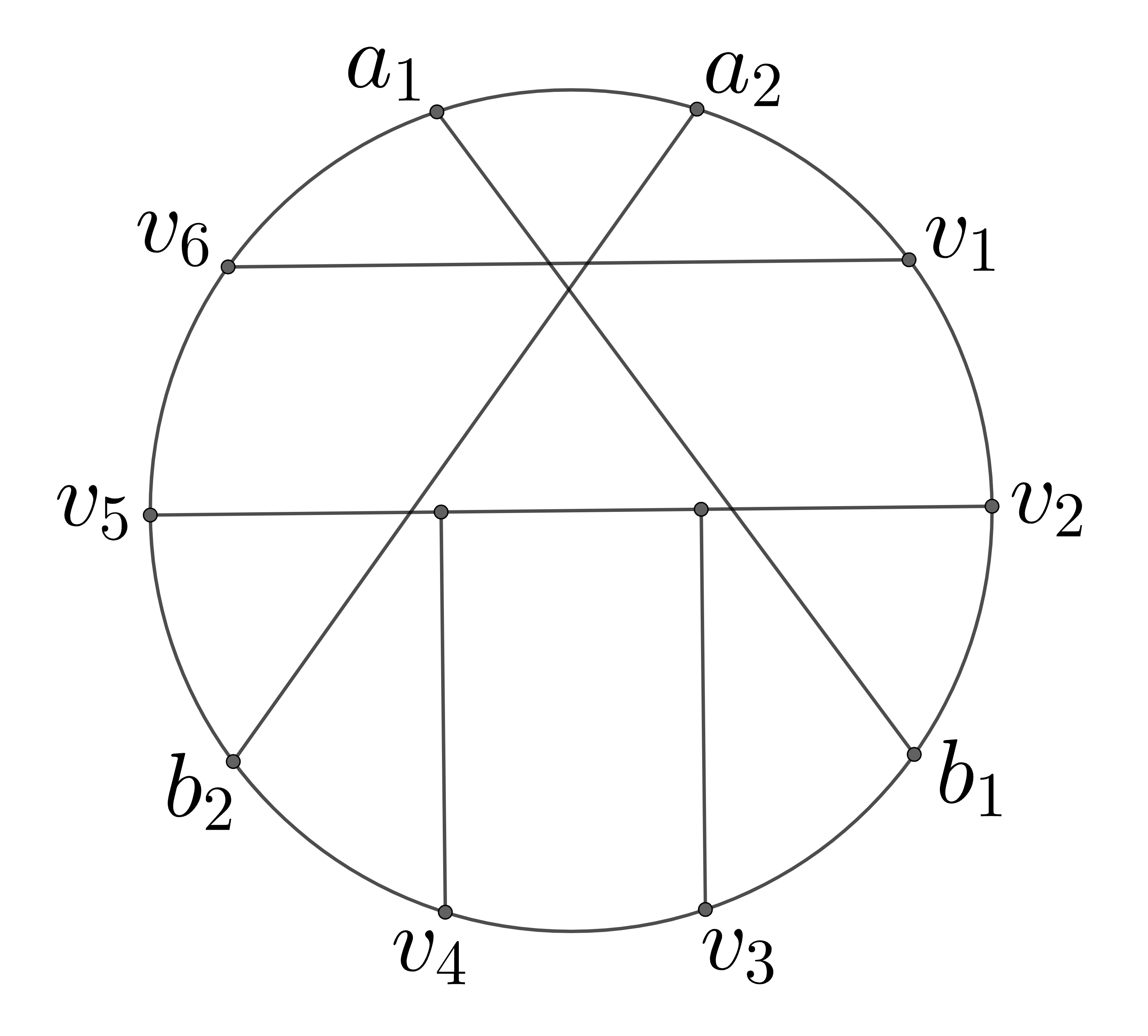}
\caption{$J$.}
\label{fig:J}
\end{subfigure}
    \hfill
\begin{subfigure}{0.24\textwidth}
\includegraphics[width=3.25cm]{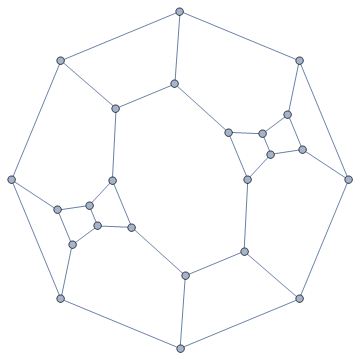}
\caption{$J\w K_2$.}
\label{fig:JP}
\end{subfigure}
\caption{Illustration of Theorem \ref{thm:cub}.}
\label{fig:JeP}
\end{figure}
\end{ex}

\subsection{Main results on cancellation and simultaneous products}
\label{sec:res2}
We now turn to the question \eqref{q3}. We begin by introducing the problem of Kronecker product cancellation. If the multigraphs $A,B,C$ satisfy
\[A\w C\simeq B\w C,\]
when is it true that $A\simeq B$? This problem has attracted much attention in recent literature \cite{lova71, fllp07, abay09, hamm09, aaghlt}.

If $C$ contains a loop, then cancellation occurs \cite[Proposition 9.6]{haimkl}. If $C$ contains an odd cycle, then cancellation occurs \cite[Theorem 9.10]{haimkl}. If we allow loops in $A,B$, then the graphs admitting cancellation are completely characterised in \cite[Theorem 9.16]{haimkl}.

Henceforth we impose instead that $A,B,C$ are simple, non-empty graphs, with $C$ bipartite (i.e., $C$ has no odd cycles). In this setting, whether $A\w C\simeq B\w C$ implies $A\simeq B$ or not is a deep and interesting open problem. As remarked in \cite{haimkl} after Proposition 9.9, here $A\w C\simeq B\w C$ implies $A\w K_2\simeq B\w K_2$. Therefore, to settle the Kronecker cancellation problem in the positive, it is sufficient to determine:
\[\text{when does } \quad A\w K_2\simeq B\w K_2 \quad \text{ imply } \quad A\simeq B?\]

Interestingly, it was observed in \cite{impi08} that the Petersen graph and the graph $B$ in Figure \ref{fig:notpet} have the same Kronecker cover, that is called the Desargues graph (Figure \ref{fig:desarg}). Therefore, even if we restrict to simple, non-empty graphs, cancellation may still fail.  
%there exist two non-isomorphic simple graphs $A,B$ such that $A\w K_2\simeq B\w K_2$ ($A$ is the Petersen graph, and the product is the Desargues graph).
%In other words, the Desargues graph is the Kronecker cover of two distinct simple graphs.
On the other hand, it was recently proved that apart from the Desargues graph, every other generalised Petersen graph is the Kronecker cover of at most one graph \cite{krpi19}.
\begin{figure}[ht]
\centering
\begin{subfigure}{0.48\textwidth}
\centering
\includegraphics[width=4.25cm]{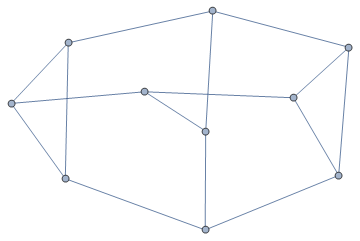}
\caption{The graph $B$.}
\label{fig:notpet}
\end{subfigure}
    \hfill
\begin{subfigure}{0.48\textwidth}
\centering
\includegraphics[width=4.cm]{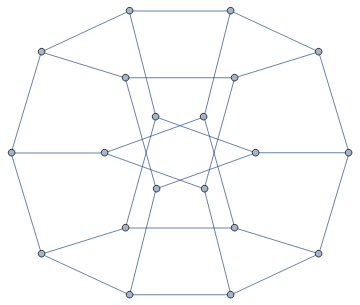}
\caption{The Desargues graph $B\w K_2$.}
\label{fig:desarg}
\end{subfigure}
\caption{We have $A\not\simeq B$, and $A\w K_2\simeq B\w K_2$, where $A$ is the Petersen graph.}
\label{fig:pet}
\end{figure}

It is worth noting that the analogous question for the Cartesian product is easy. In the expression $A\q C\simeq B\q C$ cancellation always holds provided that $C$ has at least one edge \cite[Theorem 6.21]{haimkl}.

Due to \cite[Proposition 1.1]{mafkpr}, investigating $3$-polytopes that are Kronecker products in distinct ways is equivalent to considering the following instance of the Kronecker cancellation problem. If we have a $3$-polytope
\[\calP\simeq J\w K_2\simeq L\w K_2,\]
when does it follow that $J\simeq L$? We will prove that here cancellation always holds.

\begin{thm}[Main Theorem]
\label{thm:JL}
If $J\w K_2\simeq L\w K_2$ and this product is a $3$-polytope, then $J\simeq L$.
\end{thm}

Theorem \ref{thm:JL} will be proven in Section \ref{sec:jl}.

Next, we characterise the planar graphs that are Cartesian products in two distinct ways, and the polyhedra that are simultaneously Cartesian and Kronecker products.
% As mentioned in section \ref{sec:prior}, if a $3$-polytope $\calP$ is a Cartesian product of graphs, then either $\calP$ is a stacked prism, or $\calP\simeq H\q K_2$, where $H$ is outerplanar and Hamiltonian. Accordingly, we need to find integers $n\geq 3$, $m\geq 2$ and outerplanar, Hamiltonian graphs $H$ such that $C_n\q P_m\simeq H\q K_2$.

\begin{thm}
\label{thm:dou}
A planar graph is expressible as Cartesian product in at most two distinct ways. Those expressible in two ways are the stacked cubes other than the cube itself
\[C_4\q P_m,\quad m\neq 2.\]%\simeq F_{2m}\q K_2, \quad m\neq 2.\]
The only polyhedral graphs expressible as both Cartesian and Kronecker products are the stacked prisms of the form
\[C_{4n+2}\q P_m, \quad n\geq 1, \ m\geq 2
\qquad\text{ or }\qquad
C_{4n}\q P_{2m}, \quad n,m\geq 1,\]
and the graphs of the form
\[H\q K_2,\]
where $H$ is obtained from a $2\ell$-gon $[u_1,u_2,\dots,u_{2\ell}]$ by adding diagonals $u_iu_j$ in such a way that the resulting graph is still bipartite and planar, with
\[u_iu_{(i+\ell\mod 2\ell)}\not\in E(H), \quad 1\leq i\leq 2\ell,\]
and
\[u_iu_j\in E(H)\iff u_{(i+\ell\mod 2\ell)}u_{(j+\ell\mod 2\ell)}\in E(H).\]
\end{thm}

% In the next two results, we characterise the integers $n\geq 3$, $m\geq 2$ and the graphs $J$ such that $C_n\q P_m\simeq J\w K_2$, and we characterise the graphs $H,J$ such that $H\q K_2\simeq J\w K_2$ and the product is a polyhedron. 

%(the combinations of these two results will completely answer \ref{q3}); in a related question, the integers $n,m$ and the graphs $H$ such that $C_n\q P_m\simeq H\q K_2$, i.e.~we characterise the polyhedra that can be though of as Cartesian products in two distinct ways.

Theorem \ref{thm:dou} will be proven in Sections \ref{sec:cc}, \ref{sec:ck1}, and \ref{sec:ck2}. The other way to express the stacked cube as Cartesian product is $F_{2m}\q K_2$, where $F_{2m}$ is the ladder graph. The cube itself is a Cartesian product in only one way, since $C_{4}\simeq F_4$ and $P_2\simeq K_2$. Other details on these constructions will follow in Sections \ref{sec:cc}, \ref{sec:ck1}, and \ref{sec:ck2}. Theorem \ref{thm:dou} is illustrated in Figures \ref{fig:cc}, \ref{fig:qq}, \ref{fig:ppone}, and \ref{fig:pptwo}.

Combining Theorems \ref{thm:JL} and \ref{thm:dou} %, and \cite[Theorem 6.21]{haimkl}
we deduce the following.

\begin{cor}
\label{cor:tri}
A polyhedral graph is expressible as Kronecker and/or Cartesian products in at most three distinct ways. Those expressible in three ways are exactly the stacked cubes with even index (other than the cube itself)
\[C_{4}\q P_{2m}\simeq F_{4m}\q K_{2}\simeq J\w K_2, \quad m\geq 2,\]
where $J$ is the non-planar graph obtained from the stacked cube 
\[C_4\q P_m, \quad C_4=[u_1,u_2,u_3,u_4], \quad P_m=v_1,v_2,\dots,v_m\]
by adding the two diagonals $(u_1,v_1)(u_3,v_1)$ and $(u_2,v_1)(u_4,v_1)$.
\end{cor}
The case $m=2$ of Corollary \ref{cor:tri} is depicted in Figures \ref{fig:qqtwo} and \ref{fig:ppone}.

%For an illustration, see Figure \ref{fig:Jex}.

\subsection{Further considerations, and an open problem}
We end the introduction with a few extra considerations on questions related to \eqref{q1}, \eqref{q2}, \eqref{q3}, and \eqref{q4}.

From Theorem \ref{thm:dou} we deduce the following. The only $3$-polytopes satisfying both \eqref{q1} and \eqref{q3} are the stacked cubes $C_4\q P_m$, $m\geq 2$. Moreover, the only $3$-polytopes satisfying both \eqref{q2} and \eqref{q3} are the $2n$-gonal prisms, $n\geq 2$. This is consistent with the cube being the only $4$-face regular, $3$-vertex regular polyhedron.

\paragraph{An open problem.}
Question \eqref{q4} asks to minimise the vertices of degree $3$ among polyhedral Kronecker products. We consider the analogue of this question for faces. Let $\calP$ be a $(p,q)$ polyhedron with $r$ faces, that is a Kronecker product. On one hand, recalling that $\calP$ is bipartite, by the handshaking lemma we have
\[2q=4r_4+6r_6+8r_8+\dots+2Mr_{2M}\geq 4r_4+6(r-r_4)=6r-2r_4,\]
where $r_i$ counts the quantity of $i$-gonal faces in $\calP$, and $M\geq 2$. Hence by Euler's formula,
\[r_4\geq 3r-q=3(2-p+q)-q=6-3p+2q.\]
On the other hand, by the handshaking lemma and the $3$-connectivity of $\calP$, we have $2q\geq 3p$, with equality if and only if $\calP$ is cubic. Altogether we obtain
\[r_4\geq 6,\]
meaning that any polyhedral Kronecker product has at least six $4$-gonal faces. This is the analogue for faces of the result that any polyhedral Kronecker product has at least eight vertices of valency $3$ \cite[proof of Proposition 1.2]{mafkpr}.
\\
It would be interesting to have a characterisation and/or construction of the extremal class of polyhedral Kronecker products with exactly six $4$-gonal faces. The cube and hexagonal prism are examples. Another example is given in Figure \ref{fig:3r1}, right. For any solution $\calP$, equality holds in all the above inequalities, hence any solution is a cubic polyhedron ($2q=3p$), and its faces are all quadrangular or hexagonal ($r_6=r-r_4=r-6$). Equivalently, if $\calP$ is a solution, then $\calP^*$ is a maximal planar graph of degree sequence $6^{r-6},4^6$.

Note that for polyhedral Kronecker products, we have now considered vertices of degree $3$ (the smallest possible) and their maximisation \eqref{q2} and minimisation \eqref{q4}. We have also considered faces of valency $4$ (the smallest possible) and their maximisation \eqref{q1} and minimisation (the open problem above). % The cube is (the only) solution to all four questions.
We now consider the analogous questions for the Cartesian product. They are relatively straightforward to answer. Nevertheless, it is interesting to compare the respective analogous results for Kronecker and Cartesian products, highlighting similarities and differences.

A graph $\calP$ is a $4$-face-regular polyhedral Cartesian product (analogue of \eqref{q1} for the Cartesian product) if and only if $\calP\simeq H\q K_2$, where $H$ is a quadrangulated polygon (i.e., to construct $H$ one starts with an even-sided polygon $C_{2\ell}$, $\ell\geq 2$, and then one adds diagonals in such a way that each region of $H$ is a quadrangle, save possibly $C_{2\ell}$). This follows readily from \cite[Proposition 1.9]{mafkpr}. A special case are the stacked cubes. % $C_4\q P_m$, $m\geq 2$, as we have
%\[F_{2m}\q K_2\simeq C_4\q P_m.\]
We observe that, analogously to Kronecker products, if a polyhedral Cartesian product is $t$-face regular, then $t=4$. This is because a planar Cartesian product of non-empty graphs always contains a $4$-gonal region.%, since one of the factors is always $P_m$, $m\geq 2$.
%We observe that, analogously to Kronecker products, if a polyhedral Cartesian product is $t$-face regular, then $t=4$, and if a polyhedral Cartesian product is $s$-vertex regular, then $s=3$.

The prisms are exactly the class of vertex-regular, polyhedral Cartesian products (analogue of \eqref{q2} for the Cartesian product). To see this, we start by checking that both of the two types of polyhedral Cartesian products contain vertices of degree $3$ (similarly to polyhedral Kronecker products), as follows. For the stacked prisms this is obvious. For the graphs $H\q K_2$ where $H$ is outerplanar and Hamiltonian, recall that $H$ is simply a polygon possibly with some added diagonals. Therefore, $H$ contains at least two vertices of degree $2$, thus $H\q K_2$ contains at least four vertices of degree $3$. Hence if a polyhedral Cartesian product is regular, then it is cubic. In the case of stacked prisms, the cubic ones are the prisms. If $H\q K_2$ is cubic, then $H$ is just a polygon, so that the product is again a prism.

As for polyhedral Cartesian products with fewest vertices of degree $3$ (analogue of \eqref{q4} for the Cartesian product), these are exactly the graphs $H\q K_2$ where $H$ %is a triangulated polygon other than the triangle itself.
has exactly two vertices of degree $2$. The corresponding products have exactly four vertices of degree $3$. As opposed to this, the $m$-stacked $n$-gonal prism has $2n\geq 6$ vertices of degree $3$. Interestingly, for the Cartesian product \eqref{q4} is not a special case of \eqref{q1}.% For the former one considered triangulated polygons, for the latter quadrangulated polygons.

The polyhedral Cartesian product with fewest $4$-gonal faces (analogue of the open problem above for the Cartesian product) is the triangular prism, that has three $4$-gons. Indeed, if $H$ is an $\ell$-gon with some added diagonals, then $H\q K_2$ has at least $\ell$ quadrangular faces. Recall that Cartesian products of non-empty graphs always contain $4$-cycles (in the analogous question for polyhedral Kronecker products, recall that they are bipartite hence they contain a $4$-gonal face).

\paragraph{Plan of the paper.}
The rest of this paper deals with the proofs of the results stated in Sections \ref{sec:res} and \ref{sec:res2}. Section \ref{sec:q} is dedicated to quadrangulated Kronecker products: in Section \ref{sec:quad} we will prove Theorems \ref{thm:quad} and \ref{thm:quadpl}, and in Section \ref{sec:quadpl} we will prove Theorem \ref{thm:3333}. In Section \ref{sec:cub} we will instead focus on the vertex-regular cubic case, and prove Theorems \ref{thm:cubpl} and \ref{thm:cub}. In Section \ref{sec:jl} we will prove the Main Theorem \ref{thm:JL} on Kronecker cancellation. Sections \ref{sec:cc}, \ref{sec:ck1}, and \ref{sec:ck2} are dedicated to proving Theorem \ref{thm:dou} on simultaneous products.

\section{Quadrangulated polyhedral Kronecker products}
\label{sec:q}
\subsection{Proof of Theorems \ref{thm:quad} and \ref{thm:quadpl}}
\label{sec:quad}
If $\calP$ is a face-regular polyhedron and a Kronecker product, then it is $4$-face-regular, i.e.~a quadrangulation of the sphere. This is because $\calP$ is bipartite.

\begin{comment}
\begin{lemma}
If $J$ is planar, and $\calP=J\wedge K_2$ a quadrangulation of the sphere, then $J$ has at least four odd regions.
\end{lemma}
\begin{proof}
If $J$ has exactly two odd regions, then these are disjoint \cite[Theorems 1.3 and 1.4]{mafkpr}. Referring to Condition \ref{cond1}, the vertices $a_1x$, $b_1x$, $a_1y$, $b_1y$ all lie on one face of $\calP$. Hence this face cannot be quadrangular.
\end{proof}
\end{comment}

\begin{proof}[Proof of Theorem \ref{thm:quad}]
The reader may refer to Figure \ref{fig:q} throughout this proof. The conditions of Definition \ref{cond1} hold by Theorem \ref{thm:ab}. According to \cite[Proof of Lemma 3.7]{mafkpr}, $\calP$ is obtained from two copies of $J'$ by adding the edges $(a_i,x)(b_i,y)$, and $(a_i,y)(b_i,x)$, for $1\leq i\leq m$, hence in particular if $J'$ contains a non-quadrangular region other than $\calR$, then $\calP$ is not a quadrangulation.

It remains to show the condition on the vertex labelling for $\calR$. We begin by defining $v_1:=a_1$ and the integers $s_1,r_2,s_2,\dots,r_m,s_m$ as in \eqref{eq:risi}. From Theorem \ref{thm:ab}, we know that the $a_i$'s and $b_i$'s lie on $\calR$ either in the order $a_1,a_2,\dots,a_m,b_m,b_{m-1},\dots,b_1$ or in the order $a_1,a_2,\dots,a_m,b_1,b_2,\dots,b_m$. We can exclude the former case, otherwise $a_1x,b_1x,a_1y,b_1y$ would all lie on the same face in $\calP$, hence this face would not be quadrangular because it would contain other vertices to not contradict the definition of Kronecker product.

Next, the vertices
\[a_1x,b_1y,b_2y',a_2x'\]
where either $x'=x$ and $y'=y$, or vice versa, lie on the same face by construction which is quadrangular by hypothesis. Therefore, the distance between $a_1$ and $a_2$ along $\calR$ is one of $0,1,2$, and the distance between $b_1$ and $b_2$ is $2,1,0$ respectively. The same argument of course holds with the indices $i,i+1$ in place of $1,2$, for $1\leq i\leq m-1$, and for the vertices
\[a_iy,b_ix,b_{i+1}x',a_{i+1}y'.\]
In other words, \eqref{eq:rsrel} holds for each $1\leq i\leq m-1$.

Similarly, the vertices
\[a_mx,b_my,b_1x',a_1y'\]
lie on a quadrangular face, and so do
\[a_my,b_mx,b_1y',a_1x'.\]
As these vertices are all distinct, in fact
\[[a_mx,b_my,b_1x,a_1y] \quad \text
{and} \quad [a_my,b_mx,b_1y,a_1x]\]
are faces of $\calP$. This implies $s_1=r_m+1$ and $s_m=2\ell$. Since $J'+a_1b_1$ is not bipartite, the distance between $a_1,b_1$ along $\calR$ is even, thus the index $r_m$ is even.
% As $a_m,b_1$ are adjacent in $J'$, then $a_my,b_1x$ are adjacent in $\calP$. This means that the distance from $a_1$ to $a_m$ along $\calR$ is odd, thus $r_m$ is even.

On the other hand, thanks to the above considerations, it is immediate to see that if the conditions on $J$ of this theorem are all satisfied, then the product $\calP$ is indeed a $3$-connected quadrangulation of the sphere.
\end{proof}

%ADD FIGURE
%Note that $J$ is non-planar, unless $J$ is the tetrahedron (so that $J'$ is the square and $J\w K_2$ the cube). This is analogous to Lemma \ref{le:JH}.

\begin{ex}
\label{ex:m2}
If $m=2$, then by Theorem \ref{thm:ab} the vertices $a_1,a_2,b_1,b_2$ are distinct. It follows that $a_1,a_2$ and $b_1,b_2$ are adjacent (distance $1$), thus $\calR=[a_1,a_2,b_1,b_2]$ is a $4$-gon, hence $J'$ is a 
%$2$-connected
quadrangulation of the sphere. Moreover, by Theorem \ref{thm:ab}, the only potentially possible $2$-cuts in $J'$ are $a_1,b_1$ and $a_2,b_2$. On the other hand, to not violate the $3$-connectivity of the product $\calP$, these cannot be $2$-cuts, unless $J'$ is simply the square (i.e.~$J$ is the tetrahedron and $\calP$ the cube). Apart from this special case, $J'$ is $3$-connected.
% It is either the square, or a $3$-connected quadrangulation.
\end{ex}

%\begin{ex}
%If $J$ is as in Figure 
%ADD FIGURE
%\end{ex}

\begin{proof}[Proof of Theorem \ref{thm:quadpl}]
%It suffices to rule out that $J$ satisfies Condition \ref{eq:c2}. By contradiction,
For the case where $J$ is the tetrahedron the statement of the present theorem clearly holds, thus for the remainder of this proof let us exclude this case. If $m=2$, then as in Example \ref{ex:m2} $J'$ is a $3$-connected quadrangulation, and
\[\calR=[a_1,a_2,b_1,b_2].\]
Since
\[J=J'+a_1b_1+a_2b_2\]
and $J$ is planar, we deduce that one of the edges $a_1b_1$, $a_2b_2$ is internal to $\calR$, and the other external, say $a_1b_1$ is internal (refer to Figure \ref{fig:qpl1}). As $\calR$ is a face of $J'$, the edges of $J'$ not on $\calR$ are either all internal to $\calR$ or all external, w.l.o.g. the former holds. Therefore, we have the triangular faces
\[R_2:=[a_2,b_2,a_1] \quad \text{and} \quad S_2:=[a_2,b_2,b_1]\]
in $J$. Eliminating $a_1b_1$, $a_2b_2$ from $J$ leaves the quadrangulation $J'$, hence in $J$ the edge $a_1b_1$ belongs to two triangular faces $R_1$ and $S_1$, say. Altogether, $J$ has exactly four odd faces, and these are all triangles. Now $a_1$ belongs to $R_1,S_1,R_2$, but not to $S_2$, and moreover $S_2$ has non-empty intersection with each of $R_1,S_1,R_2$, as in Figure \ref{fig:qpl1}. Thereby, $J$ satisfies Condition \ref{eq:c3} of Theorem \ref{thm:0123}. The statement of the present theorem is verified in this case.

\begin{figure}[ht]
\centering
\begin{subfigure}{0.33\textwidth}
\centering
\includegraphics[width=3.25cm]{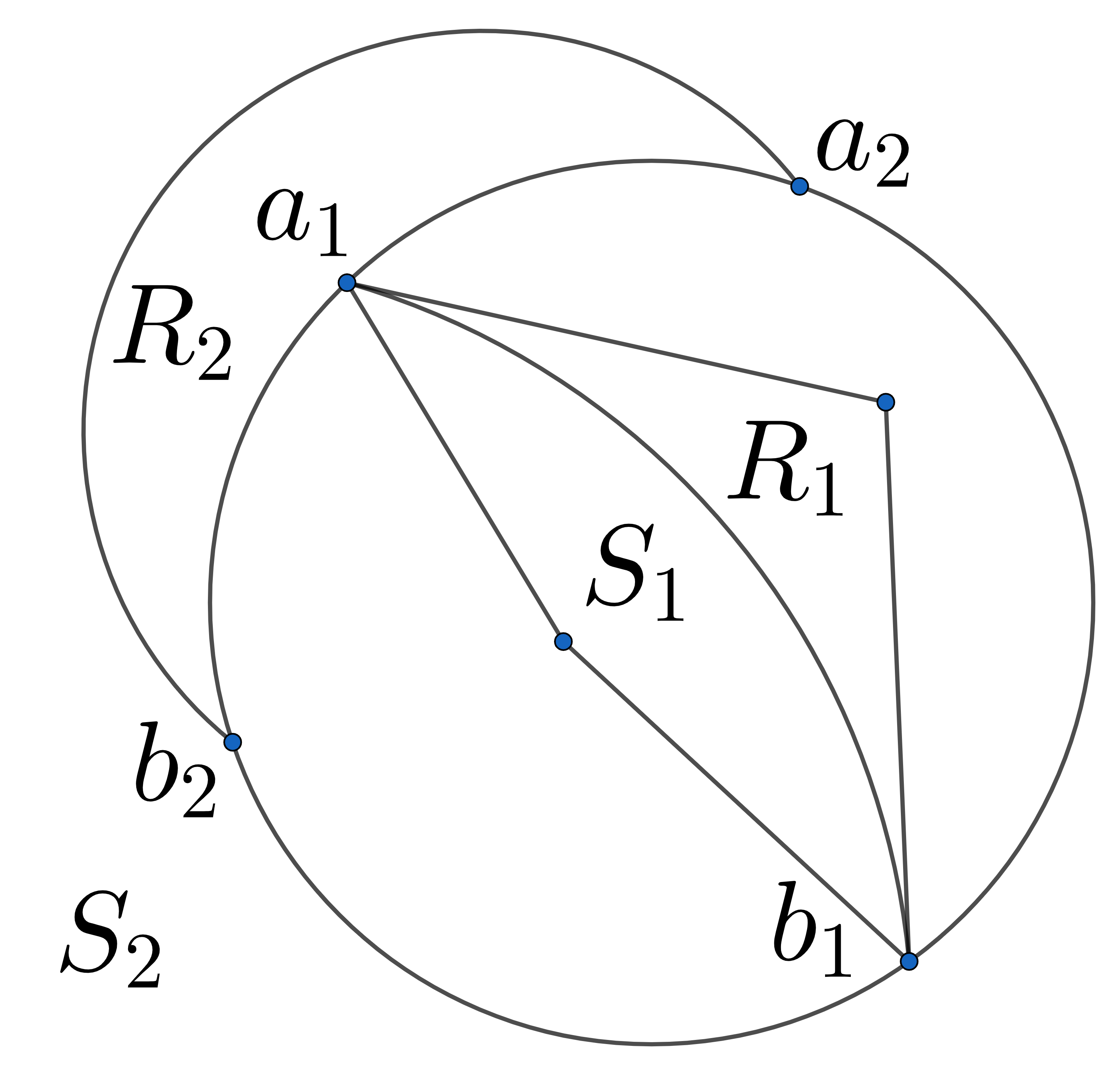}
\caption{The case $m=2$. Only a subgraph of $J$ is depicted.}
\label{fig:qpl1}
\end{subfigure}
    \hfill
\begin{subfigure}{0.65\textwidth}
\centering
\includegraphics[width=5.5cm]{qpl2.png}
\caption{An illustration of $J$ when $m\geq 3$. Here $m=6$ and $i=3$.}
\label{fig:qpl2}
\end{subfigure}
\caption{Theorem \ref{thm:quadpl}.}
\label{fig:qpl}
\end{figure}

Now let $m\geq 3$. As seen in the proof of Theorem \ref{thm:quad}, the vertices
\[a_1,a_2,\dots,a_m,b_1,b_2,\dots,b_m\]
appear around $\calR$ in this order. We deduce that there does not exists $j$ such that $a_1,a_j,a_m$ are all distinct, and $b_1,b_j,b_m$ are all distinct, otherwise $J$ would contain a copy of $K_{3,3}$ and would hence be non-planar.
\begin{comment}
Thereby, if we have the edge $a_jb_j$, then either $a_j\in\{a_1,a_m\}$ or $b_j\in\{b_1,b_m\}$ (or both). We can assume w.l.o.g. that $a_1=a_2$.

Let $i\in\{2,3,\dots,m-1\}$ be such that
\[a_1=a_2=\dots=a_i\neq a_{i+1}.\]
There are two cases:
\[\text{either} \quad a_{i+1}=a_{i+2}=\dots=a_m, \qquad\quad \text{or} \quad b_{i+1}=b_{i+2}=\dots=b_m.\]
Let us consider the first case. 
\end{comment}
Thereby, w.l.o.g. we can assume that
\begin{equation}
\exists \ i\in\{1,2,\dots,m-1\} \ : \ a_1=a_2=\dots=a_i\neq a_{i+1}=a_{i+2}=\dots=a_m.
\end{equation}

According to Theorem \ref{thm:quad}, the distance between $a_1,a_{m}$ along $\calR$ is one of $0,1,2$. As $a_1\neq a_m$, it is either $1$ or $2$. Still by Theorem \ref{thm:quad}, $a_m,b_1$ are adjacent, and the distance between $a_1,b_1$ is even, hence the distance between $a_1,a_m$ along $\cal{R}$ is $1$. Again by Theorem \ref{thm:quad}, the distance along $\calR$ between $b_j,b_{j+1}$ is $1$ for $j=i$, and $2$ otherwise. We deduce that
\begin{equation}
\label{eq:cR}
\calR=[a_1,a_m,b_1,w_1,b_2,\dots,w_{i-1},b_i,b_{i+1},w_i,b_{i+2},w_{i+1},\dots,w_{m-2},b_m],
\end{equation}
as in Figure \ref{fig:qpl2}.% In the second case, we have instead
%\[\calR=[a_1,a_{i+1},w_1,a_{i+2},\dots,w_{m-i-1},a_m,b_1,w_{m-i},b_2,\dots,w_{m-2},b_i,b_m]\]
%(either way, the length of $\calR$ is $2m$). Now note that in the second case, we can relabel $\alpha_1=\dots=\alpha_i=a_1$, $\alpha_{i+1}=\dots=\alpha_m=b_m$, $\beta_1=b_i$, $\beta_2=b_{i-1}$, ...,  $\beta_i=b_1$, $\beta_{i+1}=a_m$, $\beta_{i+2}=a_{m-1}$, ..., $\beta_{m}=a_{i+1}$, and $\psi_1=w_{m-2}$, $\psi_2=w_{m-1}$, ..., $\psi_{m-2}=w_1$, so that the two cases are in fact equivalent, and we may always assume that \eqref{eq:cR} holds.

Now similarly to the case $m=2$, w.l.o.g. in $J$ the edges $a_1b_1$, ..., $a_1b_i$ are drawn internally to $\calR$, and $a_{m}b_{i+1}$, ..., $a_mb_m$ externally. The edges of $J'$ not on $\calR$ are either all internal or all external to $J'$. Let us assume for the moment that they are internal. By Theorem \ref{thm:quad}, all regions of $J'$ other that $\calR$ are quadrangular. It follows that, for $1\leq j\leq i$, $J$ has two triangular regions $R_j,S_j$ containing $a_1,b_j$, and the only other odd regions of $J$ are $R_{i+1}:=[a_1,a_m,b_m]$ and the $2i+1$-gon
\[S_{i+1}:=[a_m,b_1,w_1,b_2,\dots,w_{m-i},b_i,b_{i+1}].\]
It follows that $a_1$ belongs to all odd regions of $J$ save $S_{i+1}$, and $S_{i+1}$ has non-empty intersection with all other odd regions of $J$. Thereby, Condition \ref{eq:c3} of Theorem \ref{thm:0123} holds for $J$. If instead the edges of $J'$ not on $\calR$ are  all external to $J'$, the above of course still holds, but with $m-i$ in place of $i$. We have thus verified the statement of the present theorem, with $\ell\in\{i,m-i\}$.

To check that there are infinitely many such graphs $J$, it suffices to consider the solutions of the type depicted in Figure \ref{fig:qpl2} for different values of $m,i$.
\begin{comment}
Now let us check that there are infinitely many solutions. Let $J_0$ be any $3$-connected quadrangulation of the sphere, and $[u_1,u_2,u_3,u_4]$ a face of $J_0$. We define
\[J:=J_0+v+vu_1+vu_3+u_1u_3+u_2v.\]
Then $J$ satisfies Condition \ref{cond1} with $J=J'+a_1b_1+a_2b_2$, where $a_1=u_1$, $a_2=u_2$, $b_1=u_3$, and $b_2=v$. By \cite[Theorem 1.4]{mafkpr}, the graph $J\wedge K_2$ is a polyhedron. By construction, $J'$ is a quadrangulation, $\calR=[u_1,u_2,u_3,v]$, and $\{u_1,u_3\}$ is the only $2$-cut in $J'$. It follows that $J\wedge K_2$ is a $3$-connected quadrangulation. Different choices of $J_0$ will yield different solutions.
\end{comment}
\end{proof}

\subsection{Proof of Theorem \ref{thm:3333}}
\label{sec:quadpl}
\begin{proof}[{Proof of Theorem \ref{thm:3333}}]
Recall that $J$ satisfies Definition \ref{cond1}, and that
\[J'=J-a_1b_1-a_2b_2-\dots-a_mb_m, \quad m\geq 2.\]
The assumptions on $J$ of the present theorem imply those of Theorem \ref{thm:quadpl}. Inspecting the proof of said theorem, we record that the region $\calR$ of Definition \ref{cond1} may be written as in \eqref{eq:cR},
\[
\calR=[a_1,a_m,b_1,w_1,b_2,\dots,w_{i-1},b_i,b_{i+1},w_i,b_{i+2},w_{i+1},\dots,w_{m-2},b_m]
\]
for some $1\leq i\leq m-1$. In $J$, the vertex $a_1$ is adjacent to the distinct vertices $a_m,b_m,b_1,b_2,\dots,b_i$, while $a_m$ is adjacent to the distinct $a_1,b_1,b_{i+1},b_{i+2},\dots,b_m$. Since the maximum valency of $J$ is $4$ by hypothesis, we deduce that $m\leq i+2\leq 4$. There are thus few possibilities, that we can split up into three cases: either $m=2$ and $i=1$, or $m=4$ and $i=2$, or $m=3$ and $i\in\{1,2\}$. %(the case $m=3$ and $i=2$ is equivalent to $m=3$ and $i=1$, after the relabelling in the proof of Theorem \ref{thm:quadpl}).

If $m=2$ and $i=1$, then $\calR=[a_1,a_2,b_1,b_2]$, as in Figure \ref{fig:qpl1}. Since we assumed $J\not \simeq K_4$, there is another vertex $v$. W.l.o.g., $v$ is adjacent to $a_1$. Appealing to the proof of Theorem \ref{thm:quadpl}, $J'=J-a_1b_1-a_2b_2$ is a quadrangulation of the sphere, hence $v$ is also adjacent to $b_1$. Now $v$ cannot be adjacent to either $a_2$ or $b_2$, else $J'$ would have odd regions. On the other hand, the valency of $v$ is at least $3$, hence there must be other vertices in $J$. However, we cannot introduce any other vertices without violating either the planarity, or the $3$-connectivity, or the maximum valency of $J$, thus we can discard this case.

If $m=4$ and $i=2$, then
\[
\calR=[a_1,a_4,b_1,w_1,b_2,b_3,w_2,b_{4}].
\]
The neighbourhood of $a_1$ in $J$ is $\{a_4,b_4,b_1,b_2\}$. Moreover, each region of $J'$ except $\calR$ is a quadrangle. It follows that there exists a vertex $v$ not on $\calR$ adjacent to $a_4,b_4$ such  that $[a_1,a_4,v,b_4]$ is a $4$-cycle in $J$. Then $a_4$ is adjacent to at least $a_1,b_1,b_3,b_4,v$, i.e. $\deg_J(a_4)\geq 5$, contradiction.

It remains to consider $m=3$ and $i\in\{1,2\}$. If $i=1$, then we have
\begin{comment}
We write
\[
\calR=[a_1,a_3,b_1,w_1,b_2,b_3].
\]
As just observed for the previous case, there exists a vertex $v$ not on $\calR$ adjacent to $a_3,b_3$. We now introduce the labels
\begin{center}
\begin{tabular}{ccc}
$c_1=v$,&$c_2=a_3$,&$c_3=b_1$,
\\$c_4=w_1$,&$c_5=b_2$,&$c_6=b_3$.
\end{tabular}
\end{center}
\end{comment}
\[
\calR=[a_1,a_3,b_1,b_2,w_1,b_3].
\]
As for the case $m=3$ and $i=2$, we have instead
\[\calR=[a_1,a_3,b_1,w_1,b_2,b_3].\]
In the latter expression, relabelling $\alpha_1=a_3$, $\alpha_3=a_1$, $\beta_1=b_3$, $\beta_2=b_2$, $\beta_3=b_1$, we obtain
\[\calR=[\alpha_1,\alpha_3,\beta_1,\beta_2,w_1,\beta_3].\]
That means we are back to the case $m=3$ and $i=1$, with the $\alpha$'s in place of the $a$'s and the $\beta$'s in place of the $b$'s: thus we can just consider that case.

Similarly to the observation made for the previous case (where $m=4$), $a_3$ already has the maximum possible number of four adjacencies, hence there exists a vertex $c_1$ not on $\calR$ adjacent to $a_1,b_1$. Note that if $a_1b_1$ is drawn internally to $\calR$ while $a_3b_2$ and $a_3b_3$ external, then by planarity $c_1$ is internal to $\calR$. We now introduce the labels
\begin{center}
\begin{tabular}{ccc}
&$c_2=b_1$,&$c_3=b_2$,
\\$c_4=w_1$,&$c_5=b_3$,&$c_6=a_1$.
\end{tabular}
\end{center}

In our sketch for $J$ so far, $[c_1,c_2,\dots,c_6]$ is a region, with vertex degrees so far $2,4,3,2,3,4$ respectively, and any other vertices must be inside this region, as in Figure \ref{fig:381} (this is the starting graph already shown in Figure \ref{fig:start}).

\begin{figure}[ht]
\centering
\begin{subfigure}{0.47\textwidth}
\centering
\includegraphics[width=4.25cm]{381.png}
\caption{Resulting graph after adding $c_1$.}
\label{fig:381}
\end{subfigure}
    \hfill
\begin{subfigure}{0.47\textwidth}
\centering
\includegraphics[width=4.25cm]{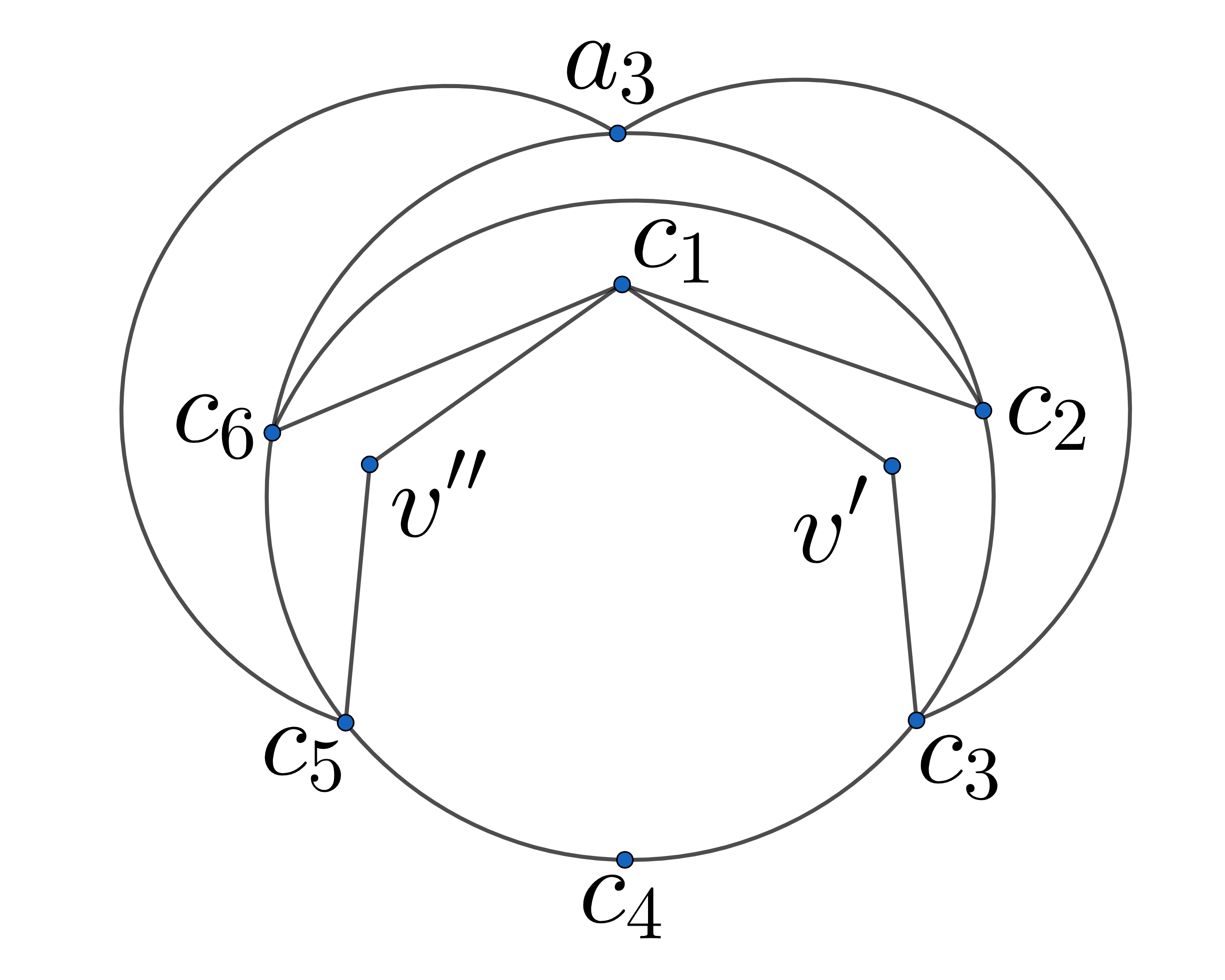}
\caption{The graph $J_0$.}
\label{fig:382}
\end{subfigure}
\caption{Starting the construction of $J$ in Theorem \ref{thm:3333}.}
\label{fig:3801}
\end{figure}

Reasoning as above, the vertex $c_2$ cannot have any further adjacencies. One possibility is that $c_1c_4\in E(J)$ (this is tantamount to applying the transformation in Figure \ref{fig:38f1}). No new vertices may then be introduced without contradicting one of the conditions on $J$: planarity, $3$-connectivity, sequence $4^{p_J-4},3^4$, every region of $J'$ save for $\calR$ is a $4$-gon. The resulting $J$, of order $7$, is the graph in Figure \ref{fig:ds1}.

If instead $c_1c_4\not\in E(J)$, then there exist new vertices $v'$ adjacent to $c_1,c_3$, and $v''$ adjacent to $c_1,c_5$. The resulting graph is shown in Figure \ref{fig:382}, and shall be denoted by $J_0$. Note that $v'\neq v''$, or one of the conditions on $J$ would be violated. 

With the same reasoning, there exist vertices $z_1,z_2,z_3$, such that
\[z_1v',z_1v'',z_2v',z_2c_4,z_3v'',z_3c_4\in E(J).\]
If $z_1=z_2=z_3$ (i.e.~we are implementing the transformation in Figure \ref{fig:38f2} to the starting graph in Figure \ref{fig:381}), then no new vertices may be introduced, and we have obtained a solution with $p_J=10$.

We cannot have $z_1=z_2\neq z_3$ (or permutations) without contradicting the conditions on $J$. This leaves the case where $z_1,z_2,z_3$ are all distinct. Continuing from the graph $J_0$ in Figure \ref{fig:382}, consider
\[\calT(1): J_0\to J_0(1)=J_0+z_1+z_1v'+z_1v''\]
(Figure \ref{fig:383}). This is tantamount to applying the transformation of Figure \ref{fig:38t1} to the starting graph in Figure \ref{fig:381}.

\begin{figure}[ht]
\centering
\begin{subfigure}{0.32\textwidth}
\centering
\includegraphics[width=4.25cm]{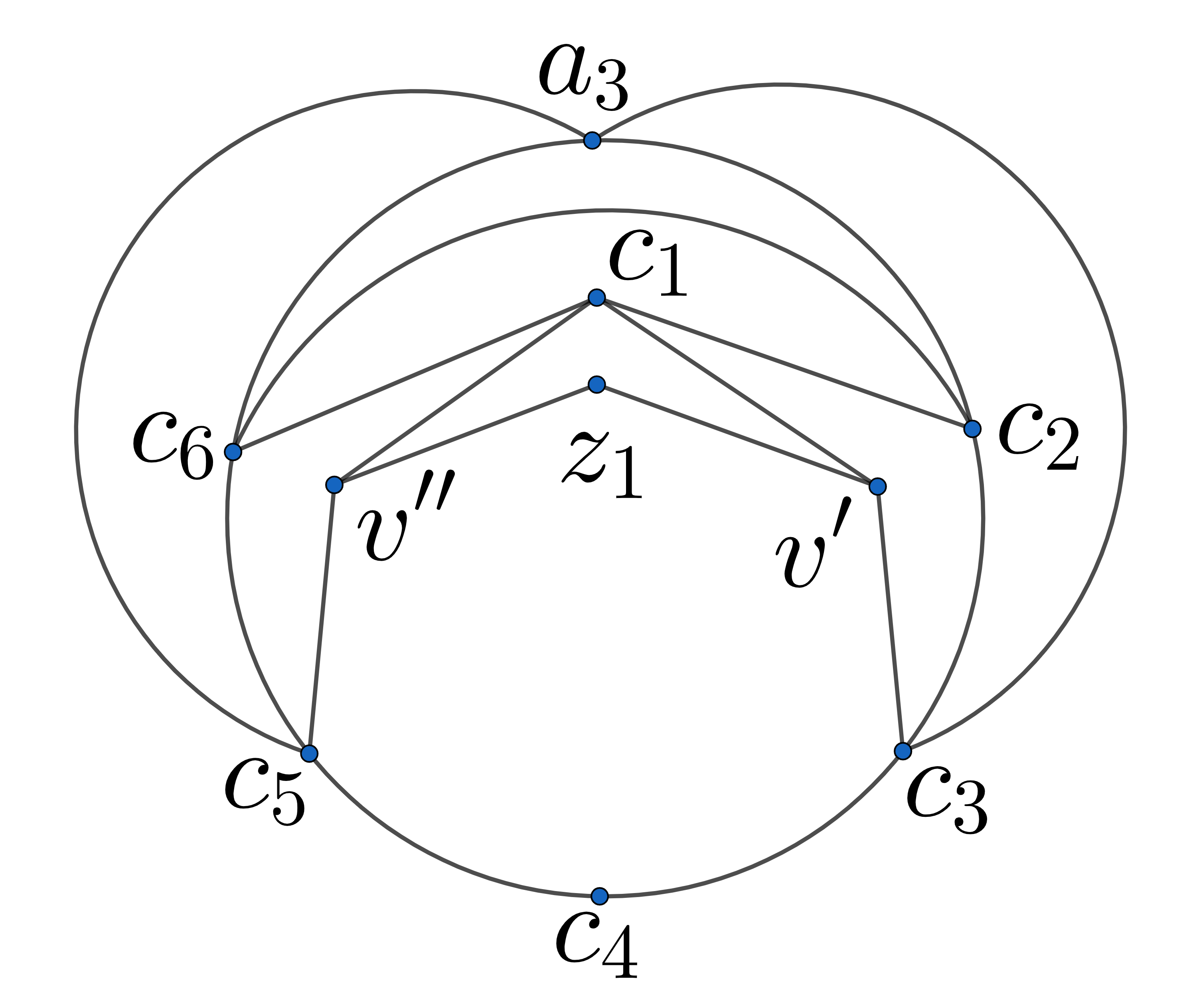}
\caption{Resulting graph after $\calT(1)$.}
\label{fig:383}
\end{subfigure}
    \hfill
\begin{subfigure}{0.32\textwidth}
\centering
\includegraphics[width=4.25cm]{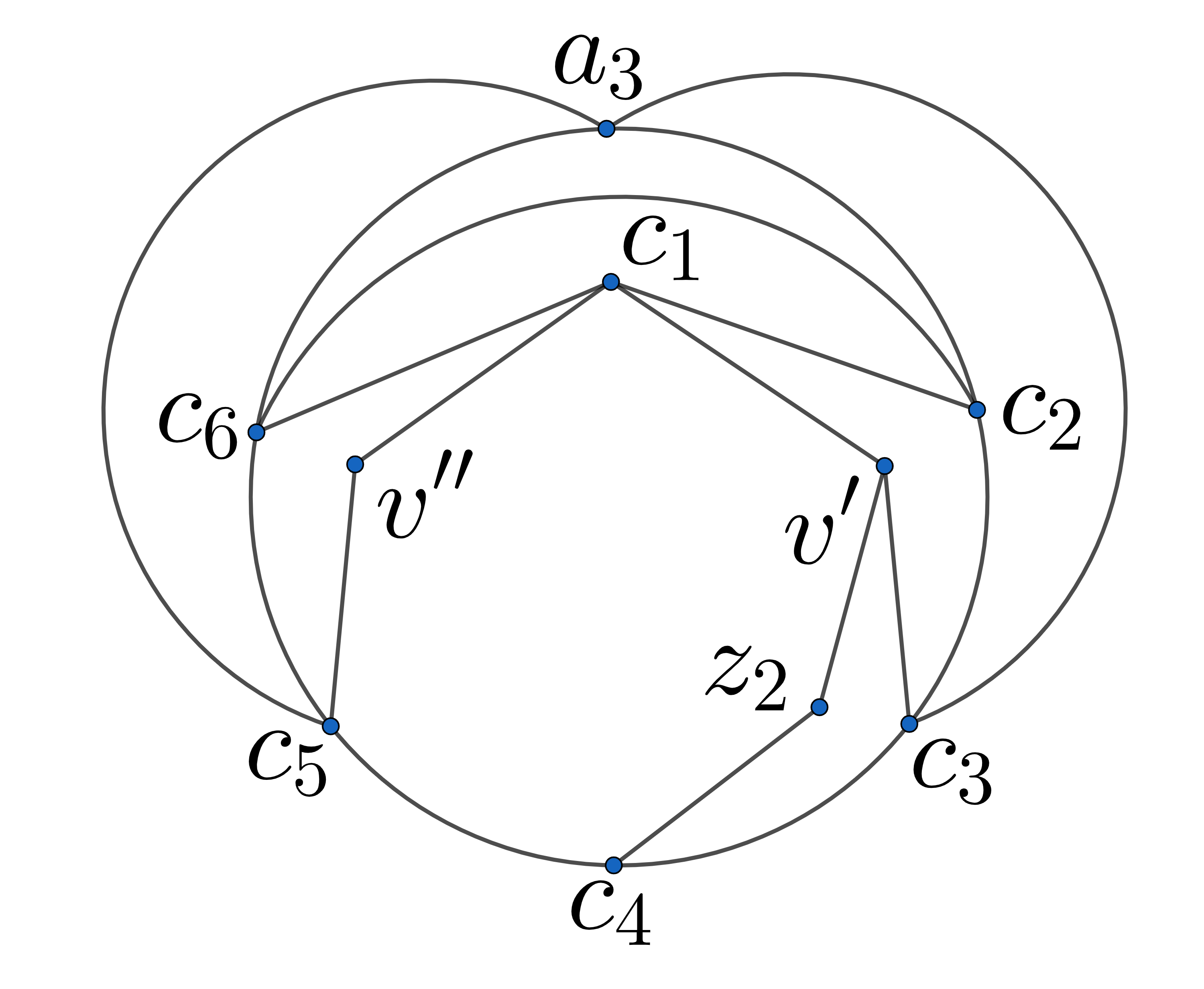}
\caption{Resulting graph after $\calT(2)$.}
\label{fig:384}
\end{subfigure}
    \hfill
\begin{subfigure}{0.32\textwidth}
\centering
\includegraphics[width=4.25cm]{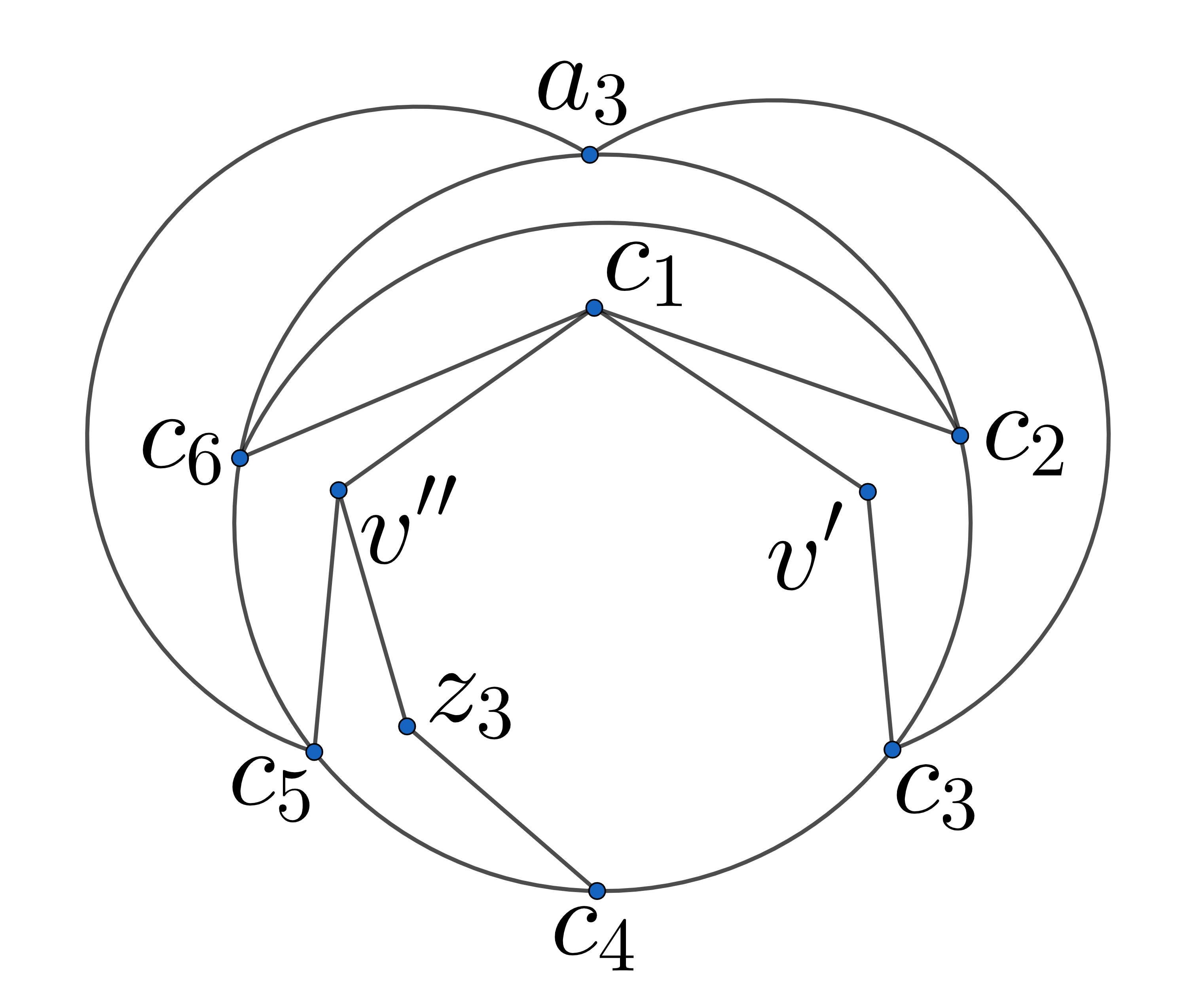}
\caption{Resulting graph after $\calT(3)$.}
\label{fig:385}
\end{subfigure}
\caption{Continuing the construction of $J$ in Theorem \ref{thm:3333}.}
\label{fig:3802}
\end{figure}
We now take the relabelling
\begin{center}
\begin{tabular}{ccc}
$d_1=c_4$,&$d_2=c_5$,&$d_3=v''$,
\\$d_4=z_1$,&$d_5=v'$,&$d_6=c_3$.
\end{tabular}
\end{center}
% From this point on, the above steps can be reiterated, to construct solutions $J$ with $p_J\equiv 1\pmod 3$, $p_J\geq 10$.

Another option is,
\[\calT(2): J_0\to J_0(2)=J_0+z_2+z_2v'+z_2c_4\]
(refer to Figure \ref{fig:384}). This is tantamount to applying the transformation of Figure \ref{fig:38t2} to the starting graph in Figure \ref{fig:381}. Here we take
\begin{center}
\begin{tabular}{ccc}
$d_1=v''$,&$d_2=c_1$,&$d_3=v'$,
\\$d_4=z_2$,&$d_5=c_4$,&$d_6=c_5$.
\end{tabular}
\end{center}
The final option is
\[\calT(3): J_0\to J_0(3)=J_0+z_3+z_3v''+z_3c_4\]
(refer to Figure \ref{fig:385}). This is tantamount to applying the mirror image of the transformation in Figure \ref{fig:38t2} to the starting graph in Figure \ref{fig:381}. Here we take the relabelling
\begin{center}
\begin{tabular}{ccc}
$d_1=v'$,&$d_2=c_3$,&$d_3=c_4$,
\\$d_4=z_3$,&$d_5=v''$,&$d_6=c_1$.
\end{tabular}
\end{center}

In any case, in our sketch so far, $[d_1,d_2,\dots,d_6]$ is a region, with vertex degrees so far $2,4,3,2,3,4$ respectively, and any new vertices are inside this region. Our graph now has $3$ additional vertices.

The above steps may be applied iteratively on the new region $[d_1,d_2,\dots,d_6]$, rather than $[c_1,c_2,\dots,c_6]$. At each step, either we add two new vertices as in Figure \ref{fig:3801}, followed by applying one of $\calT(1)$, $\calT(2)$, $\calT(3)$ (this is tantamount, respectively, to applying the transformations \ref{fig:38t1}, \ref{fig:38t2}, and the mirror image of \ref{fig:38t2}), or we complete the construction in one of the two ways described above (i.e., the transformations in Figure \ref{fig:38f}). This concludes the construction of all solutions $J$. To check that all of these graphs $J$ are in fact solutions, i.e.~that $J\w K_2$ is a polyhedron, we simply note that all of these graphs $J$ satisfy Condition \ref{eq:c3} of Theorem \ref{thm:0123}.

The other statements of the present theorem are now clear from the above construction. For instance, the four triangular faces of $J$ are always
\[[a_1,b_1,a_3], \quad [a_1,b_1,c_1], \quad [a_1,a_3,b_3], \quad [a_3,b_1,b_2].\]
The proof of Theorem \ref{thm:3333} is complete.
\end{proof}

\section{Cubic polyhedral Kronecker products}
%\subsection{$3$-regular polyhedra}
\label{sec:cub}
%As noted in \cite[Proposition 1.2]{mafkpr}, if $\calP=J\wedge K_2$ if a polyhedron, then it has at least $8$ vertices of degree $3$. Hence if $\calP=J\wedge K_2$ is a vertex-regular polyhedron and a Kronecker product, then $\calP$ and $J$ are $3$-vertex-regular.

%\begin{que}
%Which $3$-regular polyhedra are Kronecker products?
%\end{que}

We now turn to vertex-regular polyhedral Kronecker products. As mentioned in the introduction, since polyhedral Kronecker products have at least eight vertices of degree $3$, the vertex regular ones are necessarily cubic graphs. They may be also thought of as the other extreme to the graphs in Theorem \ref{thm:3333}, since here we maximise the number of vertices of degree $3$.

\begin{proof}[Proof of Theorem \ref{thm:cubpl}]
%Since $J\wedge K_2$ is cubic, then by definition so is $J$. A cubic, planar graph is $3$-connected, as it is the dual of a maximal planar graph, hence $J$ is a cubic $3$-polytope.

If $J$ is a $3$-polytope with more than four odd faces, then by Theorem \ref{thm:0123}, we are in the case \ref{eq:c3}, thus more than three odd faces share a vertex. Therefore, this vertex has degree larger than three, contradiction.
On the other hand, we readily construct solutions for each of the feasible types \ref{eq:c0}, \ref{eq:c1}, \ref{eq:c2}, \ref{eq:c3}, as in Figure \ref{fig:3reg}. To see that there are infinitely many solutions for each type, we apply the planar transformation in Figure \ref{fig:tr3r} to a quadrilateral face of any solution $J$. The resulting planar graph is still $3$-regular, the number of odd regions has not changed, and neither has the way they intersect. Thereby, this operation can be performed iteratively, and the resulting graphs still satisfy the respective Conditions \ref{eq:c0}, \ref{eq:c1}, \ref{eq:c2}, \ref{eq:c3} of Theorem \ref{thm:0123}.
\end{proof}
\begin{figure}[ht]
\centering
\begin{subfigure}{0.49\textwidth}
\centering
\includegraphics[width=2.5cm]{t0.png}
\hspace{1cm}
\includegraphics[width=2.75cm]{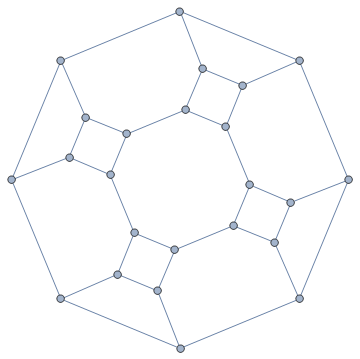}
\caption{A $3$-regular, planar graph of connectivity $2$ (from Figure \ref{fig:sh2c}) satisfying Condition \ref{eq:c0}, and its Kronecker product with $K_2$.}
\label{fig:3r0}
\end{subfigure}
    \hfill
\begin{subfigure}{0.49\textwidth}
\centering
\includegraphics[width=2.5cm]{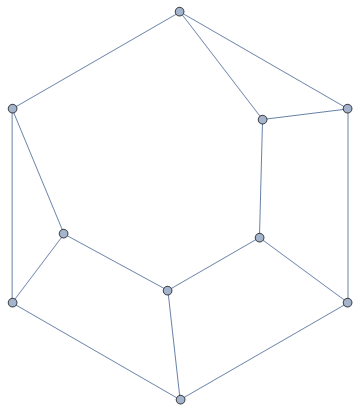}
\hspace{0.5cm}
\includegraphics[width=2.6cm]{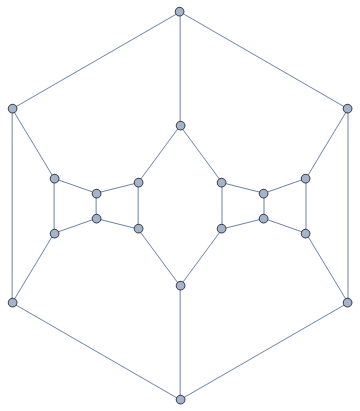}
\caption{A $3$-regular polyhedron satisfying Condition \ref{eq:c1}, and its Kronecker product with $K_2$.}
\label{fig:3r1}
\end{subfigure}
    \hfill
\begin{subfigure}{0.49\textwidth}
\centering
\includegraphics[width=2.5cm]{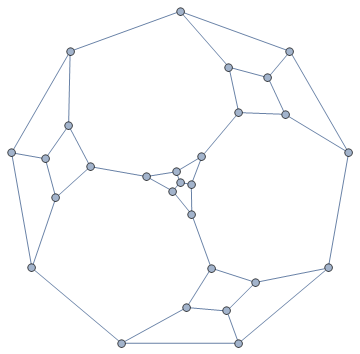}
\hspace{0.5cm}
\includegraphics[width=2.5cm]{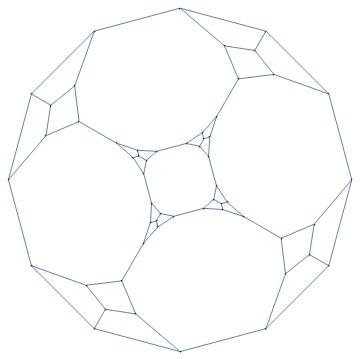}
\caption{A $3$-regular polyhedron satisfying Condition \ref{eq:c2}, and its Kronecker product with $K_2$.}
\label{fig:3r2}
\end{subfigure}
    \hfill
\begin{subfigure}{0.49\textwidth}
\centering
\includegraphics[width=2.5cm]{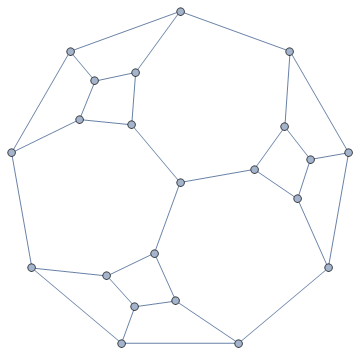}
\hspace{0.5cm}
\includegraphics[width=2.5cm]{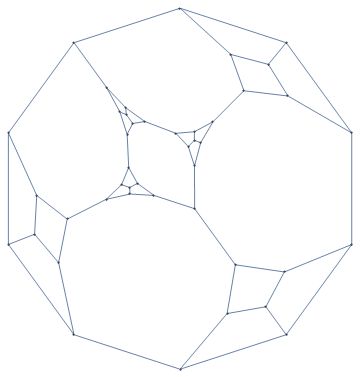}
\caption{A $3$-regular polyhedron satisfying Condition \ref{eq:c3}, with four odd faces, and its Kronecker product with $K_2$.}
\label{fig:3r3}
\end{subfigure}
\caption{Examples of $3$-regular, planar graphs $J$ such that $\calP=J\wedge K_2$ is a polyhedron.}
\label{fig:3reg}
\end{figure}

\begin{figure}[ht]
	\centering
	%\begin{subfigure}{0.49\textwidth}
		%\centering
		\includegraphics[width=2.75cm]{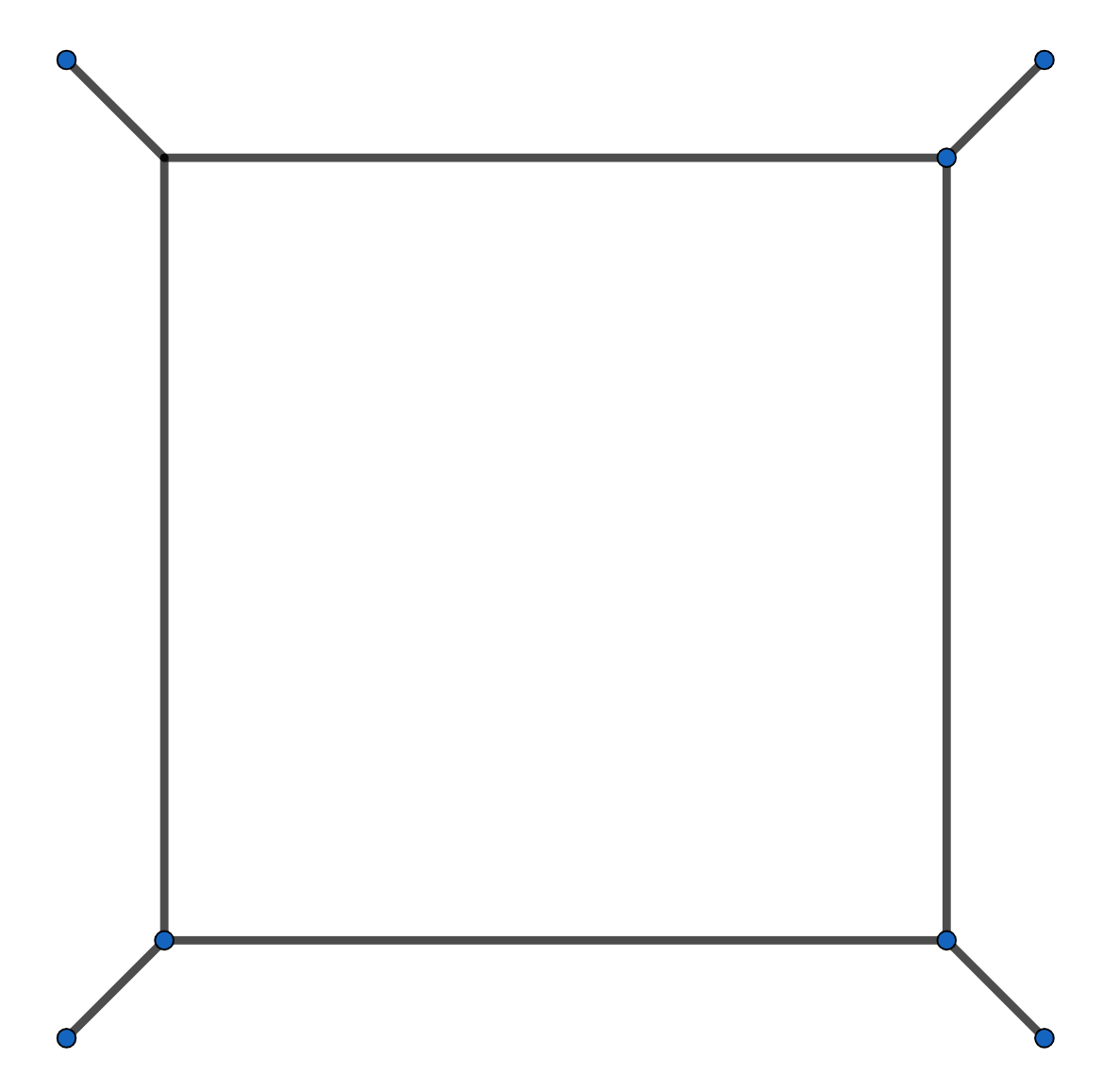}
  \hspace{1cm}
		%\caption{Example of $H$ as in Theorem \ref{thm:2}.}
		%\label{fig:tr3r}
	%\end{subfigure}
	%\begin{subfigure}{0.49\textwidth}
		%\centering
  \includegraphics[width=2.75cm]{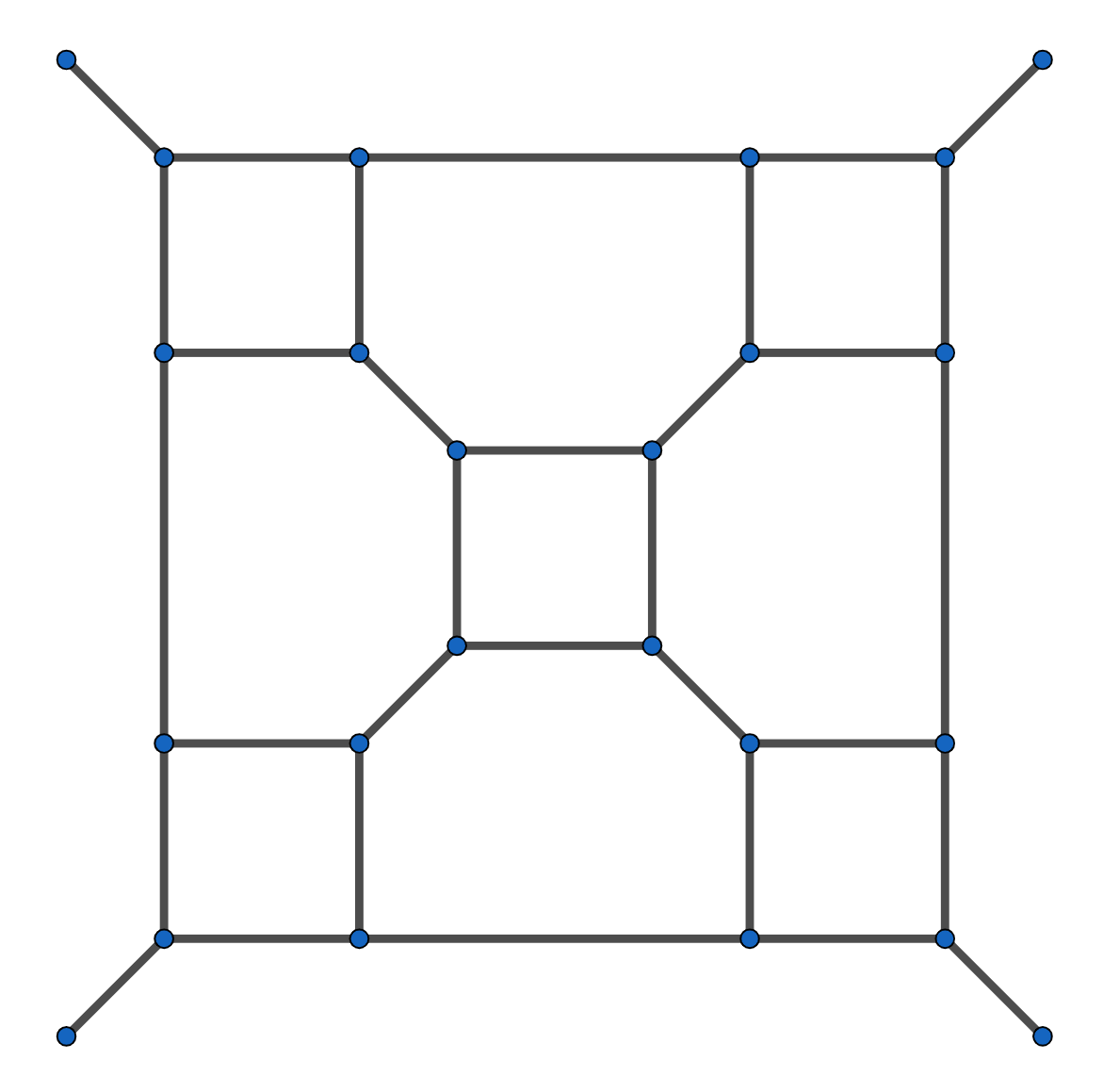}
		%\caption{The $3$-polytope $H\wedge K_2$.}
		%\label{fig:p0}
	%\end{subfigure}
	\caption{In a $3$-regular, planar graph containing a quadrilateral region (left), we transform this region into the graph on the right. The resulting planar graph is still $3$-regular, and the number of odd regions has not changed, and neither has the way they intersect.}
	\label{fig:tr3r}
\end{figure}

\begin{comment}
\begin{thm} PROBLEM
%Let $J$ be non-planar. 
The graph $\calP=J\wedge K_2$ is a $3$-regular polyhedron if and only $J$ may be obtained from a plane, %bipartite,
semi-hyper-$2$-connected, $3$-regular graph $J''$ in the following way. Let $\calR$ be the region of $J''$ containing all $2$-cuts (if there are none, then we may take any region). We subdivide certain edges of $\calR$ (one or more times each) to introduce {\em distinct} new vertices
\[a_1,a_2,\dots,a_m,b_1,b_2,\dots,b_m, \quad \text{even }m\geq 2\]
in this order around $\calR$, obtaining a %(still bipartite)
graph $J'$, such that %$a_1,a_m,b_1,b_m$ are distinct;
for every $1\leq i\leq m$, the graph $J'+a_ib_i$ is not bipartite; finally $J$ is obtained by adding the edges $a_1b_1$, $a_2b_2$, $\dots$, $a_mb_m$ to $J'$.
\end{thm}
\end{comment}

\begin{proof}[Proof of Theorem \ref{thm:cub}]
Let $\calP=J\wedge K_2$ be a cubic polyhedron. We inspect the construction of $\calP$ in Theorem \ref{thm:ab}. Since $J$ is cubic, then in particular the vertices $a_1,a_2,\dots,a_m,b_1,b_2,\dots,b_m$ of Definition \ref{cond1} all have degree $3$ in $J$. Since they have degree at least $2$ in the $2$-connected graph $J'$, it follows that to pass from $J$ to $J'$, for each $a_1,a_2,\dots,a_m,b_1,b_2,\dots,b_m$ we are deleting exactly one incident edge. It follows that these vertices are all distinct. Each has degree exactly $2$ in $J'$.

If we contract all of these vertices in $J'$, we obtain a cubic, planar graph $J''$. Since $J'$ is either $3$-connected or semi-hyper-$2$-connected, then so is $J''$. Note that $J''$ may be a multigraph, in case all vertices of a region of $J'$ save two are among $a_1,a_2,\dots,a_m,b_1,b_2,\dots,b_m$ (recall Figures \ref{fig:J2} and \ref{fig:J1}).

Call $\calR''$ the region of $J''$ that contains all the contracted edges. Since $J'$ is bipartite, and we have removed $2m$ vertices, the region $\calR''$ is even. If $J''$ contains any odd regions, then these were formed by removing an odd number of $a_1,a_2,\dots,a_m,b_1,b_2,\dots,b_m$, because $J'$ is bipartite. Further, each odd region of $J''$ is adjacent to $\calR''$.

Next, assume by contradiction that contracting the vertices results in only one contracted edge in $J''$. Then in $J'$ each of $a_1,a_2,\dots,a_m,b_1,b_2,\dots,b_m$ belongs to $\calR$ and to another region of $J'$, say $\calS$. Since $\calP$ is not a multigraph, then for all $1\leq i\leq m$, $a_ib_i\not\in E(J')$. In particular, $a_1b_1\not\in E(J')$, thus w.l.o.g.~$a_1,a_2,\dots,a_m,b_1,b_2,\dots,b_m$ are consecutive and in this order around $\calR$ and $\calS$. We may write
\[\calS=[\dots,v,a_1,a_2,\dots,a_m,b_1,b_2,\dots,b_m,w,\dots].
\]
Then $\{v,w\}$ is a $2$-cut in $J'$ lying on $\calR$ between $b_m$ and $a_1$, contradicting Definition \ref{cond1}.

Taking $J_2:=J''$, the conditions on $J_2$ are all satisfied. Note that $\calR_2=\calR''$ and $J_1=J'$.

On the other hand, starting with $J_2$ satisfying all the conditions in the statement of this theorem, one constructs $J_1$ as described in the statement. Then $J'=J_1$ satisfies all the conditions in Definition \ref{cond1}, hence $\calP$ is a polyhedron by Theorem \ref{thm:ab}. It is $3$-regular by construction.
\end{proof}

\section{Cancellation and simultaneous products}
\subsection[Simultaneous Kronecker product]{$J\w K_2\simeq L\w K_2$}
\label{sec:jl}
We now turn to graphs that may be expressed as Kronecker and/or Cartesian products in more than one way. Our main result shows that a polyhedron is a Kronecker product in at most one way. In other words, Kronecker product cancellation always holds when the product is planar and $3$-connected.

\begin{proof}[Proof of Theorem \ref{thm:JL}]
According to Definition \ref{cond1} and Theorem \ref{thm:ab}, $E(J)$ is partitioned into $E(J')$ and
\[\{a_1b_1,a_2b_2,\dots,a_mb_m\}.\]
The product $J\w K_2$ is obtained by adding the edges
\begin{equation}
\label{eq:edgesab}
(a_i,x)(b_i,y), \quad (a_i,y)(b_i,x), \quad 1\leq i\leq m,
\end{equation}
to two copies of $J'$ from Definition \ref{cond1}.

Since $J\w K_2\simeq L\w K_2$, the product is also obtained by adding certain edges
\begin{equation}
\label{eq:edgescd}
(c_i,x)(d_i,y), \quad (c_i,y)(d_i,x), \quad 1\leq i\leq n \quad (n\geq 2),
\end{equation}
to two copies of a graph $L'$ satisfying conditions analogous to $J'$ from Definition \ref{cond1}. In particular, $L'$ contains a region $\calS$ satisfying conditions analogous to $\calR$ from Definition \ref{cond1}, thus containing the vertices $c_1,c_2,\dots,c_n,d_1,d_2,\dots,d_n$. The product $J\w K_2$ thus contains two copies of the cycle $\calS$, and its vertices are partitioned in two halves: each half is inside (or on the boundary) of each copy of $\calS$.

We refer to Figure \ref{fig:JL}, drawn starting from two copies of $J'$. The only connections between the two copies of $J'$ are the edges \eqref{eq:edgesab}.
Assume by contradiction that $J\not\simeq L$: it surely means that $\calS$ does not coincide with $\calR$ thus there are non-trivial intersections between each $\calS$ and each $\calR$. Therefore, the boundary of each copy of $\calS$, since it is a cycle, contains a positive even number of edges from \eqref{eq:edgesab} from one copy of $J'$ to the other one. Let's say that one copy of $\calS$ contains the edge $(a_i,x)(b_i,y)$ for some $i$. This is illustrated in Figure \ref{fig:JL} for $i=1$.

Recall that by Theorem \ref{thm:ab}, the product $L\w K_2$ is obtained from two copies of $L'$ by adding the edges \eqref{eq:edgescd}. Therefore %$(a_i,x)$ (and of course likewise $(b_i,y)$) in the first copy of $\calS$ is adjacent to corresponding vertices of \eqref{eq:edgesab} in the other copy of $\calS$.
$(a_i,x)$, which coincides with $(c_j,x)$ for some $j$ w.r.t. $\calS$, is adjacent to some $(d_\alpha,y)$. Similarly, $(b_i,y)$ is adjacent to some $(d_\beta,x)$.

Further, we claim that $(d_\alpha,y)$
belongs to the same copy of $\calR$ as $(a_i,x)$. Otherwise if $(d_\alpha,y)$ were inside the same copy of $J'$ as $(a_i,x)$ or belonged to the other copy of $J'$ (interior or boundary), then the boundary of the copy of $\calS$ containing $(d_\alpha,y)$ should cross the boundary of one of the two copies of $\calR$, impossible by planarity. Likewise, $(d_\beta,x)$ belongs to the same copy of $\calR$ as $(b_i,y)$.

Since $(d_\alpha,y)$ and $(d_\beta,x)$ belong to the same copy of $\calS$, then by planarity these two vertices are in fact adjacent. 
%The other two vertices of the two single edges are in turn connected because they are on the boundary of a copy of $\calS$ and there are single edges connecting one copy of $\calR$ with the other.
%So, each copy of $\calS$ contains an edge from \eqref{eq:edgesab} with consecutive indices (modulo $m$), w.l.o.g.
%\[(a_i,x)(b_i,y), (a_{i+1},y)(b_{i+1},x)\]
%Moreover, the edges $(a_i,x)(a_{i+1},y)$ and $(b_i,y)(b_{i+1},x)$ are from \eqref{eq:edgescd} since they connect the two copies of $L'$ in the product,
Therefore we may set w.l.o.g.
\[c_j=a_i, \quad c_{j+1}=b_i, \quad d_j=a_{i+1}, \quad d_{j+1}=b_{i+1},\]
so that $J\w K_2$ contains the facial $4$-cycle
\[[a_ix,b_iy,b_{i+1}x,a_{i+1}y]=[c_jx,c_{j+1}y,d_{j+1}x,d_jy].\]
%but also, \textit{on the other side of the $J'$s}, the corresponding $4$-cycle
By definition of Kronecker product, we also have the facial $4$-cycle
\[[a_{i+1}x,b_{i+1}y,b_ix,a_iy]=[d_jx,d_{j+1}y,c_{j+1}x,c_jy].\]
In Figure \ref{fig:JL} the situation is illustrated for $i=j=1$.
%complying with the fact that there is an even number of edges from \eqref{eq:edgesab} connecting the two copies of $J'$ for each copy of $\calS$.

Please note that the above reasoning holds for any number of pairs of edges from \eqref{eq:edgesab} that a copy of $\calS$ may contain: more complex examples are shown in Figure \ref{fig:JL_color}.

\begin{figure}[ht]
\centering
\includegraphics[width=10cm]{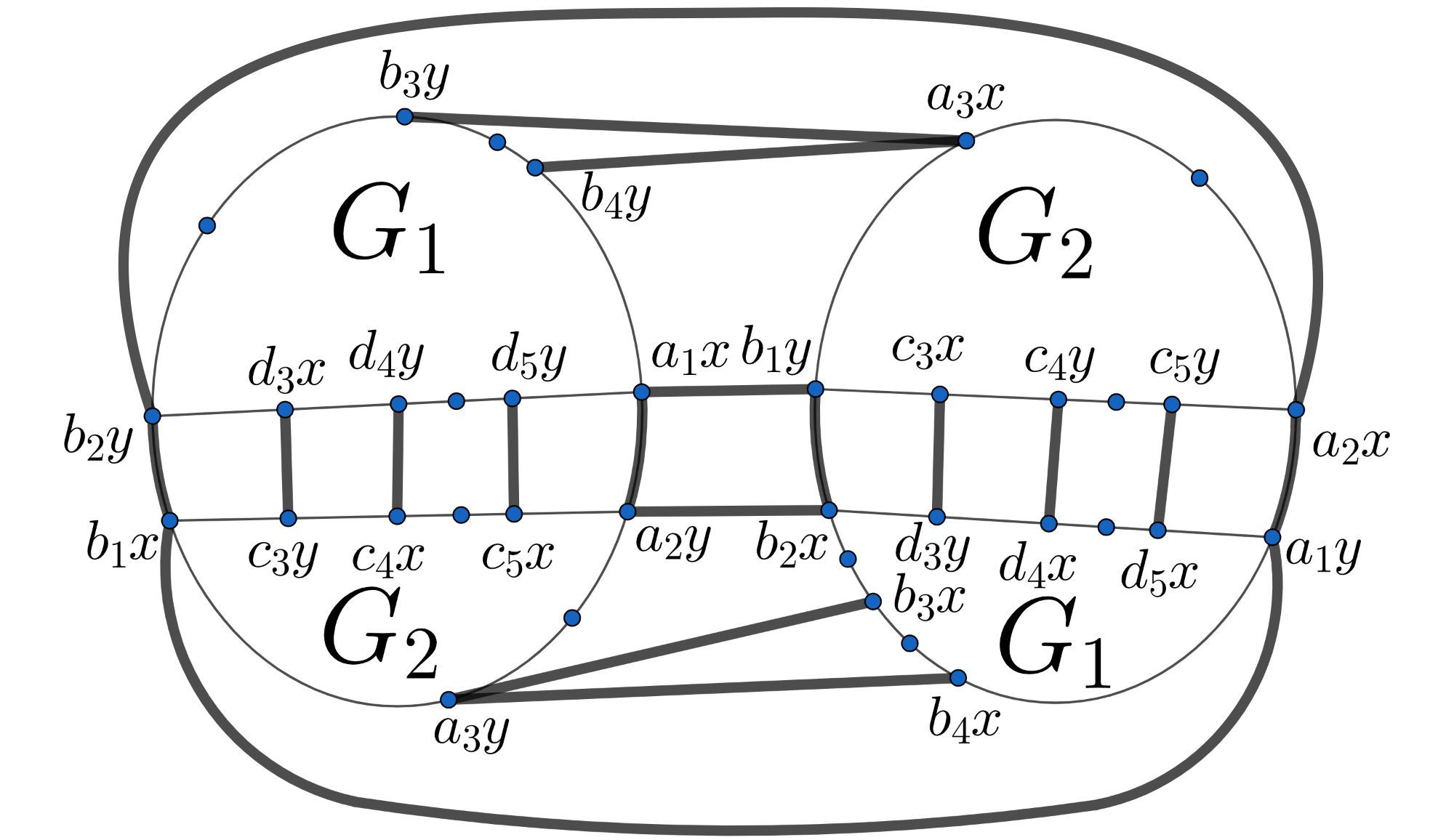}
\caption{Example sketch of a $3$-polytope $J\w K_2\simeq L\w K_2$. Here $i=j=1$, $m=4$, $n=5$, and $a_3=a_4$. Only a subgraph is drawn. Thicker lines are necessarily edges.}
\label{fig:JL}
\end{figure}

\begin{figure}[ht]
\centering
\includegraphics[width=15cm]{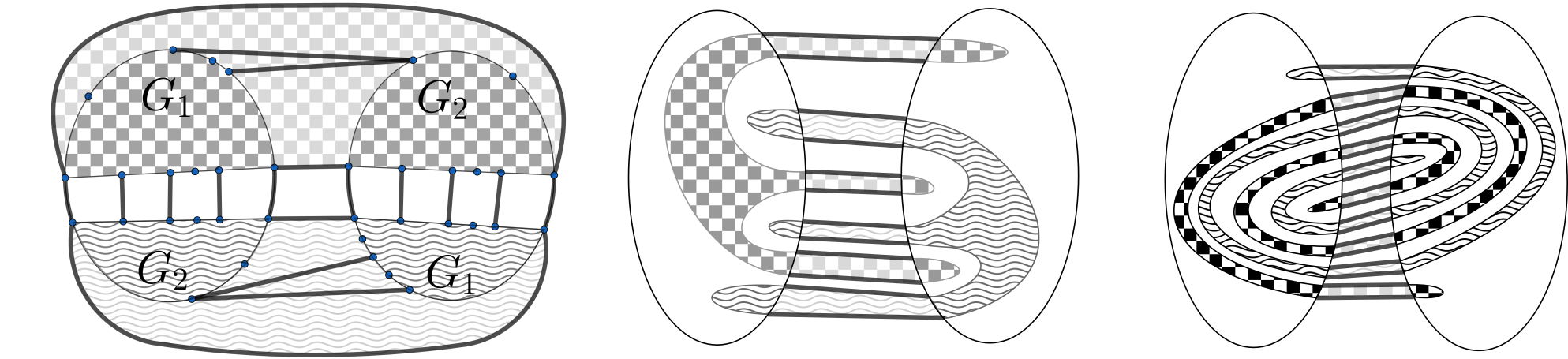}
\caption{Examples of sketches for a few $3$-polytopes $J\w K_2\simeq L\w K_2$ with a different number of connecting $4$-cycles. The first graph on the left corresponds to the one of Figure \ref{fig:JL}. The darker painted regions identify the copies of $G_1$ and $G_2$, i.e.~the parts of $J\w K_2\simeq L\w K_2$ that are inside a copy of $\calR$ and a copy of $\calS$. In each example, the chequered part on the left is one copy of $G_1$, the wavy part on the right is the other copy, the chequered part on the right is one copy of $G_2$, and the wavy part on the left is the other copy. The white quandrangles in the centre (outside of $\calR$ and $\calS$) are faces.}
\label{fig:JL_color}
\end{figure}

\begin{comment}
As a consequence, we have the following construction for $J\w K_2\simeq L\w K_2$: two copies of a certain (planar) graph $G_1$, with external region containing, among other vertices,
%\[a_1=c_1,d_n,d_{n-1},\dots,d_3,b_2=d_2,b_3,\dots,b_m\]
\[a_i=c_j,c_{j-1},\dots,d_{j+2},b_{i+1}=d_{j+1},b_{i+2},\dots,a_{i-1}\]
in this order (where we have omitted the $x$'s and $y$'s for brevity, and if $a_ix$ belongs to one copy of $G_1$, then $a_iy$ belongs to the other copy, and so forth); two copies of a graph $G_2$, with external region containing, among other vertices,
%\[a_2=d_1,c_n,c_{n-1},\dots,c_3,b_1=c_2,a_m,a_{m-1},\dots,a_3\]
\[a_{i+1}=d_j,d_{j-1},\dots,c_{j+2},b_i=c_{j+1},b_{i-1},\dots,a_{i+2}\]
in this order (considering $c_n$ in place of $d_{j-1}$ if $j=1$, and analogously the other vertices); the edges
\[(a_k,x)(b_k,y), \quad (a_k,y)(b_k,x), \quad 1\leq k\leq m, \ k\neq i,i+1,\]
the edges
\[(c_k,x)(d_k,y), \quad (c_k,y)(d_k,x), \quad 1\leq k\leq n, \ k\neq j,j+1,\]
and the two facial $4$-cycles
\[[a_ix,a_{i+1}y,b_{i+1}x,b_iy]=[c_jx,d_jy,d_{j+1}x,c_{j+1}y]\]
and
\[[a_iy,a_{i+1}x,b_{i+1}y,b_ix]=[c_jy,d_jx,d_{j+1}y,c_{j+1}x].\]
\end{comment}

As a consequence, we have the following construction for $J\w K_2\simeq L\w K_2$: two copies of a certain planar graph $G_1$, two copies of a planar graph $G_2$, the edges
\[(a_k,x)(b_k,y), \quad (a_k,y)(b_k,x), \quad 1\leq k\leq m,\]
and the edges
\[(c_k,x)(d_k,y), \quad (c_k,y)(d_k,x), \quad 1\leq k\leq n.\]
The $(a_k,x)(b_k,y)$'s connect the first copy of $G_1$ to the first copy of $G_2$, and the second to the second, while the $(c_k,x)(d_k,y)$'s connect the first copy of $G_1$ to the second copy of $G_2$, and the second copy of $G_1$ to the first copy of $G_2$.

The graph $J$ then admits the following construction: a copy of $G_1$ (top left in Figure \ref{fig:JL}; in the labelling for $J$, we leave out the $x$'s and $y$'s of course), a copy of $G_2$ (bottom left), the edges
\[c_kd_k, \quad 1\leq k\leq n,\]
and the edges
\[a_kb_k, \quad 1\leq k\leq m.\]

On the other hand, $L$ admits the following construction: a copy of $G_1$ (top left in Figure \ref{fig:JL}),
a copy of $G_2$ (top right), 
the edges
\[a_kb_k, \quad 1\leq k\leq m,\]
and the edges
\[c_kd_k, \quad 1\leq k\leq n.\]
We have obtained $J\simeq L$. Note that $G_1$ and $G_2$ may have more than one connected component as in Figure \ref{fig:JL_color}, centre and right.
%In other words, both $J,L$ admit the following construction: a copy of $G_1$,
%a copy of $G_2$, the edges
%\[a_ib_i, \quad 1\leq i\leq m,\]
%and the edges
%\[c_id_i, \quad 1\leq i\leq n.\]
\end{proof}

\subsection[Simultaneous Cartesian product]{$C_n\q P_m\simeq H\q K_2$}
\label{sec:cc}
\paragraph{Strategy to prove Theorem \ref{thm:dou}.} The rest of this paper is dedicated to proving Theorem \ref{thm:dou}. As mentioned in the introduction, there are two classes of polyhedral Cartesian products, namely the stacked prisms $C_n\q P_m$ with $n\geq 3$ and $m\geq 2$, and the graphs $H\q K_2$ with $H$ outerplanar and Hamiltonian \cite[Proposition 1.9]{mafkpr}. Moreover, for non-empty graphs $A,B,C$, the isomorphism $A\q C\simeq B\q C$ implies $A\simeq B$ \cite[Theorem 6.21]{haimkl}. Therefore, a polyhedral graph is expressible as Cartesian product in at most two distinct ways, and to characterise those expressible in two ways, it suffices to find integers $n,m$ and graphs $H$ such that $C_n\q P_m\simeq H\q K_2$. We will do so in the present section.
\\
To find the polyhedra that are both Cartesian and Kronecker products, we will analyse $C_n\q P_m\simeq J\w K_2$ in Section \ref{sec:ck1} and $H\q K_2\simeq J\w K_2$ in Section \ref{sec:ck2}.

\begin{prop}
\label{le:CC}
Let $n\geq 3$ and $m\geq 1$. There exists a graph $H$ satisfying
\[C_n\q P_m\simeq H\q K_2\]
if and only if either $H\simeq C_n$, $n\geq 3$, and $m=2$ (the products are prisms), or $n=4$ and $H\simeq F_{2m}$, $m\geq 1$ is the ladder graph (the products are stacked cubes).
%\\
%The only planar graphs expressible as Cartesian products in two distinct ways are the stacked cubes for $m\neq 2$.
\end{prop}
\begin{proof}
%Note that $C_n\q P_m$ is planar for every $n,m\geq 1$, hence $H\q K_2$ is planar, thus $H$ is planar. We quickly see that
%\[C_{2n}\q P_1\simeq C_n\q K_2,\]
%while if $m=1$ and $n$ is odd, then $C_n\q P_m$ is of odd order, while $H\q K_2$ is of even order for all graphs $H$.
Let $m=1$. %, since $H\q K_2$ has an even number of vertices, $n$ must be even too. Moreover,
For each graph $H$ with at least one edge, its Cartesian product with $K_2$ gives birth to a graph containing a copy of $C_4$. Since the only cycle graph containing a copy of $C_4$ is $C_4$ itself, $n$ must be 4. $H$ can then be $K_2$. If $m=2$, we trivially take $H=C_n$ and the product is the $n$-gonal prism.

Now let $m\geq 3$. Here $C_n\q P_m$ is a $3$-polytope, so that $C_n\q P_m\simeq H\q K_2$ implies that $H\q K_2$ is a $3$-polytope, hence $H$ is outerplanar and Hamiltonian, i.e., $H$ is obtained from the polygon $C_\ell$ by adding some diagonals \cite[Proposition 1.9]{mafkpr}. Equivalently, $H\q K_2$ may be constructed starting from a prism of edges
\begin{equation}\label{eq:Mprism}\begin{aligned}
    \{&v_1v_2,v_2v_3,\dots,v_\ell v_1,\\
&w_1w_2,w_2w_3,\dots,w_\ell w_1,\\
&v_1w_1,v_2w_2,\dots,v_\ell w_\ell \}
\end{aligned}
\end{equation}
by adding certain corresponding diagonals $v_iv_j$ and $w_iw_j$ to the two bases.

\begin{figure}[ht]
\centering
\begin{subfigure}{0.64\textwidth}
\centering
\includegraphics[width=3.5cm]{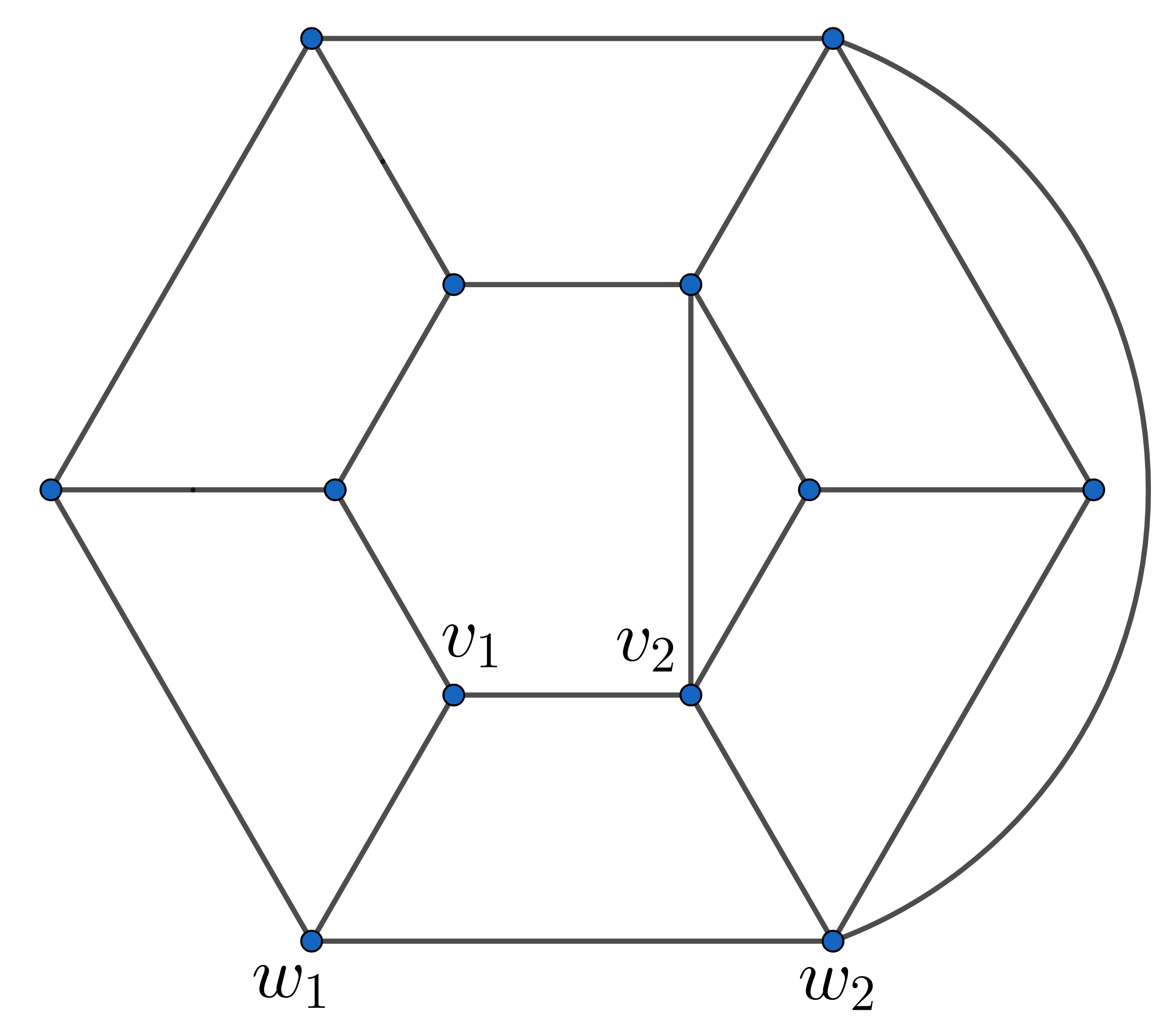}
\hspace{0.75cm}
\includegraphics[width=3.5cm]{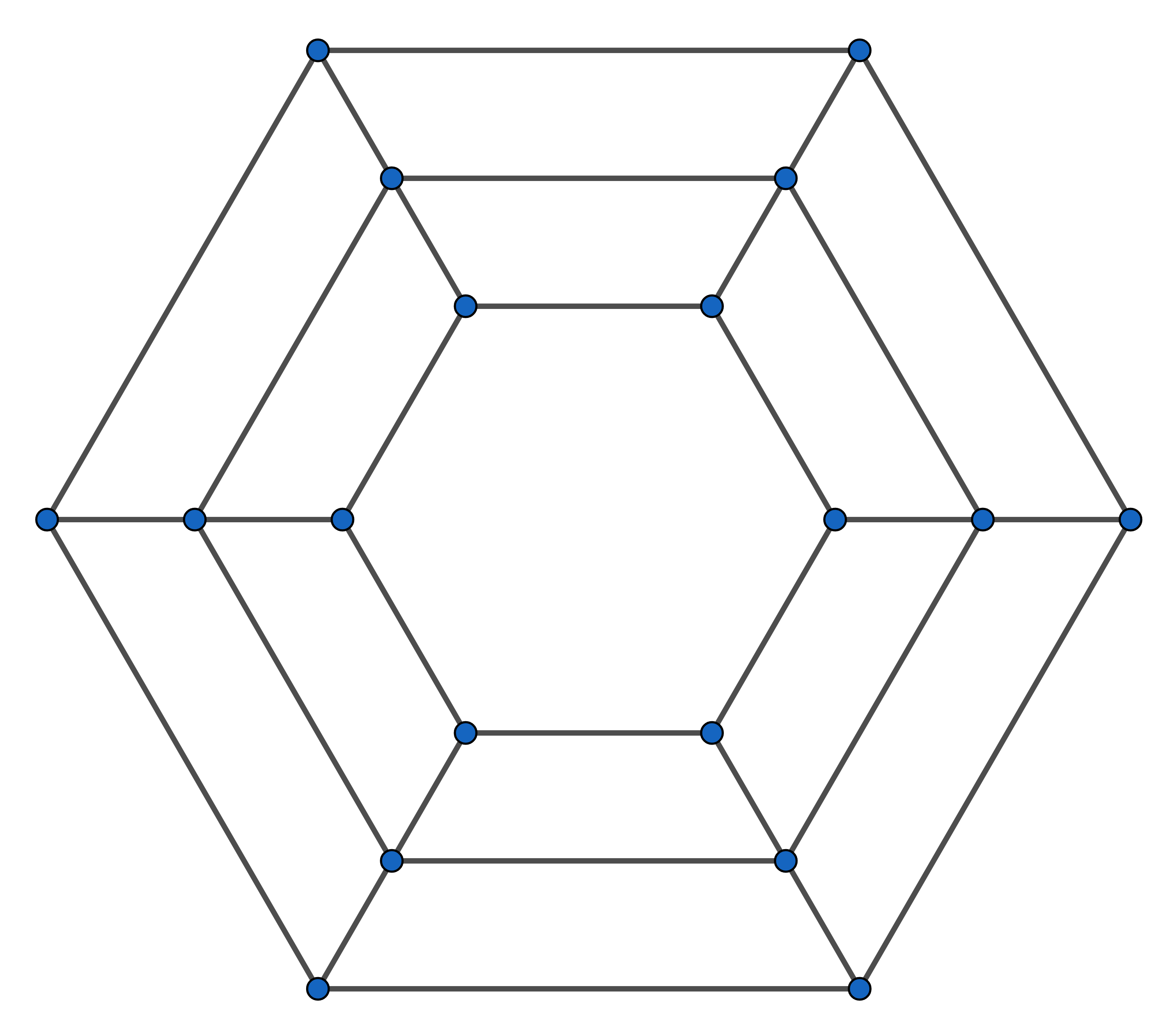}
\caption{In the polyhedron on the left, of the form $H\q K_2$, the quadrangular face
$[v_1,v_{2},w_{2},w_1]$
is adjacent to two quadrangles and two $5$-gons. In $C_6\q P_3$ (right), each quadrangular face is adjacent to three quadrangles.}
\label{fig:ccone}
\end{subfigure}
    \hfill
\begin{subfigure}{0.32\textwidth}
\centering
\includegraphics[width=4cm]{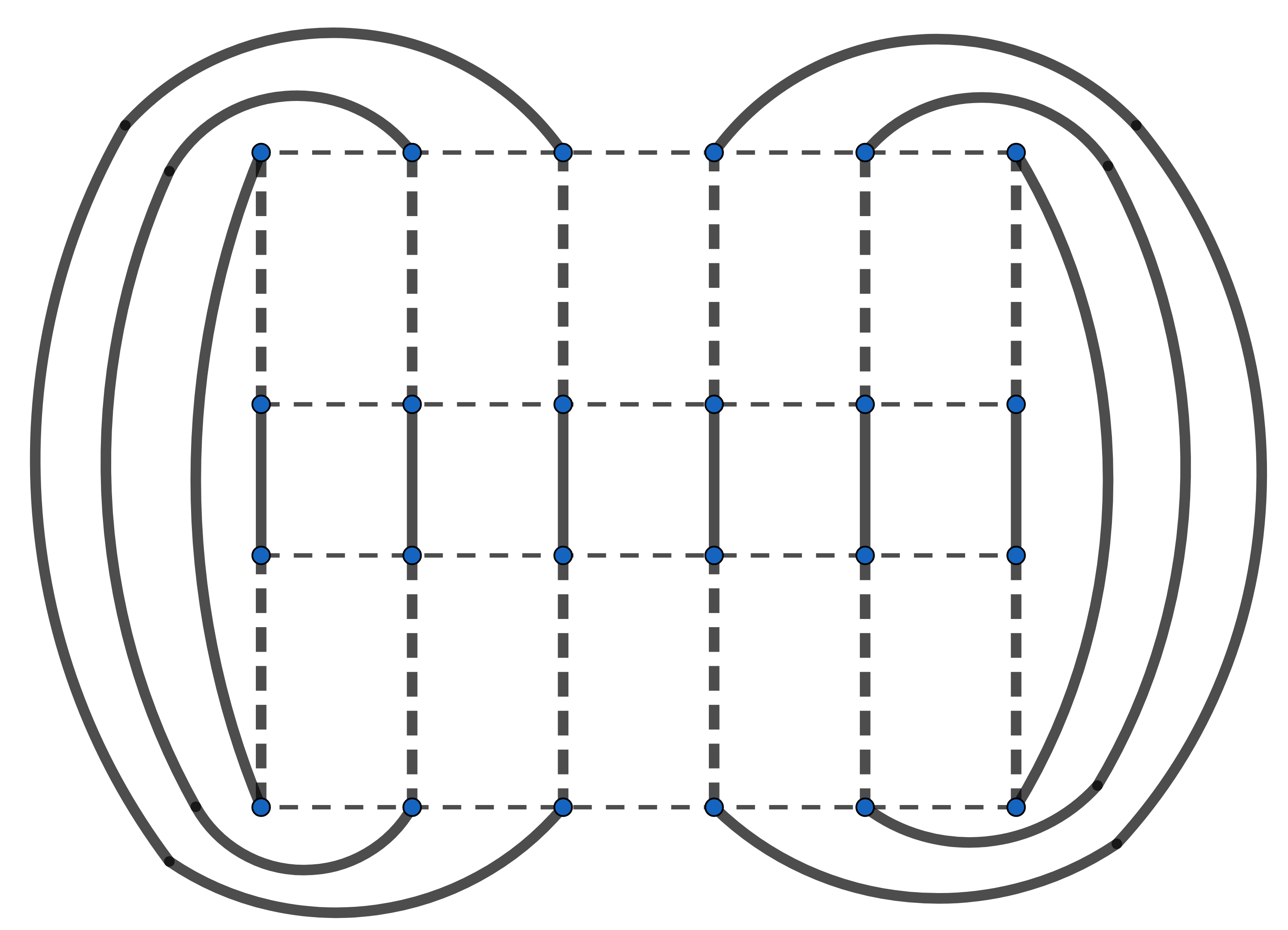}
\caption{$C_4\q P_6\simeq F_{12}\q K_2$. The dashed lines are the copies of $F_{12}$. The bold lines (whether dashed or not) are the bases of the stacked cubes.}
\label{fig:cctwo}
\end{subfigure}
\caption{Proposition \ref{le:CC}.}
\label{fig:cc}
\end{figure}

Therefore, suppose $H\q K_2$ has an $h$-gonal face with $h\neq 4$, containing the edge $v_iv_{i+1}$, say; then the quadrangular face
\[[v_i,v_{i+1},w_{i+1},w_i]\]
is adjacent to two quadrangles and two $h$-gons (one on each copy of $H$ in its Cartesian product with $K_2$), as in Figure \ref{fig:ccone}, left. But on the other hand, since $m\geq 3$, each quadrangular face of $C_n\q P_m$ is adjacent to at least three other quadrangular faces, as in Figure \ref{fig:ccone}, right. It follows that, if $C_n\q P_m\simeq H\q K_2$, then %our quadrangular face must be adjacent to at least three other quadrangular faces, so
every face of $H\q K_2$ is quadrangular, i.e.~$n=4$ (stacked cubes). By considering the ladder graph
\[H=F_{2m},\]
we indeed obtain $C_4\q P_m\simeq F_{2m}\q K_2$ for every $m\geq 1$, as in Figure \ref{fig:cctwo}.

%To prove the last statement of this proposition, we simply combine the above with \cite{behmah}.

%$m$ is even, and $m\geq 4$. We have
%\[C_4\q P_m\simeq H\q K_2, \quad m\geq 2,\]
%where $H$ is obtained from $C_{2m}$ by adding the edges
%\[\{v_2v_{2m-1},v_3v_{2m-2},\dots,v_{m-1}v_{m+2},w_2w_{2m-1},w_3w_{2m-2},\dots,w_{m-1}w_{m+2}\}\]
%to \eqref{eq:Mprism}. In other words,
%In $C_n\q P_m$, each vertex is of degree $3$ or $4$, and if $3$, then it lies on an $n$-gonal face.
\end{proof}

\subsection[Simultaneous Cartesian and Kronecker product I]{$C_n\q P_m\simeq J\w K_2$}
\label{sec:ck1}
\begin{prop}
\label{prop:stacked}
The polyhedron $C_n\q P_m$, $n\geq 3$, $m\geq 2$ is a Kronecker product if and only if either $n\equiv 2 \pmod 4$, or $n\equiv 0 \pmod 4$ and $m\equiv 0 \pmod 2$.
\end{prop}
\begin{proof}
For the case $n\equiv 2 \pmod 4$ we have
\[C_{4N+2}\q P_M\simeq (C_{2N+1}\q P_M)\w K_2, \qquad N,M\geq 1,\]
as in Figure \ref{fig:qqone}.

\begin{figure}[ht]
\centering
\begin{subfigure}{0.48\textwidth}
\centering
\includegraphics[width=6.5cm]{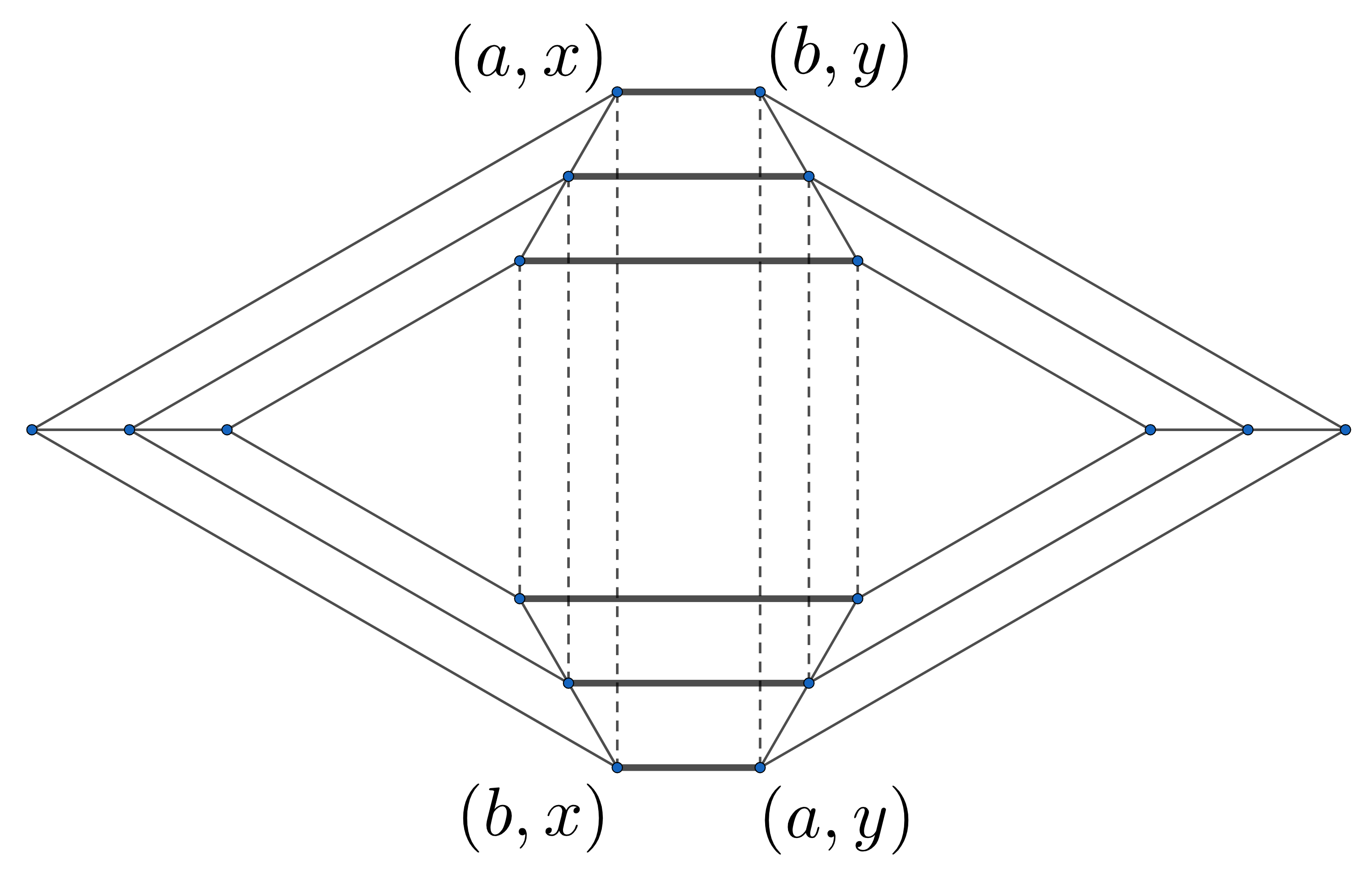}
\caption{Here $N=1$, $n=4N+2=6$, and $M=m=3$. Referring to the construction in Theorem \ref{thm:ab}, we start with two copies of $C_{3}\q P_3$, then we delete the dashed edges, and add the thick ones, obtaining $(C_{3}\q P_3)\w K_2\simeq C_{6}\q P_3$.}
\label{fig:qqone}
\end{subfigure}
\hfill
\begin{subfigure}{0.48\textwidth}
\centering
\includegraphics[width=3.cm]{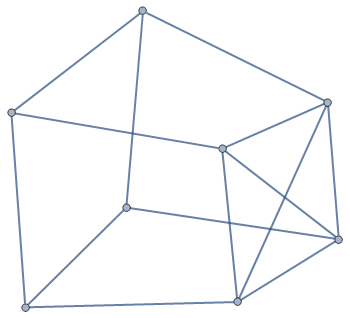}
\hspace{0.25cm}
\includegraphics[width=3.cm]{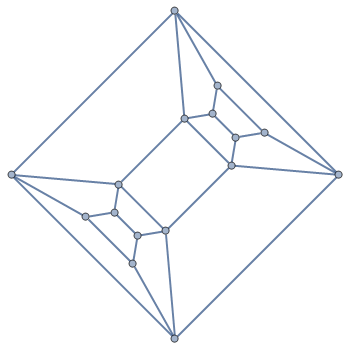}
\caption{Here $N=1$, $n=4N=4$, and $M=m=2$. Letting $J'=C_{4}\q P_{2}$, then $J$ is the (non-planar) graph on the left, and $J\w K_2\simeq C_{4}\q P_{4}$, depicted on the right (from \cite[Figure 4]{mafkpr}).}
\label{fig:qqtwo}
\end{subfigure}
\caption{Proposition \ref{prop:stacked}, case $n\equiv 2 \pmod 4$ (\ref{fig:qqone}), and case $n\equiv 0 \pmod 4$ and $m\equiv 0 \pmod 2$ (\ref{fig:qqtwo}).}
\label{fig:qq}
\end{figure}

Next, let $J'=C_{4N}\q P_M$,  where $N,M\geq 1$. Call
\[[v_1,v_2,\dots,v_{4N}]\]
one of the two bases of the stacked prism $J'$. Letting
\begin{equation}
\label{eq:opp}
J=J'+v_1v_{1+2N}+v_2v_{2+2N}+\dots+v_{2N}v_{4N},
\end{equation}
we obtain
\[C_{4N}\q P_{2M}\simeq J\w K_2,\]
as in Figure \ref{fig:qqtwo}. Note that $J$ is always non-planar, unless $N=M=1$, i.e., $J$ is the tetrahedron and $J\w K_2$ the cube.

On the other hand, if $C_n\q P_m$ is a Kronecker product with $K_2$, then it is bipartite, hence every copy of $C_n$ is bipartite so $n$ must be even. It remains to prove that $C_n\q P_m$ is not the Kronecker cover of any graph when $n\equiv 0\pmod 4$ and $m\equiv 1\pmod 2$. Assume by contradiction that for some $N,M\geq 1$ there exists a graph $J$ satisfying
\[
G:=C_{4N}\q P_{2M+1}\simeq J\w K_2.
%, \quad N,M\geq 1.
\]
Let $G'$ be the subgraph of $G$ generated by its vertices of degree $3$. Then
\[G'=A\dot\cup B,\]
where $A,B$ are $4N$-cycles, bases of the stacked prism $G$. We may write
\[A: v_1x,v_2y,v_3x,\dots,v_{4N}y,\]
where  as usual $V(K_2)=\{x,y\}$.
By definition of Kronecker product, the vertices
\[v_1,v_2,\dots,v_{4N}\] 
have degree $3$ in $J$. We claim that these are all distinct.

By contradiction, let $v_1=v_i$ for some $i\neq 1$. We cannot have two copies of $v_1x$ in the cycle $A$, thus $i$ is even. Further, by definition of Kronecker product, in $J$ the only two vertices of degree $3$ adjacent to $v_1$ are $v_2$ and $v_{4N}$. Thereby,
\[\{v_{i-1}, v_{i+1}\}=\{v_2, v_{4N}\}.\]
Say that $v_{i-1} = v_2$. As the only vertices of degree $3$ adjacent to $v_2$ in $J$ are $v_1$ and $v_3$, we also have $v_{i-2}=v_3$, and continuing in this fashion we obtain the following path in $G$,
\[v_1x,v_2y,v_3x,\dots,v_jx_1,v_jx_2,\dots,v_3y,v_2x,v_1y,\]
where $1\leq j\leq 4N$ and $\{x_1,x_2\}=\{x,y\}$. That means, $v_jx,v_jy$ are adjacent in $G$, contradicting the Kronecker product definition. Otherwise, if
$v_{i-1}=v_{4N}$, then $A$ becomes
\[v_1x,v_2y,\dots,v_{4N}x,v_1y,v_2x,\dots,v_{4N}y,\]
which cannot happen as $4N$ is even.
%$v_{i+1}=v_2$, $A$ becomes

%\[v_1x,v_2y,v_3x,\dots,v_{i-1}x,v_1y,v_2x,\dots,v_{4n}y\]
%that if $i\leq 2n$ leads to

%\[v_1x,\dots,v_{i-1}x,v_1y,v_2x,\dots,v_1x,\dots,v_{4n}y\]
%which cannot be a cycle, while if $i>2n$, then $A$ has the form

%\[v_1x,\dots,v_{i-1}x,v_1y,v_2x,\dots,v_jx,v_{4n}y\]
%for some $j<2n<i$. Therefore $A$ must be

%\[v_1x,\dots,v_jy,v_{4n}x,\dots,v_1y,v_2x,\dots,v_jx,v_{4n}y\]
%and the only possibility is that $v_{4n}x,v_1y$ are adjacent which means $j=i-2$, in contradiction with $j<2n<i$ since $i$ and $j$ are both even.

Therefore, \[v_1,v_2,\dots,v_{4N}\] are all distinct and they are the vertices of degree $3$ in $J$, as claimed. Moreover
\[B: v_1y,v_2x,v_3y,\dots,v_{4N}x.\]
Now in $G$, there is a path starting from $v_1x$ and ending at $v_ix$ in $B$ for some $1\leq i\leq 4N$, traversing exactly $2M-1$ vertices of degree $4$, as in Figure \ref{fig:imp}. Since $v_ix$ belongs to $B$, $i$ is even. On one hand, by definition of Kronecker product, there is also a path from $v_iy$ to $v_1y$, via $2M-1$ vertices of degree $4$. On the other hand it also follows that, for every $1\leq j\leq 4N$, there is a path from $v_jx$ to $v_{(i\pm(j-1)\mod{4N})}x$ (and the same holds with $y$ in place of $x$), along $2M-1$ vertices of degree $4$, where the $\pm$ sign has the same value for each $j$ (you can see that the plus occurs when the indexes increase in the same direction on the cycles $A$ and $B$, the minus when they increase in opposite directions).

\begin{figure}[ht]
\centering
\begin{subfigure}{0.45\textwidth}
\centering
\includegraphics[width=5.25cm]{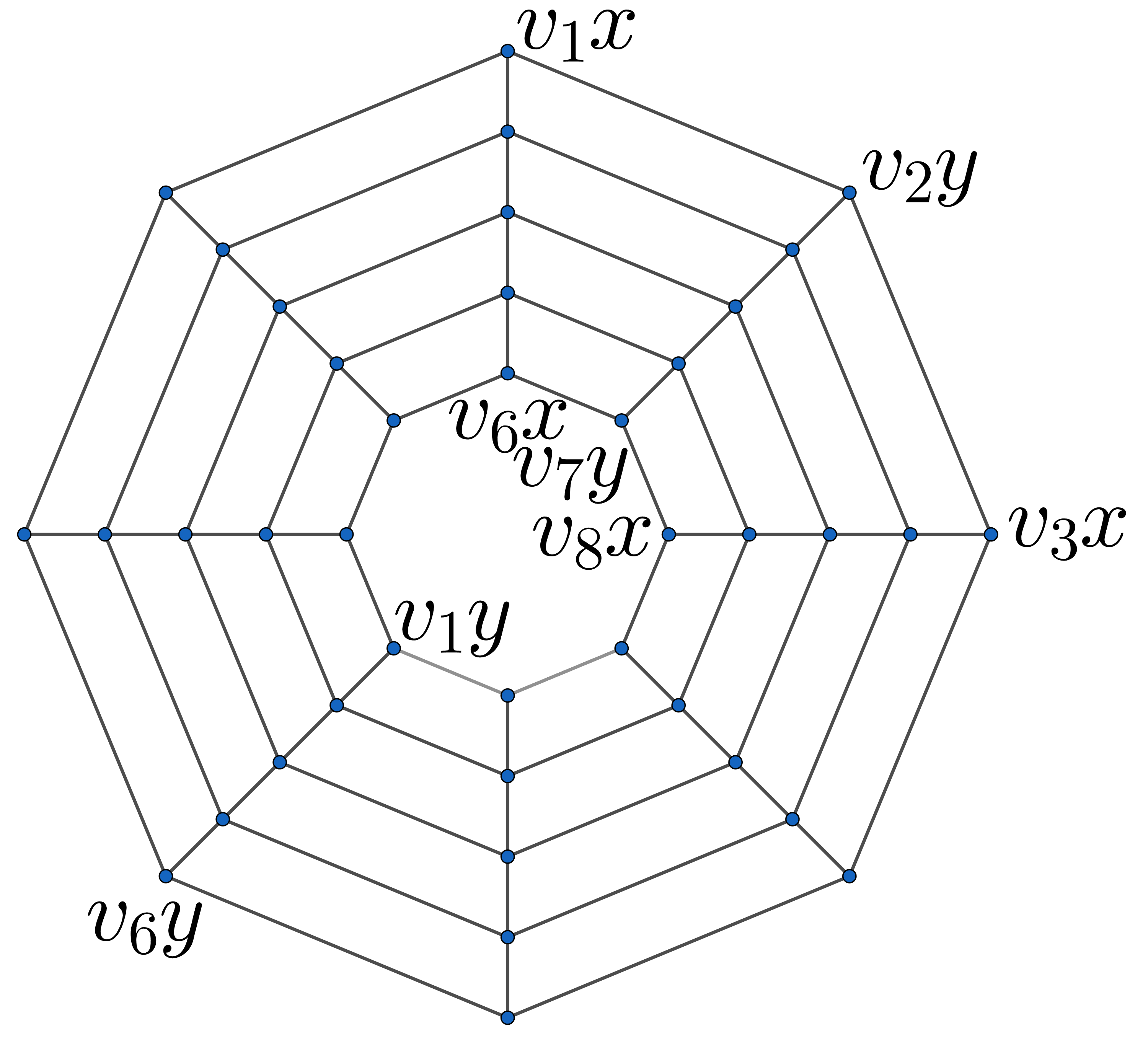}
\caption{Here the value of the $\pm$ is a plus, leading to impossible scenarios such as $v_1y$ in place of $v_3y$.}
\label{fig:imp1}
\end{subfigure}
\hspace{0.125cm}
\begin{subfigure}{0.45\textwidth}
\centering
\includegraphics[width=5.25cm]{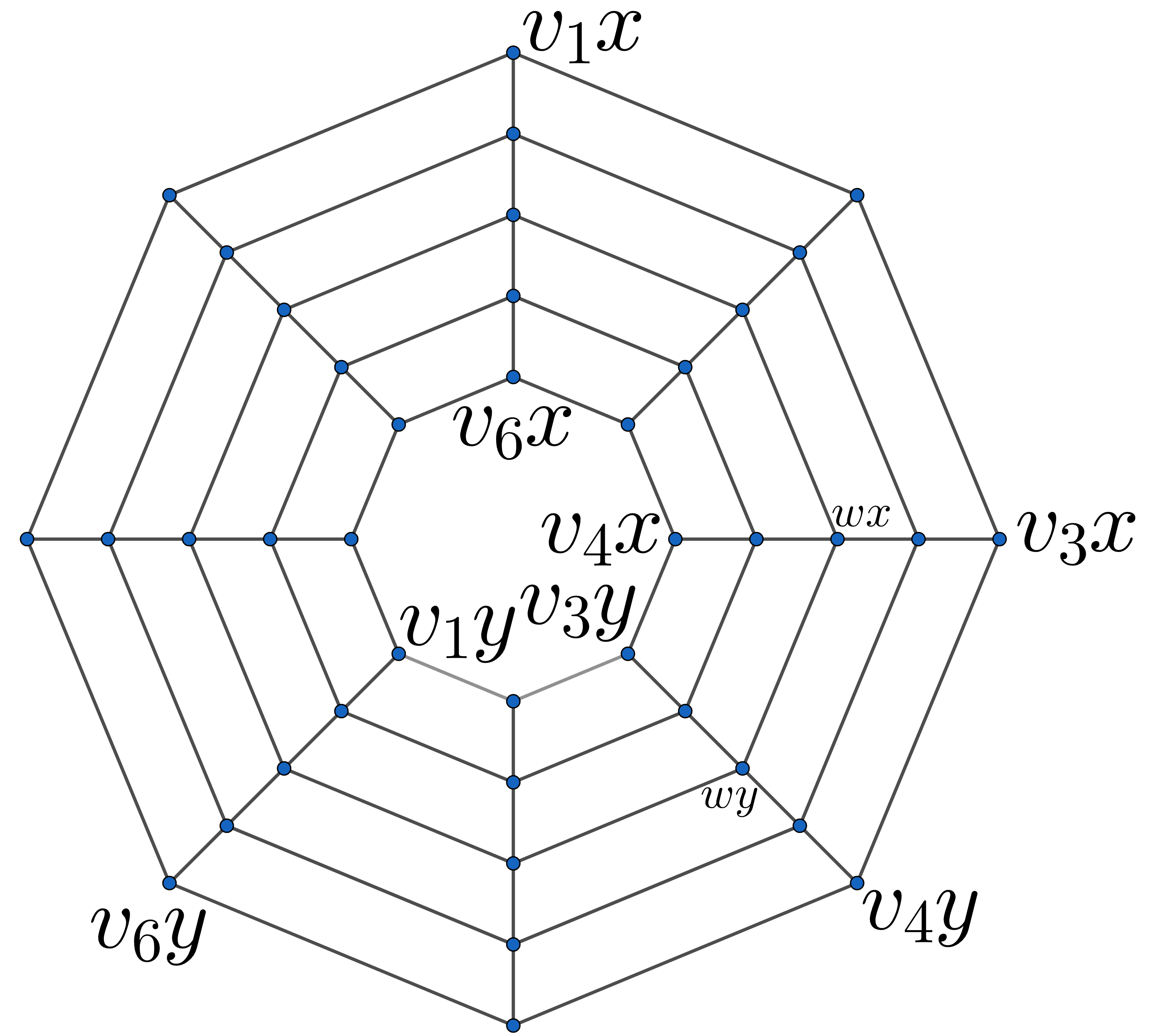}
\caption{Here the value of the $\pm$ is a minus, leading to the contradiction that there exists $w\in V(J)$ such that $wx,wy$ are adjacent in $G$.}
\label{fig:imp2}
\end{subfigure}
\caption{Proposition \ref{prop:stacked}, case $n\equiv 0 \pmod 4$ and $m\equiv 1 \pmod 2$. In this illustration $N=2$, $n=8$, $M=2$, $m=5$, and $i=6$.}
\label{fig:imp}
\end{figure}

If the value of the $\pm$ is a plus (as in Figure \ref{fig:imp1}), we combine the above observations with $j=i$ to deduce that there is a path from $v_ix$ to $v_{(2i-1\mod{4N})}x$ along $2M-1$ vertices of degree $4$. Thereby,
\[2i-1\equiv 1 \pmod{4N},\]
i.e.~$i\in\{1,2N+1\}$, contradicting the fact that $i$ is even. 

Hence the value of the $\pm$ is a minus, as in Figure \ref{fig:imp2}. In particular, there is a path $P_x$ in $G$ between $v_{i/2}x$ and $v_{i/2+1}x$ along $2M-1$ vertices of degree $4$, and likewise there is a path $P_y$ in $G$ between $v_{i/2}y$ and $v_{i/2+1}y$ along $2M-1$ vertices of degree $4$.
Keep in mind that, by construction, $v_{i/2}x,v_{i/2+1}y$ and $v_{i/2}y,v_{i/2+1}x$ are adjacent.
Let $wx$ be the central vertex of $P_x$, well-defined since $P_x$ has $2M+1$ total vertices. Then $wy$ is the central vertex of $P_y$, hence $wx,wy$ are adjacent in $G$, in contradiction with the Kronecker product definition.
\end{proof}

\subsection[Simultaneous Cartesian and Kronecker product II]{$H\q K_2\simeq J\w K_2$}
\label{sec:ck2}
\begin{prop}
\label{le:JH}
Let $H,J$ be graphs such that $H\q K_2\simeq J\w K_2$, and the product is a $3$-polytope. Then $H$ is obtained from a $2\ell$-gon $[u_1,u_2,\dots,u_{2\ell}]$ by adding diagonals $u_iu_j$ in such a way that the resulting graph is still bipartite and planar, with
\[u_iu_{(i+\ell\mod 2\ell)}\not\in E(H), \quad 1\leq i\leq 2\ell,\]
and
\[u_iu_j\in E(H)\iff u_{(i+\ell\mod 2\ell)}u_{(j+\ell\mod 2\ell)}\in E(H).\]
Further, we have
\begin{equation}
\label{eq:JH}
J\simeq
\begin{cases}
H+u_1u_{1+\ell}+u_2u_{2+\ell}+\dots+u_{\ell}u_{2\ell}&\ell \text{ even,}\\
\widehat{H}\q K_2&\ell \text{ odd,}
\end{cases}
\end{equation}
where $\widehat{H}$ is an outerplanar, Hamiltonian graph of order $\ell$.
\end{prop}

\begin{proof}
Since $H\q K_2$ is a $3$-polytope, then $H$ is outerplanar and Hamiltonian \cite[Proposition 1.9]{mafkpr}. Since $H\q K_2$ is a Kronecker product, it must be bipartite, hence $H$ is bipartite. Thereby, $H$ may be constructed by adding diagonals to a $2\ell$-cycle
\[C_{2\ell}: u_1,u_2,\dots,u_{2\ell}\]
(if the cycle were of odd length, then $H$ would not be bipartite as it would contain a face of odd length, no matter how many diagonals one were to add).
%where $\ell=|V(H)|$.% If there are no diagonals, then $H\q K_2$ is simply the $\ell$-gonal prism, hence by Proposition \ref{prop:stacked}, in this case $H\q K_2$ is a Kronecker product $J\w K_2$ if and only if $\ell$ is even. Moreover, if $\ell\equiv 0 \pmod 4$, then $J$ is as in \eqref{eq:JH}; if $\ell\equiv 2 \pmod 4$, then $J$ is the $\ell/2$-gonal prism. The present proposition is proven in this case.

\begin{comment}
Now let $u_iu_j$ be a diagonal of $C_\ell$, such that
\[\deg_H(u_{i+1})=\deg_H(u_{i+2})=\dots=\deg_H(u_{j-1})=2.\]
Note that $j\geq i+2$. In fact, since $H\q K_2$ is a Kronecker product, it must be bipartite, thus $\ell$ is even, $j\geq i+3$, and $i,j$ are of different parity.
\end{comment}

%Henceforth we assume that $H$ contains at least one diagonal.

We write
\[C_{2\ell}\q K_2\simeq J_{-}\w K_2,\]
with $J_{-}$ to be determined. There are two cases.

First case, $\ell$ is even. By \eqref{eq:opp}, we obtain
\[J_{-}\simeq [v_1,v_2,\dots,v_{2\ell}]+v_1v_{1+\ell}+v_2v_{2+\ell}+\dots+v_{\ell}v_{2\ell}.\]
%On one hand, since $H\q K_2$ is the Kronecker cover of $J$, we can assign a vertex labelling to $H\q K_2$ such that each vertex is of the form either $ax$ or $ay$, for some $a\in V(J)$. On the other hand, $H\q K_2$ may be obtained from a $2n$-gonal prism by adding certain corresponding diagonals to the bases. One of these bases will be labelled
%\[v_1x,v_2y,v_3x,\dots,v_{2n}y,\]
%where $v_1,v_2,\dots,v_{2n}$ are (not necessarily distinct) vertices of $J$.
%There are two cases. First case, $v_1,v_2,\dots,v_\ell$ are all distinct. Then $H\q K_2$ contains the disjoint cycles
%\[A:v_1x,v_2y,v_3x,\dots,v_{2n}y\]
%and
%\[B:v_1y,v_2x,v_3y,\dots,v_{2n}x.\]
%Thus $v_1,v_2,\dots,v_{2n}$ is a cycle of $J$. 
%there exists $i\in\{1,2,\dots,\ell\}$ such that for every $1\leq j\leq \ell$, the vertices $v_jx$ and $v_{(j+i\mod{\ell})}y$ are adjacent, and we deduce that $i=\ell/2$.
%\\
%By construction, for every $1\leq i\leq 2n$, there exists $\alpha(i)$ such that the
%vertices $v_ix$ and $v_{\alpha(i)}y$ are adjacent in the graph product. By the proof of Proposition \ref{prop:stacked}, for every $1\leq i\leq 2n$, we have 
%\[\alpha(i)=i+n \pmod {2n}\]
%(with the usual convention $u_0=u_{2n}$). In particular, $i,i+n$ have the same parity, thus $n$ is even. In addition, since $v_ix$ and $v_{(i+n\mod{2n})}y$ are adjacent, then
%\[v_iv_{(i+n\mod{2n})}\in E(J).\]
Consequently, the prism $J_{-}\w K_2$ consists of the two bases
\[v_1x,v_2y,\dots,v_{2\ell}y
\quad\text{ and }\quad
v_1y,v_2x,v_3y,\dots,v_{2\ell}x,\]
plus the edges
\[v_ixv_{(i+\ell\mod 2\ell)}y, \quad 1\leq i\leq 2\ell,\]
as in Figure \ref{fig:ppone}.

If $u_iu_j$ is a diagonal of the cycle $C_{2\ell}$ in $H$, then first of all $i,j$ are of different parity since $H$ is bipartite. Moreover,
\[v_ix,v_jy\]
are adjacent in $J\w K_2$. Since $v_ix$ and $v_{(i+\ell\mod{2\ell})}y$ are adjacent, and also $v_jy$ and $v_{(j+\ell\mod{2\ell})}x$ are adjacent, and since $J\w K_2\simeq H\q K_2$, we deduce that
\[v_{(i+\ell\mod{2\ell})}y,v_{(j+\ell\mod{2\ell})}x\]
are adjacent as well, hence $u_{(i+\ell\mod{2\ell})}u_{(j+\ell\mod{2\ell})}\in E(H)$. In addition, $v_iv_j\in E(J)$. The claimed conditions on $H,J$ hold in the first case.

\begin{figure}[ht]
\centering
\begin{subfigure}{0.3\textwidth}
\centering
\includegraphics[width=3.25cm]{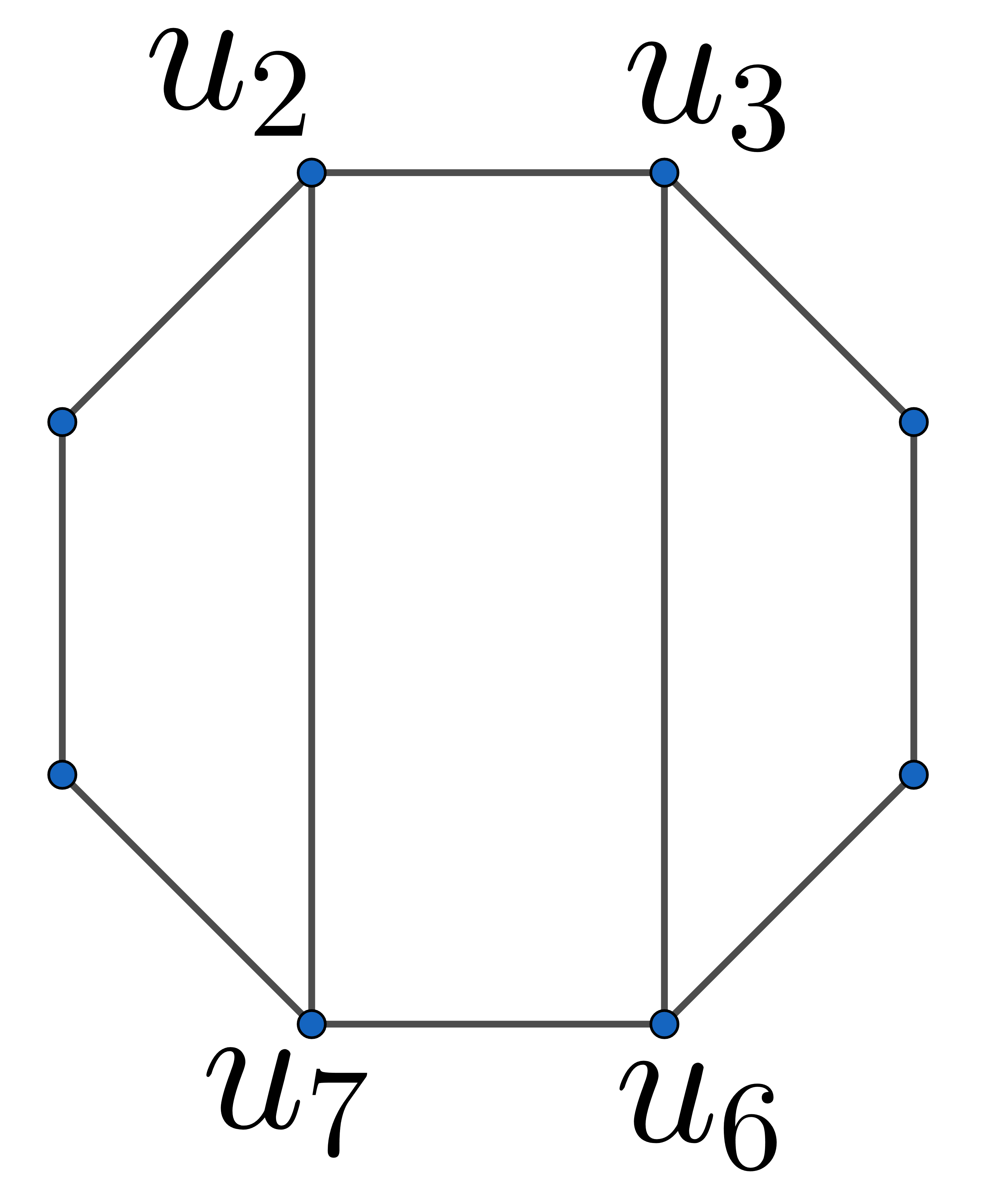}
\caption{$H$.}
\label{fig:HH1}
\end{subfigure}
%\hspace{0.5cm}
\begin{subfigure}{0.3\textwidth}
\centering
\includegraphics[width=4.5cm]{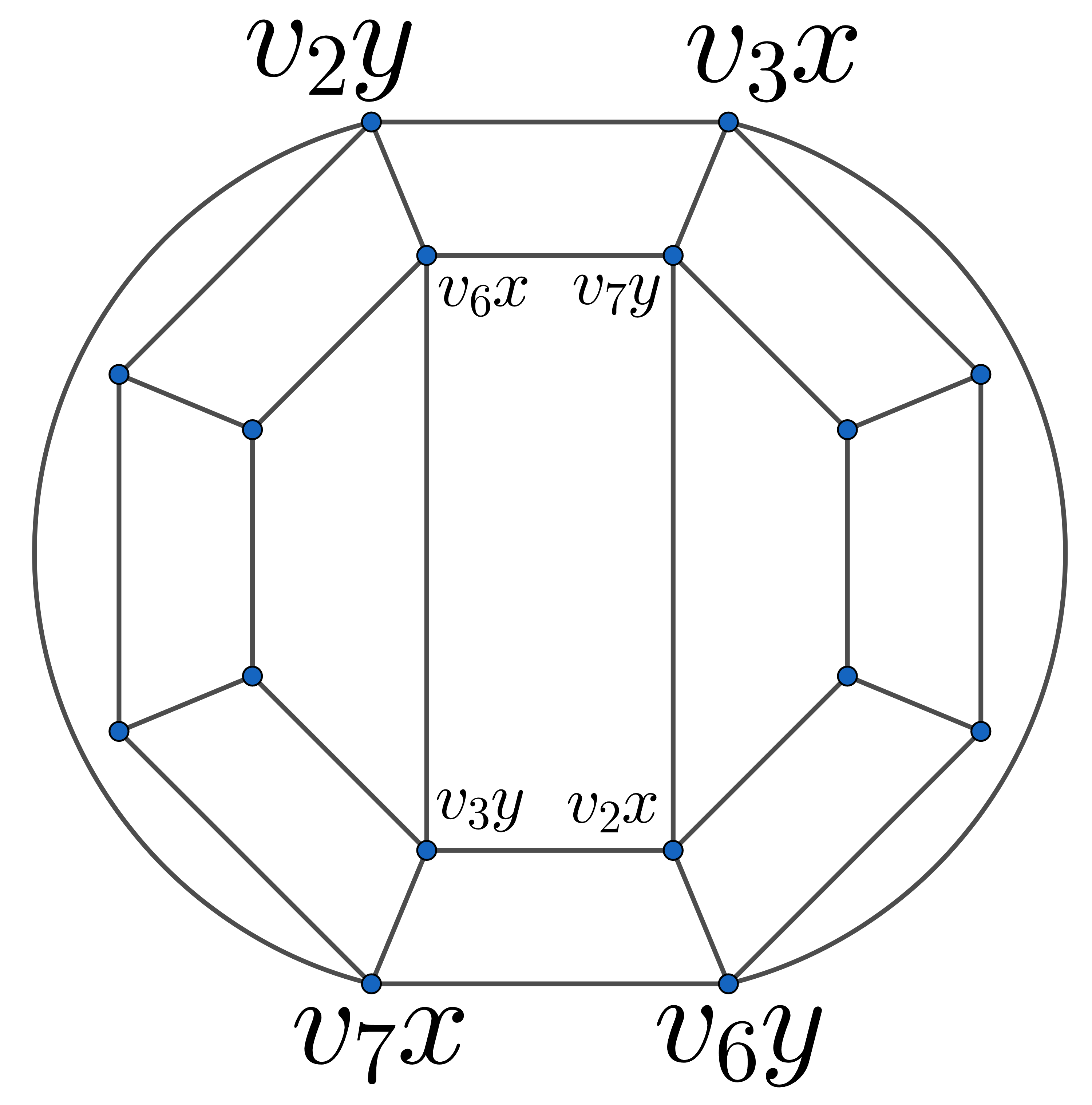}
\caption{$H\q K_2\simeq J\w K_2$.}
\label{fig:pp1}
\end{subfigure}
%\hspace{0.5cm}
\begin{subfigure}{0.3\textwidth}
\centering
\includegraphics[width=3.25cm]{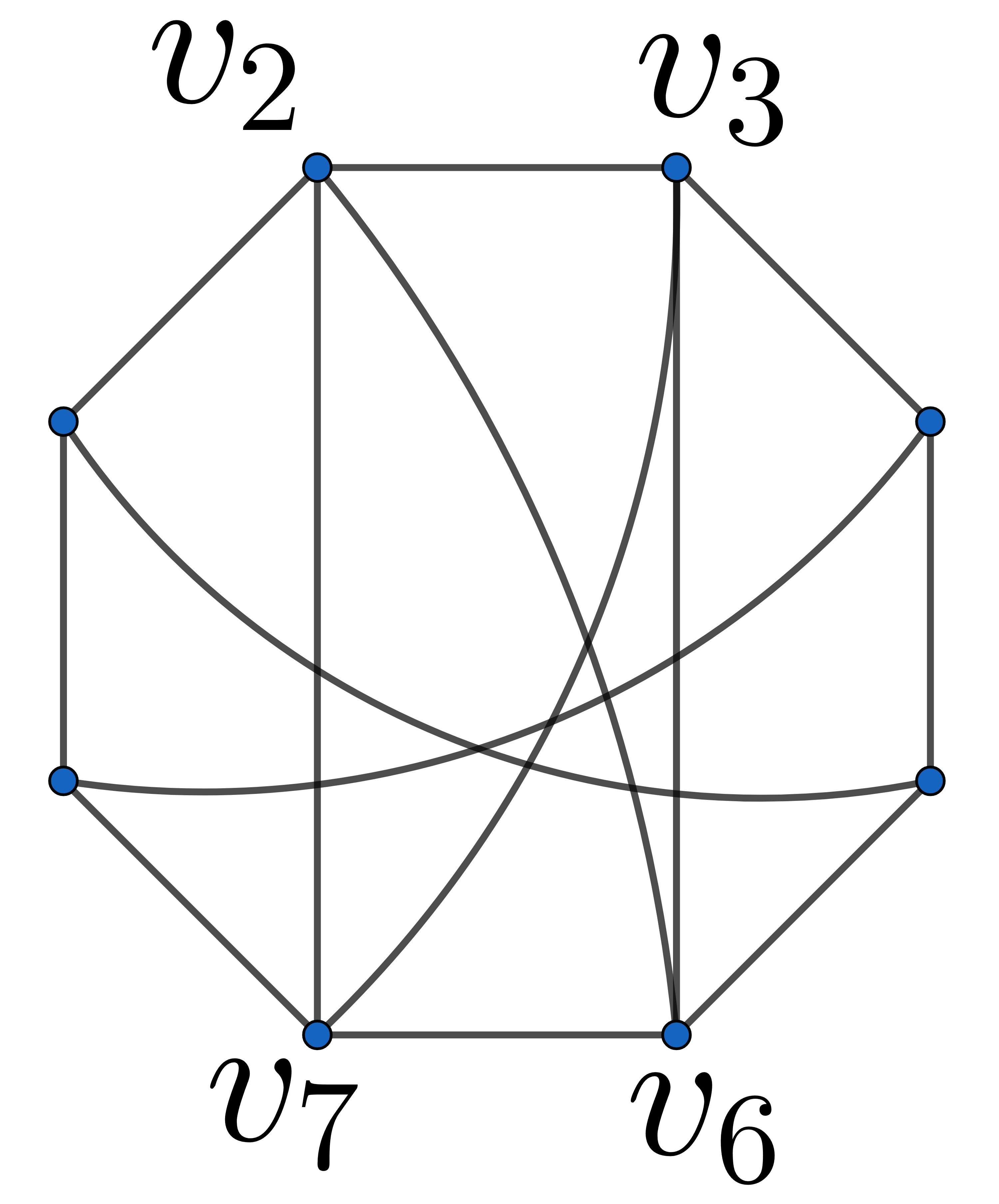}
\caption{$J$.}
\label{fig:JJ1}
\end{subfigure}
\caption{Illustration of Proposition \ref{le:JH} for $\ell=4$. It is worth noting that $H=F_8$ (ladder graph), so that the product $H\q K_2\simeq J\w K_2$ is also isomorphic to $C_4\q P_4$ -- recall Corollary \ref{cor:tri}.
% In the pictures, the graphs $H$ (left), $H\q K_2\simeq J\w K_2$ (centre), and $J$ (right).
}
\label{fig:ppone}
\end{figure}

%By the proof of Proposition \ref{prop:stacked}, then $\ell\equiv 2\pmod 4$ and
%\[v_i=v_{(i+\ell\mod 2\ell)}, \qquad 1\leq i\leq\ell.\]
%Here $H\q K_2$ contains the disjoint cycles
%\[A':v_1x,v_2y,\dots,v_{\ell/2}x,v_1y,v_2x,\dots,v_{\ell/2}y\]
%
%\[B':w_1x,w_2y,\dots,w_{\ell/2}x,w_1y,w_2x,\dots,w_{\ell/2}y,\]
%where $v_1,\dots,v_{\ell/2}$ and $w_1,\dots,w_{\ell/2}$ are disjoint cycles of $J$.
%\[B:v_{\ell/2+1}y,v_{\ell/2+2}y,\dots,v_{\ell}x,v_{\ell/2+1}x,v_{\ell/2+2}y,\dots,v_{\ell}x.\]
%Further, w.l.o.g. we can assume that for every $1\leq i\leq \ell$, the vertices $v_ix$ and $w_{i}y$ %$v_{(j+\ell/2\mod{\ell})}y$
%are adjacent in $J\w K_2$. Hence for every $1\leq i\leq \ell$, we have $v_iw_{i}\in E(J)$. The edges of $J$ that we have constructed so far form an $n$-gonal prism.
%\\

Second case, $\ell$ is odd. Then $J_{-}$ is the $\ell$-gonal prism (refer again to Proposition \ref{prop:stacked}, in particular Figure \ref{fig:qqone}). We introduce a vertex labelling such that the bases are
\[v_1,v_2,\dots,v_\ell \quad\text{ and }\quad w_1,w_2,\dots,w_\ell,\]
and such that $v_iw_i\in E(J_{-})$ for every $1\leq i\leq\ell$, as in Figure \ref{fig:pptwo}.

If we have the diagonal $u_iu_j$ of the cycle $C_{2\ell}$ in $H$, then $i,j$ are of different parity since $H$ is bipartite, and moreover,
\[v_{(i\mod \ell)}x,v_{(j\mod \ell)}y\]
are adjacent in $J\w K_2$, and likewise $v_{(j\mod \ell)}x,v_{(i\mod \ell)}y$ are adjacent. Since $J\w K_2\simeq H\q K_2$, we also have
\[u_{(i+\ell\mod 2\ell)}u_{(j+\ell\mod 2\ell)}\in E(H).\]
In $J$, we get $v_{(i\mod\ell)}v_{(j\mod\ell)}\in E(J)$. In addition, by construction $(w_{(i\mod\ell)}x,w_{(j\mod\ell)}y)$ and $(w_{(j\mod\ell)}x,w_{(i\mod\ell)}y)$ are edges of the product, thus $w_{(i\mod\ell)}w_{(j\mod\ell)}\in E(J)$.

We cannot have $u_iu_{(i+\ell\mod 2\ell)}\in E(H)$ for some $1\leq i\leq 2\ell$, otherwise $v_ix$ would be adjacent to $v_iy$ in the graph product.

We claim that the diagonals added to the bases of $J_{-}$ cannot cross. By contradiction, suppose that
\[v_{a}v_{b}\quad\text{ and }\quad v_{c}v_{d}\]
cross. Then in $C_{2\ell}$, the vertices
\[u_a,u_c,u_b,u_d,u_{a+\ell},u_{c+\ell},u_{b+\ell},u_{d+\ell}\]
appear in this order around the cycle. In $H$, we have either the diagonals $u_au_b$ and $u_{a+\ell}u_{b+\ell}$, or $u_au_{b+\ell}$ and $u_{a+\ell}u_{b}$, and also either the diagonals $u_cu_d$ and $u_{c+\ell}u_{d+\ell}$, or $u_cu_{d+\ell}$ and $u_{c+\ell}u_{d}$. In any case some diagonals of $H$ cross, contradiction. Hence $J$ is indeed of the form $\widehat{H}\q K_2$, with $\widehat{H}$ outerplanar and Hamiltonian. The claimed conditions on $H,J$ hold in the second case also.

\begin{figure}[ht]
\centering
\begin{subfigure}{0.3\textwidth}
\centering
\includegraphics[width=3.25cm]{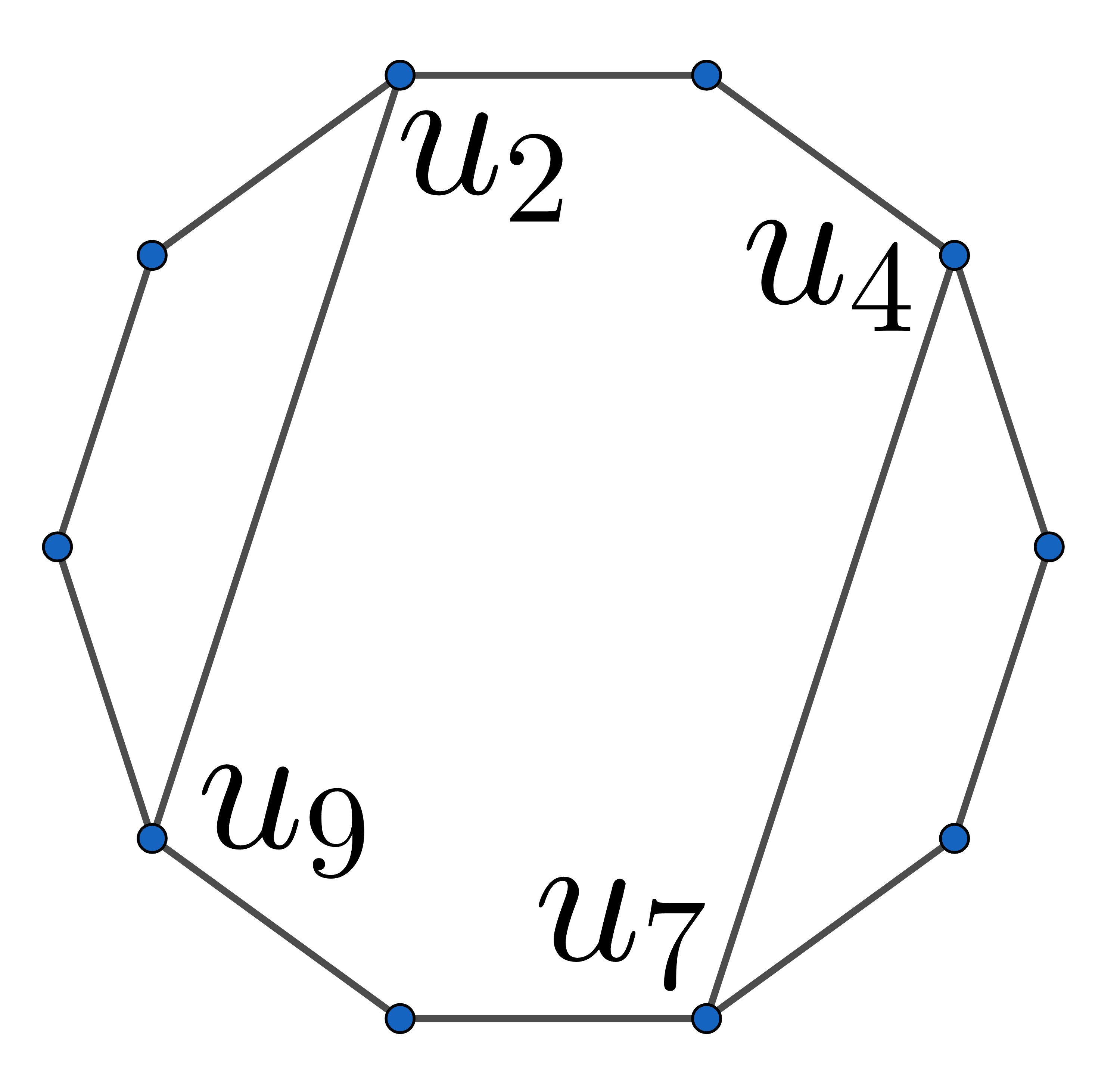}
\caption{$H$.}
\label{fig:HH2}
\end{subfigure}
%\hspace{0.5cm}
\begin{subfigure}{0.3\textwidth}
\centering
\includegraphics[width=4.25cm]{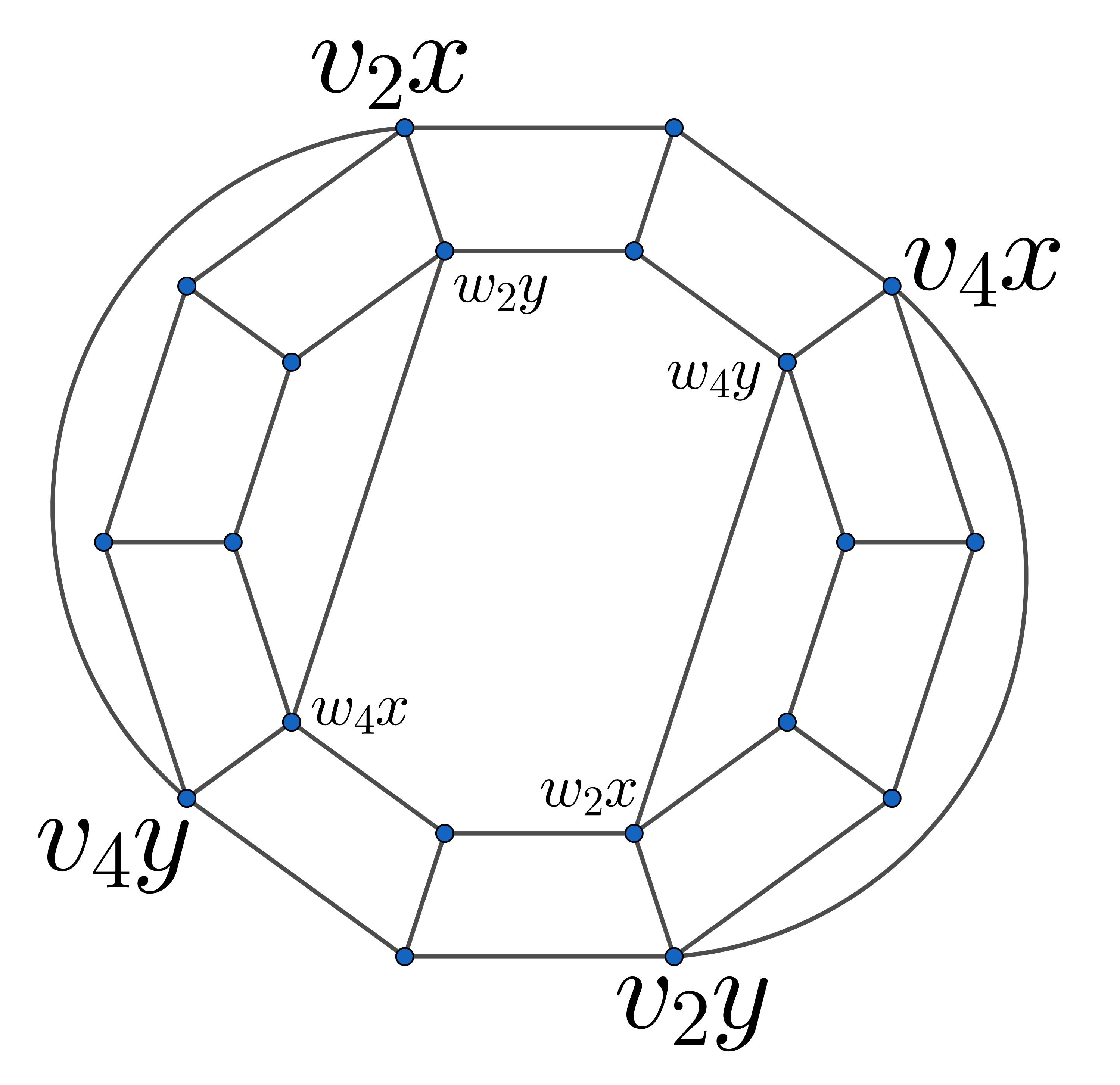}
\caption{$H\q K_2\simeq J\w K_2$.}
\label{fig:pp2}
\end{subfigure}
%\hspace{0.5cm}
\begin{subfigure}{0.3\textwidth}
\centering
\includegraphics[width=3.25cm]{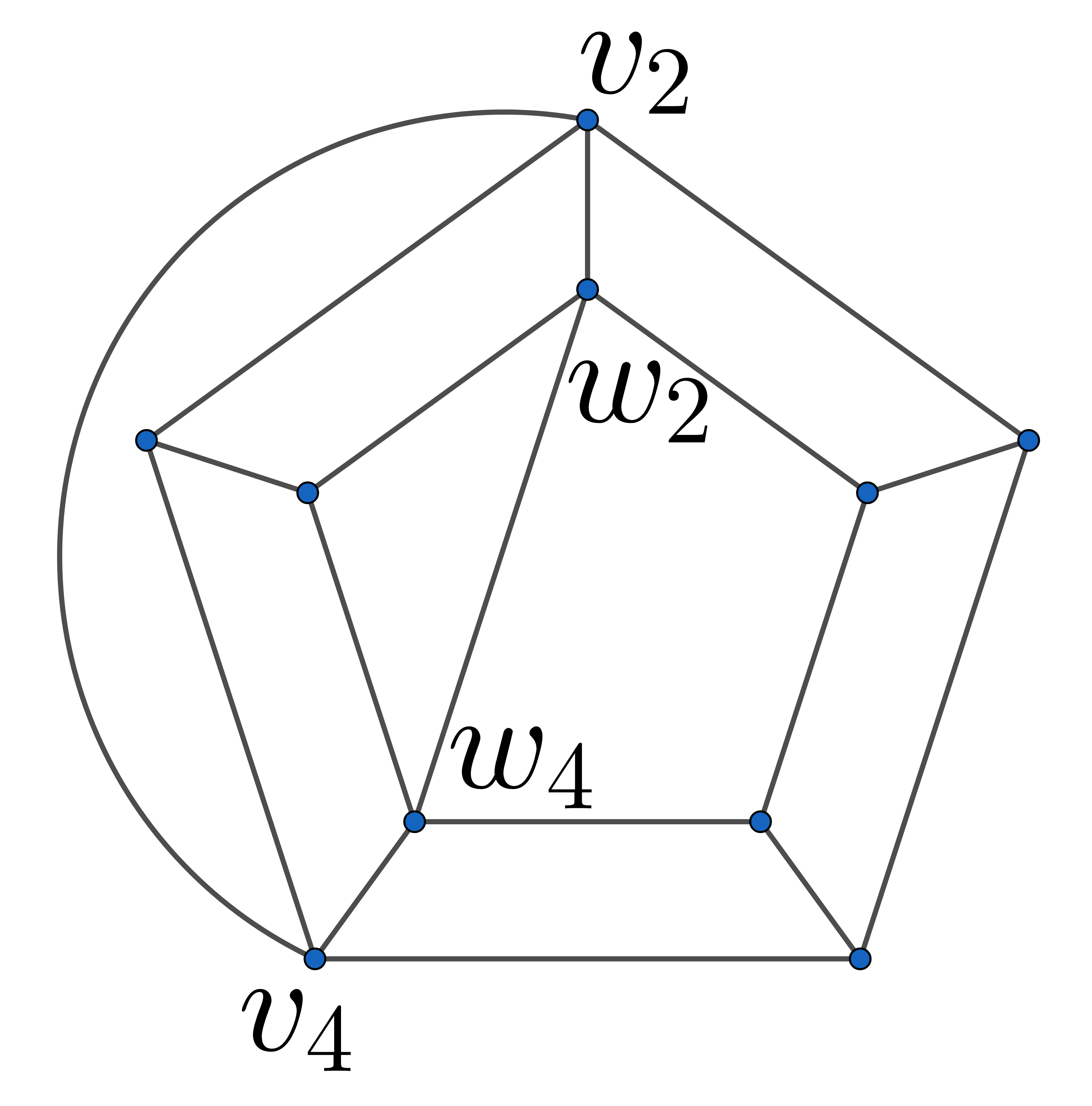}
\caption{$J$.}
\label{fig:JJ2}
\end{subfigure}
\caption{Illustration of Proposition \ref{le:JH} for $\ell=5$.% In the pictures, the graphs $H$ (left), $H\q K_2\simeq J\w K_2$ (centre), and $J$ (right).
}
\label{fig:pptwo}
\end{figure}

Vice versa, it is immediate to check that all graphs $H,J$ satisfying the assumptions of the present proposition indeed verify $H\q K_2\simeq J\w K_2$ with the product being a $3$-polytope.
\end{proof}

Note that in Proposition \ref{le:JH}, if $\ell$ is even then $J$ is non-planar, unless $J$ is the tetrahedron and $J\w K_2$ the cube.

On the other hand when $\ell$ is odd, we note that both $J$ and $J\w K_2$ are polyhedra. Such $J$'s constitute a subclass of \ref{eq:c1} from Theorem \ref{thm:0123}, since the other cases \ref{eq:c0}, \ref{eq:c2}, and \ref{eq:c3} are impossible. Indeed, $J$ always contains at least two disjoint odd faces, regardless of the number of diagonals added to the bases of $J_{-}$ to construct $J$.  %Indeed, such $J$ from Proposition \ref{le:JH} in the case of $\ell$ odd are obtained by adding certain edges to the bases of an $\ell$-gonal prism. Each added edge does not alter the number of odd faces, or how they intersect, hence such $J$ are polyhedra with exactly two odd faces, and these odd faces are disjoint.

\begin{comment}
\begin{ex}
If $H=C_{4\ell}$ then the product is the $4\ell$-gonal prism. For $\ell=2$, the resulting graph $J$ is depicted in Figure \ref{fig:Jex1}.
\\
If
\[H=F_{4\ell}, \quad \ell\geq 2\]
(ladder graph), then
\[H\q K_2\simeq J\w K_2\simeq C_{4}\q P_{2\ell},\]
with $J$ as in \eqref{eq:JH}. The case $\ell=2$ is depicted in
Figure \ref{fig:Jex2}. Note that $J$ of
Figure \ref{fig:Jex2} may also be obtained from the cube by adding the diagonals of a face (cf. Corollary \ref{cor:tri}).
\begin{figure}[ht]
\centering
\begin{subfigure}{0.49\textwidth}
\centering
\includegraphics[width=3.25cm]{Jex1.png}
\caption{The graph $J$ when $H=C_8$.}
\label{fig:Jex1}
\end{subfigure}
    \hfill
\begin{subfigure}{0.49\textwidth}
\centering
\includegraphics[width=3.5cm]{Jex2.png}
\caption{The graph $J$ when $H=F_8$.}
\label{fig:Jex2}
\end{subfigure}
\caption{%Examples of $3$-regular polyhedra $J$ such that $\calP=J\wedge K_2$ is a polyhedron.
}
\label{fig:Jex}
\end{figure}
\end{ex}
\end{comment}

%As a consequence of Propositions \ref{le:CC}, \ref{prop:stacked}, and \ref{le:JH}, given the $3$-polytope $C_{n}\q P_{m}$, there exist graphs $H,J$ such that
%\begin{equation}
%\label{eq:nmhj}
%C_{n}\q P_{m}\simeq H\q K_2\simeq J\w K_2
%\end{equation}
%if and only if either $n=2N$ is even and $m=2$ ($2N$-gonal prisms), or $n=4$ and $m$ is even (stacked cubes with $m$ even).

% Recalling the final remark in section \ref{sec:intro}, the only graphs $C_{n}\q P_{m}$ satisfying \eqref{eq:nmhj} are the simultaneous Kronecker and Cartesian products that are either vertex-regular (

\paragraph{Acknowledgements.}
Riccardo W. Maffucci was partially supported by Programme for Young Researchers `Rita Levi Montalcini' PGR21DPCWZ \textit{Discrete and Probabilistic Methods in Mathematics with Applications}, awarded to Riccardo W. Maffucci.

%\clearpage
\bibliographystyle{abbrv}
\bibliography{bibgra}

\begin{thebibliography}{10}

\bibitem{aaghlt}
G.~Abay-Asmerom, R.~H. Hammack, C.~E. Larson, and D.~T. Taylor.
\newblock Direct product factorization of bipartite graphs with
  bipartition-reversing involutions.
\newblock {\em SIAM Journal on Discrete Mathematics}, 23(4):2042--2052, 2010.

\bibitem{abay09}
G.~Abay-Asmerom, R.~H. Hammack, and D.~T. Taylor.
\newblock Factorial properties of graphs.
\newblock {\em Australas. J Comb.}, 44:265--272, 2009.

\bibitem{behmah}
M.~Behzad and S.~Mahmoodian.
\newblock On topological invariants of the product of graphs.
\newblock {\em Canadian Mathematical Bulletin}, 12(2):157--166, 1969.

\bibitem{bieell}
D.~P. Biebighauser and M.~N. Ellingham.
\newblock Prism-hamiltonicity of triangulations.
\newblock {\em Journal of Graph Theory}, 57(3):181--197, 2008.

\bibitem{farwal}
M.~Farzan and D.~A. Waller.
\newblock Kronecker products and local joins of graphs.
\newblock {\em Canadian Journal of Mathematics}, 29(2):255--269, 1977.

\bibitem{fllp07}
A.~Fern{\'a}ndez, T.~Leighton, and J.~L. L{\'o}pez-Presa.
\newblock Containment properties of product and power graphs.
\newblock {\em Discrete applied mathematics}, 155(3):300--311, 2007.

\bibitem{gasmaf}
S.~Gaspoz and R.~W. Maffucci.
\newblock Independence numbers of polyhedral graphs.
\newblock {\em Applied Mathematics and Computation}, 462:128349, 2024.

\bibitem{hamm09}
R.~H. Hammack.
\newblock On direct product cancellation of graphs.
\newblock {\em Discrete Mathematics}, 309(8):2538--2543, 2009.

\bibitem{haimkl}
R.~H. Hammack, W.~Imrich, and S.~Klav{\v{z}}ar.
\newblock {\em Handbook of product graphs}, volume~2.
\newblock CRC press Boca Raton, 2011.

\bibitem{hajetk}
J.~Harant, S.~Jendrol', and M.~Tk{\'a}c.
\newblock On 3-connected plane graphs without triangular faces.
\newblock {\em Journal of Combinatorial Theory, Series B}, 77(1):150--161,
  1999.

\bibitem{hahimo}
J.~M. Harris, J.~L. Hirst, and M.~J. Mossinghoff.
\newblock {\em Combinatorics and Graph Theory}.
\newblock Springer, 2008.

\bibitem{impi08}
W.~Imrich and T.~Pisanski.
\newblock Multiple {K}ronecker covering graphs.
\newblock {\em European Journal of Combinatorics}, 29(5):1116--1122, 2008.

\bibitem{krpi19}
M.~Krnc and T.~Pisanski.
\newblock Generalized {P}etersen graphs and {K}ronecker covers.
\newblock {\em Discrete Mathematics \& Theoretical Computer Science}, 21(Graph
  Theory), 2019.

\bibitem{lova71}
L.~Lov{\'a}sz.
\newblock On the cancellation law among finite relational structures.
\newblock {\em Periodica Mathematica Hungarica}, 1(2):145--156, 1971.

\bibitem{mafkpr}
R.~W. Maffucci.
\newblock Classification and construction of planar, 3-connected {K}ronecker
  products.
\newblock {\em arXiv:2402.01407}.

\bibitem{mafpo2}
R.~W. Maffucci.
\newblock Constructing certain families of $3$-polytopal graphs.
\newblock {\em Journal of Graph Theory}, pages 1--18, 2022.

\bibitem{mafpo3}
R.~W. Maffucci.
\newblock Self-dual polyhedra of given degree sequence.
\newblock {\em Art Discrete Appl. Math. 6 (2023), P1.04.}, 2023.

\bibitem{mafwil}
R.~W. Maffucci and N.~Willems.
\newblock On smallest 3-polytopes of given graph radius.
\newblock {\em Discrete Mathematics}, 346(5):113322, 2023.

\bibitem{radste}
E.~Steinitz and H.~Rademacher.
\newblock {\em Vorlesungen {\"u}ber die {Theorie} der {Polyeder} unter
  {Einschlu{{\ss}}}\ der {Elemente} der {Topologie}. {Reprint}}, volume~41 of
  {\em Grundlehren Math. Wiss.}
\newblock Springer, Cham, 1976.

\bibitem{wall76}
D.~A. Waller.
\newblock Double covers of graphs.
\newblock {\em Bulletin of the Australian Mathematical Society},
  14(2):233--248, 1976.

\bibitem{whitco}
H.~Whitney.
\newblock Congruent graphs and the connectivity of graphs.
\newblock {\em Hassler Whitney Collected Papers}, pages 61--79, 1992.

\bibitem{zhan09}
Z.~Zhang.
\newblock Semi-hyper-connected vertex transitive graphs.
\newblock {\em Discrete mathematics}, 309(4):899--907, 2009.

\bibitem{zhme06}
Z.~Zhang and J.~Meng.
\newblock Semi-hyper-connected edge transitive graphs.
\newblock {\em Discrete mathematics}, 306(7):705--710, 2006.

\end{thebibliography}
\end{document}